\documentclass[a4paper,reqno,10pt]{amsart}
\usepackage{amsfonts}
\usepackage{amsmath}
\usepackage{amssymb}
\usepackage{tikz-cd}
\usepackage{cite}
\usepackage{bbm}
\usepackage{amscd}
\usepackage[left=2cm,right=2cm,bottom=3cm,top=3cm]{geometry}
\usepackage{etoolbox}
\usepackage{filecontents}
\usepackage[hidelinks]{hyperref}

\newtheorem{theorem}{Theorem}[section]
\newtheorem{proposition}[theorem]{Proposition}
\theoremstyle{definition}
\newtheorem{lemma}[theorem]{Lemma}
\newtheorem{definition}[theorem]{Definition}
\newtheorem{corollary}[theorem]{Corollary}

\newtheorem{remark}[theorem]{Remark}

\begin{document}
\title[The noncommutative Gurarij space]{Uniqueness, universality, and
homogeneity of the noncommutative Gurarij space}
\author{Martino Lupini}
\address{Martino Lupini\\
Department of Mathematics and Statistics\\
N520 Ross, 4700 Keele Street\\
Toronto Ontario M3J 1P3, Canada.}
\curraddr{Mathematics Department California Institute of Technology 1200 E.
California Blvd MC 253-37 Pasadena, CA 91125}
\email{lupini@caltech.edu}
\urladdr{http://www.lupini.org/}
\thanks{The author was supported by the York University Susan Mann
Dissertation Scholarship.}
\dedicatory{}

\begin{abstract}
We realize the noncommutative Gurarij space $\mathbb{NG}$ defined by
Oikhberg as the Fra\"{\i}ss\'{e} limit of the class of finite-dimensional $1$%
-exact operator spaces. As a consequence we deduce that the noncommutative
Gurarij space is unique up to completely isometric isomorphism, homogeneous,
and universal among separable $1$-exact operator spaces. We also prove that $%
\mathbb{NG}$ is the unique separable nuclear operator space with the
property that the canonical triple morphism from the universal TRO to the
triple envelope is an isomorphism. We deduce from this fact that $\mathbb{NG}
$ does not embed completely isometrically into an exact C*-algebra, and it
is not completely isometrically isomorphic to a C*-algebra or to a TRO. We
also provide a canonical construction of $\mathbb{NG}$, which shows that the
group of surjective complete isometries of $\mathbb{NG}$ is universal among
Polish groups. Analog results are proved in the commutative setting and,
more generally, for $M_{n}$-spaces. In particular, we provide a new
characterization and canonical construction of the Gurarij Banach space.
\end{abstract}

\maketitle

\section{Introduction}

The \emph{Gurarij space} $\mathbb{G}$ is a Banach space first constructed by
Gurarij in \cite{gurarij_spaces_1966}. It has the following universal
property: whenever $X\subset Y$ are finite-dimensional Banach spaces, $\phi
:X\rightarrow \mathbb{G}$ is a linear isometry, and $\varepsilon >0$, there
is an injective linear map $\psi :Y\rightarrow \mathbb{G}$ extending $\phi $
such that $||\psi ||{\;}||\psi ^{-1}||{}<1+\varepsilon $. The uniqueness of
such an object was proved by Lusky \cite{lusky_gurarij_1976}. A short proof
was later provided by Kubi\'{s} and Solecki \cite{kubis_proof_2013}. The
Gurarij space was realized as a Fra\"{\i}ss\'{e} limit by Ben Yaacov in \cite%
{ben_yaacov_fraisse_2012}.

\emph{Fra\"{\i}ss\'{e} theory} is a subject at the border between model
theory and combinatorics originating from the seminar work of Fra\"{\i}ss%
\'{e} \cite{fraisse_lextension_1954}. Broadly speaking, Fra\"{\i}ss\'{e}
theory studies homogeneous structures and ways to construct them. In the
discrete setting Fra\"{\i}ss\'{e} established in \cite%
{fraisse_lextension_1954} a correspondence between countable homogeneous
structures and what are now called Fra\"{\i}ss\'{e} classes. Let the \emph{%
age }of a countable structure $\mathbb{S}$ be the collection of finitely
generated substructures of $\mathbb{S}$. Any Fra\"{\i}ss\'{e} class is the
age of a countable homogeneous structure. Conversely from any Fra\"{\i}ss%
\'{e} class one can build a countable homogeneous structure that has the
given class as its age. Moreover such a structure is uniquely determined up
to isomorphism by this property.

This correspondence has been recently generalized in \cite%
{ben_yaacov_fraisse_2012} by Ben Yaacov from the purely discrete setting to
the setting where metric structures are considered; see also \cite%
{Schoretsanitis_fraisse_2007}. The main results of discrete Fra\"{\i}ss\'{e}
theory are recovered in this more general framework. In particular any Fra%
\"{\i}ss\'{e} class of metric structures is the age of a separable
homogeneous structure, which is unique up to isometric isomorphism. An
alternative category-theoretic approach to Fra\"{\i}ss\'{e} limits in the
metric setting has been developed by Kubi\'{s} \cite{kubis_fraisse_2014}.

The Gurarij space is the limit of the Fra\"{\i}ss\'{e} class of
finite-dimensional Banach spaces. This has been showed in \cite%
{ben_yaacov_fraisse_2012} building on previous work of Henson. In particular
this has yielded an alternative proof of the uniqueness of the Gurarij space
up to isometric isomorphism. Other naturally occurring examples of Fra\"{\i}%
ss\'{e} limits are the Urysohn universal metric space \cite%
{melleray_extremely_2014}, the hyperfinite II$_{1}$ factor, the infinite
type UHF\ C*-algebras, the Cuntz algebra $\mathcal{O}_{2}$ \cite%
{eagle_fraisse_2014}, and the Jiang-Su algebra $\mathcal{Z}$ \cite%
{masumoto_jiang-su_2016}.

In this paper we consider a noncommutative analog of the Gurarij space
introduced by Oikhberg in \cite{oikhberg_non-commutative_2006} within the
framework of operator spaces. Operator spaces can be regarded as
noncommutative Banach spaces. In fact Banach spaces can be concretely
defined as closed subspaces of $C(K)$ spaces, where $K$ is a compact
Hausdorff space. These are precisely the \emph{abelian }unital C*-algebras.
Replacing abelian C*-algebras with arbitrary C*-algebras
or---equivalently---the algebra $B(H)$ of bounded linear operators on some
Hilbert space $H$ provides the notion of an \emph{operator space}.

An operator space $X\subset B(H)$ is endowed with matricial norms on the
algebraic tensor product $M_{n}\otimes X$ obtained by the inclusion $%
M_{n}\otimes X\subset M_{n}\otimes B(H)\cong B\left( H\oplus \cdots \oplus
H\right) $. A linear operator $\phi $ between operator spaces is completely
bounded with norm at most $M$ if all its amplifications $id_{M_{n}}\otimes
\phi $ are bounded with norm at most $M$. The notion of complete isometry is
defined similarly. Operator spaces form then a category with completely
bounded (or completely isometric) linear maps as morphisms. Any Banach space 
$X$ has a canonical operator space structure induced by the inclusion $%
X\subset C(\mathrm{Ball}\left( X^{\ast }\right) )$ where $\mathrm{Ball}%
\left( X^{\ast }\right) $ is the unit ball of the dual of $X$. However in
this case the matricial norms do not provide any new information, and any
linear map $\phi $ between Banach spaces is automatically completely bounded
with same norm. For more general operator spaces it is far from being true
that any bounded linear map is completely bounded. The matricial norms play
in this case a crucial role.

According to \cite{oikhberg_non-commutative_2006} an operator space is \emph{%
noncommutative Gurarij} if it satisfies the same universal property of the
Gurarij Banach space, where finite-dimensional Banach spaces are replaced
with arbitrary\emph{\ finite-dimensional }$1$\emph{-exact} \emph{operator
spaces}, and the operator norm is replaced by the \emph{completely bounded
norm}. The restriction to $1$-exact spaces is natural since a famous result
of Junge and Pisier implies that there is no separable operator space
containing all the finite-dimensional operator spaces as subspaces \cite[%
Theorem 2.3]{junge_bilinear_1995}; see also \cite[Chapter 21]%
{pisier_introduction_2003}. The existence of a noncommutative Gurarij space
has been established in \cite{oikhberg_non-commutative_2006}. In this paper,
we prove that such a space is in fact unique, and satisfies the natural
noncommutative analogs of the uniqueness and universality property of the
Gurarij Banach space. We also provide a characterization and canonical
construction of the noncommutative Gurarij space using the theory of \emph{%
ternary ring of operators} (TRO) \cite[\S 4.4]{blecher_operator_2004}. TROs
are a natural nonselfadjoint generalization of C*-algebras, and play a key
role in the study of operator spaces \cite%
{effros_injectivity_2001,junge_OL_2003}. A TRO is, briefly, an off-diagonal
corner of a unital C*-algebra. Equivalently, TROs can be defined as the
closed subspaces of $B\left( H\right) $ that are invariant under the \emph{%
ternary product} $\left( x,y,z\right) \mapsto xy^{\ast }z$. There are two
canonical TROs associated with an operator space: the universal TRO $%
\mathcal{T}_{u}\left( X\right) $---see \S \ref{Subsection:triple-NG}---and
the triple envelope $\mathcal{T}_{e}\left( X\right) $ \cite%
{hamana_triple_1999}. These are, respectively, the largest and the smallest
(in a projective sense) TROs containing a completely isometric copy of $X$
as a generating set. The triple envelope $\mathcal{T}_{e}\left( X\right) $
is also known as the \emph{noncommutative Shilov boundary }of $X$, and can
be regarded as the noncommutative generalization of the Shilov boundary of a
function system \cite{blecher_shilov_2001}. The universal property of these
objects yields a canonical morphism $\sigma _{X}:\mathcal{T}_{u}\left(
X\right) \rightarrow \mathcal{T}_{e}\left( X\right) $.

The following statements summarizes the main results of the present paper.

\begin{theorem}
\label{Theorem:main-NG}Let $\mathbb{NG}$ denote the noncommutative Gurarij
space as defined by Oikhberg.

\begin{enumerate}
\item There exists a unique noncommutative Gurarij space up to completely
isometric isomorphism (uniqueness); see \S \ref{Subsection: NG as limit}.

\item Every separable $1$-exact operator space embeds completely
isometrically into $\mathbb{NG}$ (universality); see \S \ref{Subsection: NG
as limit}.

\item Every complete isometry between finite-dimensional subspaces of $%
\mathbb{NG}$ extends to a surjective complete isometry of $\mathbb{NG}$
(homogeneity). More precisely, if $E\subset \mathbb{NG}$ is a
finite-dimensional subspace and $\phi :E\rightarrow \mathbb{NG}$ is an
injective linear map such that $||\phi ||_{cb}<1+\delta $ and $||\phi
^{-1}||_{cb}<1+\delta $, then there exists a surjective complete isometry $%
\alpha $ of $\mathbb{NG}$ such that $\left\Vert \alpha |_{E}-\phi
\right\Vert <\delta $; see \S \ref{Subsection: NG as limit}.

\item $\mathbb{NG}$ is the unique separable nuclear operator space of
dimension at least $1$ with the property that the canonical triple morphism
from the universal TRO to the triple envelope is injective; see \S \ref%
{Subsection:triple-NG} and \S \ref{Subsection:characterization-NG}.

\item Suppose that $X$ is any separable $M_{n}$-space of dimension at least $%
1$. Define recursively for $k>n$, $X_{k+1}$ to be the universal $k$-minimal
TRO of $X_{k}$. Then the limit of the inductive sequence $\left(
X_{k}\right) $ is completely isometric to $\mathbb{NG}$, any surjective
complete isometry $\alpha $ of $X$ extends to a surjective complete isometry 
$\widehat{\alpha }$ of $\mathbb{NG}$, and the map $\alpha \mapsto \widehat{%
\alpha }$ is a group homomorphism and a topological embedding; see \S \ref%
{Subsection:canonical-NG}.

\item Every Polish group is isomorphic to a closed subgroup of the group $%
\mathrm{Aut}\left( \mathbb{NG}\right) $ of surjective complete isometries of 
$\mathbb{NG}$; see \S \ref{Subsection:universal}.

\item $\mathbb{\mathbb{NG}}$ does not embed completely isometrically into an
exact C*-algebra or an exact TRO; see \S \ref{Subsection:triple-NG}.

\item $\mathbb{NG}$ is not completely isometrically isomorphic to a
C*-algebra or a TRO; see \S \ref{Subsection:triple-NG}.
\end{enumerate}
\end{theorem}

The Gurarij operator space $\mathbb{NG}$ is the first example of a separable 
$1$-exact operator space that contains a completely isometric copy of any
other separable $1$-exact operator space. It is also the first example of a
separable $1$-exact operator space that does not embed completely
isometrically into an exact C*-algebra, and the first example of a separable 
$1$-exact operator space that is not an $M_{n}$-space for any $n\in \mathbb{N%
}$ and whose automorphism group is universal among Polish groups. Item (8)
of Theorem \ref{Theorem:main-NG} answers a question of Oikhberg from \cite%
{oikhberg_non-commutative_2006}.

The analogous statements as in Theorem \ref{Theorem:main-NG} are also proved
for the class of $M_{n}$-spaces. We say that an $M_{n}$-space is Gurarij if
it satisfies the same property as the noncommutative Gurarij space, where $%
X\subset Y$ are only assumed to be $M_{n}$-spaces. When $n=1$ one recovers
the classical Gurarij Banach space. A TRO is $n$-minimal if it is an $M_{n}$%
-space. To any $M_{n}$-space $X$ one can associated the $n$-minimal analog
of the universal TRO, which we call the universal $n$-minimal TRO $\mathcal{T%
}_{u}^{n}\left( X\right) $. Again, one has a canonical morphism $\sigma
_{X}^{n}:\mathcal{T}_{u}^{n}\left( X\right) \rightarrow \mathcal{T}%
_{e}\left( X\right) $.

\begin{theorem}
\label{Theorem:main-G_n}Let $\mathbb{G}_{n}$ be a Gurarij $M_{n}$-space.

\begin{enumerate}
\item There exists a unique Gurarij $M_{n}$-space up to completely isometric
isomorphism (uniqueness); see \S \ref{Subsection:LimitMn}.

\item Every separable $M_{n}$-space embeds completely isometrically into $%
\mathbb{G}_{n}$ (universality); see \S \ref{Subsection:LimitMn}.

\item Every complete isometry between finite-dimensional subspaces of $%
\mathbb{G}_{n}$ extends to a surjective complete isometry of $\mathbb{G}_{n}$
(homogeneity); see \S \ref{Subsection:LimitMn}.

\item $\mathbb{G}_{n}$ is the unique separable nuclear $M_{n}$-space with
the property that the canonical triple morphism from the universal $n$%
-minimal TRO to the triple envelope is injective; see \S \ref%
{Subsection:triple-Gn}.

\item Suppose that $X$ is any separable $M_{n}$-space of dimension at least $%
1$. Define recursively $X_{k+1}$ to be the universal $n$-minimal TRO of $%
X_{k}$. Then the limit of the inductive sequence $\left( X_{k}\right) $ is
completely isometric to $\mathbb{G}_{n}$, any surjective complete isometry $%
\alpha $ of $X$ extends to a surjective complete isometry $\widehat{\alpha }$
of $\mathbb{G}_{n}$, and the map $\alpha \mapsto \widehat{\alpha }$ is a
group homomorphism and a topological embedding; see \S \ref%
{Subsection:canonical-NG}.

\item Every Polish group is isomorphic to a closed subgroup of the group $%
\mathrm{Aut}\left( \mathbb{G}_{n}\right) $ of surjective complete isometries
of $\mathbb{G}_{n}$; see \S \ref{Subsection:universal};

\item $\mathbb{G}_{n}$ can not be written in a nontrivial way as a tensor
product of two rigid rectangular $\mathcal{OL}_{\infty ,1+}$ spaces.
\end{enumerate}
\end{theorem}

Items (4),(5) of Theorem \ref{Theorem:main-G_n} are new even in the case $%
n=1 $. They provide a new characterization of the Gurarij Banach space $%
\mathbb{G}$ among separable Lindenstrauss space, and an explicit canonical
construction of $\mathbb{G}$. The fact that the group of surjective linear
isometries of the \emph{real }Gurarij Banach space is a universal Polish
group is a result of Ben Yaacov from \cite{ben_yaacov_linear_2014}. Ben
Yaacov's proof seems to make an essential use of the fact that the scalars
are real. Item (6) of Theorem \ref{Theorem:main-G_n} for $n=1$ shows in
particular that the same conclusion holds over the complex field.

The operator system analog of Theorem \ref{Theorem:main-NG}, where one
replaces operator spaces with operator systems and TROs with unital
C*-algebras, also holds. The Gurarij operator system $\mathbb{GS}$ is the
unique nuclear operator system that is universal in the sense of sense of
Kirchberg and Wasserman\cite[\S 6]{kirchberg_c*-algebras_1998}. It is also
the unique nuclear operator system whose matrix state space has dense matrix
extreme boundary in the sense of \cite{farenick_extremal_2000}. This shows
that the matrix state space of $\mathbb{GS}$ can be regarded as the
noncommutative analog of the Poulsen simplex \cite%
{poulsen_simplex_1961,lindenstrauss_poulsen_1978}. These results will be
included in \cite{davidson_noncommutative_2016}, where many classical
results from the theory of simplices are generalized to the noncommutative
setting.

This rest of this paper is divided into four sections. Section \ref%
{Section:background-material} contains some background material on Fra\"{\i}%
ss\'{e} theory and operator spaces. We follow the presentation of Fra\"{\i}ss%
\'{e} theory for metric structures as introduced by Ben Yaacov in \cite%
{ben_yaacov_fraisse_2012}. Similarly as \cite{melleray_extremely_2014} we
adopt the slightly less general point of view---sufficient for our
purposes---where one considers only structures where the interpretation of
function and relation symbols are Lipschitz with a constant that does not
depend on the structure. The material on operator spaces is standard and can
be found for example in the monographs \cite%
{pisier_introduction_2003,effros_operator_2000,paulsen_completely_2002}. The
topic of $M_{n}$-spaces is perhaps less well known and can be found in
Lehrer's PhD thesis \cite{lehner_mn-espaces_1997} as well as in \cite%
{oikhberg_operator_2004,oikhberg_non-commutative_2006}.

In Section \ref{Section:LimitMn} we show that the class of
finite-dimensional $M_{n}$-spaces is a Fra\"{\i}ss\'{e} class. This can be
seen as a first step towards proving that the class of finite-dimensional $1$%
-exact operator spaces is a Fra\"{\i}ss\'{e} class. Any $M_{n}$-space can be
canonically endowed with a compatible operator space structure. Therefore in
principle it is possible to rephrase all the arguments and results in terms
of operator spaces. Nonetheless we find it more convenient and enlightening
to deal with $M_{n}$-space. This allows one to recognize and use the analogy
with the Banach space case.

Section \ref{Section:NG} contains the proof that the class of
finite-dimensional $1$-exact operator spaces is a Fra\"{\i}ss\'{e} class,
and the limit is the noncommutative Gurarij space as defined by Oikhberg.
Finally Section \ref{Section:properties} contains the proof of further
properties of the noncommutative Gurarij space $\mathbb{NG}$ and the Gurarij 
$M_{n}$-spaces $\mathbb{G}_{n}$, as stated in Theorem \ref{Theorem:main-NG}
and Theorem \ref{Theorem:main-G_n}.

\subsubsection*{Acknowledgments}

We would like to thank Kenneth Davidson, Ilijas Farah, Isaac Goldbring,
Michael Hartz, Marius Junge, Alexander Kechris, Matthew Kennedy, Jorge L\'{o}%
pez-Abad, Wieslaw Kubi\'{s}, Timur Oikhberg, and Todor Tsankov for many
helpful comments and suggestions. Many thanks are due to the anonymous
referee, whose suggestions and remarks significantly contributed to improve
the present paper. We are also grateful to Nico Spronk for the inspiring
course \textquotedblleft Fourier and Fourier-Stieltjes Algebras, and their
Operator Space Structure\textquotedblright\ that he gave at the Fields
Institute in March-April 2014. Finally, we would like to thank Caleb
Eckhardt for suggesting the proof of Proposition \ref%
{Proposition:perturbation}, and for letting us include it here.

\section{Background material\label{Section:background-material}}

\subsection{Approximate isometries\label{Subsection:approximate isometries}}

Suppose that $A,B$ are complete metric spaces. A\emph{\ Katetov function} on 
$A$ is a $1$-Lipchitz map $f:X\rightarrow \left[ 0,+\infty \right] $
satisfying $d\left( x,y\right) \leq f(x)+f(y)$. An \emph{approximate isometry%
} from $A$ to $B$ is a map $\psi :A\times B\rightarrow \left[ 0,+\infty %
\right] $ that is a Katetov function in each variable. If $\psi $ is an
approximate isometry from $A$ to $B$, then we write $\psi :A\leadsto B$. The
set of all approximate isometries from $A$ to $B$ is denoted by $\mathrm{Apx}%
\left( A,B\right) $. This is a compact space endowed with the product
topology from $\left[ 0,+\infty \right] ^{A\times B}$. A \emph{partial
isometry }$f$\emph{\ }from $A$ to $B$ is an isometry from a subset $\mathrm{%
dom}\left( f\right) $ of $A$ to $B$.

\begin{remark}
\label{Remark: piso is apx}Any partial isometry $f$ will be identified with
the approximate isometry $\psi _{f}$ given by the distance function from the
graph of $f$.
\end{remark}

If $\psi :A\leadsto B$ one can consider its \emph{pseudo-inverse }$\psi
^{\ast }:B\leadsto A$ defined by $\psi ^{\ast }\left( b,a\right) =\psi
\left( a,b\right) $. Moreover one can take \emph{composition} of approximate
isometries $\psi :A\leadsto B$ and $\phi :B\leadsto C$ by setting $\left(
\phi \psi \right) \left( a,c\right) $ to be the infimum of $\psi \left(
a,b\right) +\phi \left( b,c\right) $ for $b\in B$. These definitions are
consistent with composition and inversion of partial isometries when
regarded as approximate isometries.

If $A_{0}\subset A$, $B_{0}\subset B$, and $\psi :A\leadsto B$ then one can
define the \emph{restriction} $\psi |_{A_{0}\times B_{0}}=j^{\ast }\psi
i:A_{0}\leadsto B_{0}$ where $i$ and $j$ are the inclusion maps of $A_{0}$
into $A$ and $B_{0}$ into $B$. Conversely if $\phi :A_{0}\leadsto B_{0}$
then one can consider its \emph{trivial extension} $j\phi i^{\ast
}:A\leadsto B$. This allows one to regard $\mathrm{Apx}\left(
A_{0},B_{0}\right) $ as a subset of $\mathrm{Apx}\left( A,B\right) $ by
identifying an approximate isometry with its trivial extension.

For approximate isometries $\phi ,\psi :A\leadsto B$ we say that $\phi $ 
\emph{refines }$\psi $ and $\psi $ \emph{coarsens }$\phi $---written $\phi
\leq \psi $---if $\phi \left( a,b\right) \leq \psi \left( a,b\right) $ for
every $a\in A$ and $b\in B$. The set of approximate isometries that refine $%
\psi $ is denoted by $\mathrm{Apx}^{\leq \psi }\left( A,B\right) $. The
interior of $\mathrm{Apx}^{\leq \psi }\left( A,B\right) $ is denote by $%
\mathrm{Apx}^{<\psi }\left( A,B\right) $. The \emph{closure under coarsening 
}$\mathcal{A}^{\uparrow }$ of a set $\mathcal{A}\subset \mathrm{Apx}\left(
A,B\right) $ is the collection of $\phi \in \mathrm{Apx}\left( A,B\right) $
that coarsen some element of $\mathcal{A}$.

\subsection{Languages and structures\label{Subsection: languages and
structures}}

A \emph{language} $\mathcal{L}$ is given by sets of predicate symbols and of
function symbols. Every symbol has two natural numbers attached: its arity
and its Lipschitz constant. An $\mathcal{L}$-\emph{structure} $\mathfrak{A}$
is given by: a complete metric space $A$; a $c_{B}$-Lipschitz function $B^{%
\mathfrak{A}}:A^{n_{B}}\rightarrow \mathbb{R}$ for every predicate symbol $B$%
, where $c_{B}$ is the Lipschitz constant of $B$ and $n_{B}$ is the arity of 
$B$; a $c_{f}$-Lipschitz function $f^{\mathfrak{A}}:A^{n_{f}}\rightarrow A$
for every function symbol $f$, where $c_{f}$ is the Lipschitz constant of $f$
and $n_{f}$ is the arity of $f$.

Here and in the following we assume the power $A^{n}$ to be endowed with the
max metric $d(\bar{a},\bar{b})=\max_{i}d(a_{i},b_{i})$. An \emph{embedding }%
of $\mathcal{L}$-structures $\phi :\mathfrak{A}\rightarrow \mathfrak{B}$ is
a function that commutes with the interpretation of all the predicate and
function symbols. An \emph{isomorphism }is a surjective embedding. An \emph{%
automorphism} of $\mathfrak{A}$ is an isomorphism from $\mathfrak{A}$ to $%
\mathfrak{A}$. If $\bar{a}$ is a finite tuple in $\mathfrak{A}$ then $%
\left\langle \bar{a}\right\rangle $ denotes the smallest substructure of $%
\mathfrak{A}$ containing $\bar{a}$. A \emph{partial isomorphism }$\phi :%
\mathfrak{A}\dashrightarrow \mathfrak{B}$ is an embedding from $\left\langle 
\bar{a}\right\rangle $ to $\mathfrak{B}$ for some finite tuple $\bar{a}$ in $%
\mathfrak{A}$. An $\mathcal{L}$-structure $\mathfrak{A}$ is \emph{finitely
generated }if $\mathfrak{A}=\left\langle \bar{a}\right\rangle $ for some
finite tuple $\bar{a}$ in $\mathfrak{A}$.

We will assume that the language $\mathcal{L}$ contains a distinguished
binary predicate symbol to be interpreted as the metric. In particular this
ensures that all the embeddings and (partial) isomorphisms are (partial)
isometries. Therefore consistently with the convention from Remark \ref%
{Remark: piso is apx} partial isomorphisms will be regarded as approximate
isometries.

\begin{definition}
\label{Definition: NAP}Suppose that $\mathcal{C}$ is a class of
finitely-generated $\mathcal{L}$-structure. We say that $\mathcal{C}$
satisfies

\begin{itemize}
\item the \emph{hereditary property }(HP)\emph{\ }if $\left\langle \bar{a}%
\right\rangle \in \mathcal{C}$ for every $\mathfrak{A}\in \mathcal{C}$ and
finite tuple $\bar{a}\in \mathfrak{A}$,

\item the \emph{joint embedding property }(JEP)\emph{\ }if for any $%
\mathfrak{A},\mathfrak{B}\in \mathcal{C}$ there is $\mathfrak{C}\in \mathcal{%
C}$ and embeddings $\phi :\mathfrak{A}\rightarrow \mathfrak{C}$ and $\psi :%
\mathfrak{B}\rightarrow \mathfrak{C}$,

\item the \emph{near amalgamation property }(NAP) if, whenever $\mathfrak{A}%
\subset \mathfrak{B}_{0}$ and $\mathfrak{B}_{1}$ are elements of $\mathcal{C}
$, $\phi :\mathfrak{A}\rightarrow \mathfrak{B}_{1}$ is an embedding, $\bar{a}
$ is a finite tuple in $\mathfrak{A}$, and $\varepsilon >0$, there exists $%
\mathfrak{C}\in \mathcal{C}$ and embeddings $\psi _{0}:\mathfrak{B}%
_{0}\rightarrow \mathfrak{C}$ and $\psi _{1}:\mathfrak{B}_{1}\rightarrow 
\mathfrak{C}$ such that $d\left( \psi _{0}(\bar{a}),\left( \psi _{1}\circ
\varphi \right) (\bar{a})\right) \leq \varepsilon $.

\item the \emph{amalgamation property }(AP) if it satisfies\emph{\ }(NAP)
even when one takes $\varepsilon =0$.
\end{itemize}
\end{definition}

\subsection{Fra\"{\i}ss\'{e} classes and limits\label%
{Subsection:Fraisseclasses}}

Suppose in the following that $\mathcal{C}$ is a class of finitely generated 
$\mathcal{L}$-structures satisfying (HP), (JEP), and (NAP).

\begin{definition}
\label{Definition: C-structure}A $\mathcal{C}$-structure is an $\mathcal{L}$%
-structure $\mathfrak{A}$ such that $\left\langle \bar{a}\right\rangle \in 
\mathcal{C}$ for every finite tuple $\bar{a}$ in $\mathfrak{A}$.
\end{definition}

Let $\mathfrak{A}$ and $\mathfrak{B}$ be $\mathcal{C}$-structures. Define $%
\mathrm{Apx}_{1,\mathcal{C}}\left( \mathfrak{A},\mathfrak{B}\right) \subset 
\mathrm{Apx}\left( A,B\right) $ to be the set of all partial isomorphisms
from $\mathfrak{A}$ to $\mathfrak{B}$. Define $\mathrm{Apx}_{2,\mathcal{C}%
}\left( \mathfrak{A},\mathfrak{B}\right) $ to be the set of approximate
isometries $\phi :A\leadsto B$ of the form $\phi =g^{\ast }f$, where $f\in 
\mathrm{Apx}_{1,\mathcal{C}}\left( \mathfrak{A},\mathfrak{C}\right) $ and $%
g\in \mathrm{Apx}_{1,\mathcal{C}}\left( \mathfrak{B},\mathfrak{C}\right) $
for some $\mathcal{C}$-structure $\mathfrak{C}$. Finally set $\mathrm{Apx}_{%
\mathcal{C}}\left( \mathfrak{A},\mathfrak{B}\right) \subset \mathrm{Apx}%
\left( A,B\right) $ to be $\overline{\mathrm{Apx}_{2,\mathcal{C}}\left( 
\mathfrak{A},\mathfrak{B}\right) ^{\uparrow }}$. Elements of $\mathrm{Apx}_{%
\mathcal{C}}\left( \mathfrak{A},\mathfrak{B}\right) $ are called ($\mathcal{C%
}$-intrinsic) \emph{approximate morphism}. A ($\mathcal{C}$-intrinsic) \emph{%
strictly approximate morphism from }$\mathfrak{A}$ to $\mathfrak{B}$ is an
approximate morphism $\phi $ such that the interior $\mathrm{Apx}_{\mathcal{C%
}}^{<\phi }\left( \mathfrak{A},\mathfrak{B}\right) $ of $\mathrm{Apx}_{%
\mathcal{C}}^{\leq \phi }\left( \mathfrak{A}\mathbf{,}\mathfrak{B}\right) $
is nonempty. The set of strictly approximate morphisms from $\mathfrak{A}$
to $\mathfrak{B}$ is denoted by $\mathrm{Stx}_{\mathcal{C}}\left( \mathfrak{A%
},\mathfrak{B}\right) $.

Fix $k\in \mathbb{N}$ and denote by $\mathcal{C}(k)$ the set of pairs $%
\left( \bar{a},\mathfrak{A}\right) $ where $\mathfrak{A}\in \mathcal{C}$ and 
$\bar{a}$ is a finite tuple in $\mathfrak{A}$ such that $\mathfrak{A}%
=\,\left\langle \bar{a}\right\rangle $. Two such pairs $\left( \bar{a},%
\mathfrak{A}\right) $ and $\left( \bar{b},\mathfrak{B}\right) $ are
identified if there is an isomorphism $\phi :\mathfrak{A}\rightarrow 
\mathfrak{B}$ such that $\phi (\bar{a})=\bar{b}$. By abuse of notation we
will denote $\left( \bar{a},\mathfrak{A}\right) $ simply by $\bar{a}$.

\begin{definition}
\label{Definition:FraisseMetric}The Fra\"{\i}ss\'{e} metric $d_{\mathcal{C}}(%
\bar{a},\bar{b})$ on $\mathcal{C}(k)$ is defined to be the infimum of $%
\max_{i}\phi (a_{i},b_{i})$, where $\phi $ ranges in $\mathrm{Apx}_{\mathcal{%
C}}(\left\langle \bar{a}\right\rangle ,\left\langle \bar{b}\right\rangle )$
or, equivalently, in $\mathrm{Stx}_{\mathcal{C}}(\left\langle \bar{a}%
\right\rangle ,\left\langle \bar{b}\right\rangle )$.
\end{definition}

Such a metric can be equivalently described in terms of embeddings: it is
the infimum of $d(f(\bar{a}),g(\bar{b}))$, where $f,g$ range over all the
embeddings of $\left\langle \bar{a}\right\rangle $ and $\left\langle \bar{b}%
\right\rangle $ into a third structure $\mathfrak{C}\in \mathcal{C}$.

\begin{definition}
\label{Definition:Fraisseclass}Suppose that $\mathcal{C}$ is a class of
finitely-generated $\mathcal{L}$-structures satisfying (HP), (JEP), and
(NAP) from Definition \ref{Definition: NAP}. We say that $\mathcal{C}$ is a%
\emph{\ Fra\"{\i}ss\'{e} class} if the metric space $\left( \mathcal{C}%
(k),d_{\mathcal{C}}\right) $ is complete and separable for every $k\in 
\mathbb{N}$.
\end{definition}

\begin{remark}
In \cite[Definition 2.12]{ben_yaacov_fraisse_2012} a Fra\"{\i}ss\'{e} class
is moreover required to satisfy the Continuity Property. Such a property is
automatically satisfied in our more restrictive setting, where we assume
that the interpretation of any symbol from $\mathcal{L}$ is a Lipschitz
function with Lipschitz constant that does not depend from the structure.
\end{remark}

\begin{definition}
\label{Definition:Fraisselimit}Suppose that $\mathcal{C}$ is a Fra\"{\i}ss%
\'{e} class. A limit of $\mathcal{C}$ is a separable $\mathcal{C}$-structure 
$\mathfrak{M}$ satisfying the following property: For every $\mathfrak{A}\in 
\mathcal{C}$, finite tuple $\bar{a}$ in $\mathfrak{A}$, embedding $\phi
:\left\langle \bar{a}\right\rangle \rightarrow \mathfrak{M}$, and $%
\varepsilon >0$ there is an embedding $\psi :\mathfrak{A}\rightarrow 
\mathfrak{M}$ such that $d\left( \psi (\bar{a}),\phi (\bar{a})\right)
<\varepsilon $.
\end{definition}

The definition given above is equivalent to \cite[Definition 2.14]%
{ben_yaacov_fraisse_2012} in view of \cite[Corollary 2.20]%
{ben_yaacov_fraisse_2012}.

\begin{definition}
An $\mathcal{L}$-structure $\mathfrak{M}$ is \emph{homogeneous} if for every
finite tuple $\bar{a}$ in $\mathfrak{A}$, embedding $\phi :\left\langle \bar{%
a}\right\rangle \rightarrow \mathfrak{A}$, and $\varepsilon >0$, there is an
automorphism $\psi $ of $\mathfrak{M}$ such that $d\left( \phi (\bar{a}%
),\psi (\bar{a})\right) <\varepsilon $.
\end{definition}

The following theorem is a combination of the main results from \cite%
{ben_yaacov_fraisse_2012}.

\begin{theorem}[Ben Yaacov]
\label{Theorem:Fraisselimit}Suppose that $\mathcal{C}$ is a Fra\"{\i}ss\'{e}
class. Then $\mathcal{C}$ has a limit $\mathfrak{M}$. If $\mathfrak{M}%
^{\prime }$ is another limit of $\mathcal{C}$ then $\mathfrak{M}$ and $%
\mathfrak{M}^{\prime }$ are isomorphic as $\mathcal{L}$-structures. Moreover 
$\mathfrak{M}$ is homogeneous and contains any separable $\mathcal{C}$%
-structure as a substructure.
\end{theorem}

\subsection{Operator spaces\label{Subsection: operator spaces}}

An operator space is a closed subspace of $B(H)$. Here and in the following
we denote by $M_{n}$ the algebra of $n\times n$ complex matrices. If $X$ is
a complex vector space, then we denote by $M_{n}\otimes X$ the algebraic
tensor product. Observe that this can be canonically identified with the
space $M_{n}(X)$ of $n\times n$ matrices with entries from $X$. If $X\subset
B(H)$ is an operator space, then $M_{n}(X)$ is naturally endowed with a norm
given by the inclusion $M_{n}(X)\subset B(H\oplus \cdots \oplus H)$. If $%
\phi :X\rightarrow Y$ is a linear map between operator spaces, then its $n$%
-th amplification is the linear map $id_{M_{n}}\otimes \phi :M_{n}\otimes
X\rightarrow M_{n}\otimes Y$. Under the identification of $M_{n}\otimes X$
with $M_{n}(X)$ and $M_{n}\otimes Y$ with $M_{n}(Y)$, the map $%
id_{M_{n}}\otimes \phi $ is defined by applying $\phi $ to a matrix $\left[
x_{ij}\right] $ entrywise. We say that $\phi $ is completely bounded if the
norms of its amplifications are uniformly bounded. In such case we define
its completely bounded norm to be the least uniform upper bound for the
norms of its amplifications. A linear map $\phi $ is a complete contraction
if $id_{M_{n}}\otimes \phi $ is a contraction for every $n\in \mathbb{N}$.
It is a complete isometry if $id_{M_{n}}\otimes \phi $ is an isometry for
every $n\in \mathbb{N}$. Finally it is a completely isometric isomorphism if 
$id_{M_{n}}\otimes \phi $ is an isometric isomorphism for every $n\in 
\mathbb{N}$.

Operator spaces admit an abstract characterization due to Ruan \cite%
{ruan_subspaces_1988}. A vector space $X$ is matrix-normed if for every $%
n\in \mathbb{N}$ the space $M_{n}(X)$ is endowed with a norm such that
whenever $\alpha \in M_{k,n}$, $x\in M_{n}(X)$, and $\beta \in M_{n,k}$, $%
\left\Vert \alpha .x.\beta \right\Vert \leq \left\Vert \alpha \right\Vert
\left\Vert x\right\Vert \left\Vert \beta \right\Vert $, where $\alpha
.x.\beta $ denotes the matrix product, and $\left\Vert \alpha \right\Vert
,\left\Vert \beta \right\Vert $ are the norms of $\alpha $ and $\beta $
regarded as operators between finite-dimensional Hilbert spaces. A
matrix-normed vector space is $L^{\infty }$-matrix-normed provided that $%
\left\Vert x\oplus y\right\Vert =\max \left\{ \left\Vert x\right\Vert
,\left\Vert y\right\Vert \right\} $ for $x\in M_{n}(X)$ and $y\in M_{m}(X)$,
where $x\oplus y$ is the block-diagonal $\left( n+m\right) \times \left(
n+m\right) $ matrix having $x,y$ in the diagonal. Every operator space $%
X\subset B(H)$ is canonically an $L^{\infty }$-matrix-normed space. Ruan's
theorem asserts that, conversely, any $L^{\infty }$-matrix-normed space is
completely isometrically isomorphic to an operator space \cite[Theorem 13.4]%
{paulsen_completely_2002}.

Equivalently one can think of an operator space $X$ as a structure on $%
\mathcal{K}_{0}\otimes X$; see \cite[Section 2.2]{pisier_introduction_2003}.
Suppose that $H$ is the separable Hilbert space with a fixed orthonormal
basis $(e_{k})_{k\in \mathbb{N}}$. Let $P_{n}$ be the orthogonal projection
of $\mathrm{span}\left\{ e_{1},\ldots ,e_{n}\right\} $ for $n\in \mathbb{N}$%
. We can identify $M_{n}$ with the subspace of $A\in B(H)$ such that $%
AP_{n}=P_{n}A=A$. The direct union $\mathcal{K}_{0}=\bigcup_{n}M_{n}$ is a
subspace of $B\left( H\right) $. We can identify $\bigcup_{n}M_{n}(X)$ with $%
\mathcal{K}_{0}\left[ X\right] \cong \mathcal{K}_{0}\otimes X$. Then $%
\mathcal{K}_{0}\left[ X\right] $ is a complex vector space with a natural
structure of $\mathcal{K}_{0}$-bimodule. The metric on $\mathcal{K}_{0}\left[
X\right] $ is not complete. Nonetheless one can pass to the completion $%
\mathcal{K}\overline{\otimes }X$ and extend all the operations. (Here $%
\mathcal{K}$ is the closure of $\mathcal{K}_{0}$ inside $B(H)$, i.e.\ the
ideal of compact operators.)

The abstract characterization of operator spaces mentioned above allows one
to regard operator spaces as $\mathcal{L}_{OS}$-structures for a suitable
language $\mathcal{L}_{OS}$. Denote by $\mathcal{K}_{0}(\mathbb{Q}(i))$ the
space of matrices of arbitrarily large finite size and with coefficients\ in
the field of Gauss rationals $\mathbb{Q}(i)$. Then $\mathcal{L}_{OS}$
contains, in addition to the special symbol $d$ for the metric, a symbol $+$
for the addition in $\mathcal{K}\overline{\otimes }X$, a constant $0$ for
the zero vector in $\mathcal{K}\overline{\otimes }X$, function symbols $%
\sigma _{\alpha ,\beta }$ for $\alpha ,\beta \in \mathcal{K}_{0}(\mathbb{Q}%
(i))$ for the bimodule operation. The Lipschitz constant for the symbol $+$
is $2$, while the Lipschitz constant of $\sigma _{\alpha ,\beta }$ is $%
\left\Vert \alpha \right\Vert \left\Vert \beta \right\Vert $. An alternative
description of operator spaces as metric structures---which fits in the
framework of continuous logic \cite{ben_yaacov_model_2008,farah_model_2010}%
---has been provided in \cite[Section 3.3]{elliott_isomorphism_2013} and 
\cite[Appendix B]{goldbring_kirchbergs_2015}.

It is worth noting that the space $X$ can be described as the set of $x\in 
\mathcal{K}\overline{\otimes }X$ such that $1.x=x$. Moreover a linear map $%
\phi :\mathcal{K}\overline{\otimes }X\rightarrow \mathcal{K}\overline{%
\otimes }Y$ that respects the $\mathcal{K}_{0}$-bimodule operations is the
amplification of a linear map from $X$ to $Y$. Hence when operator spaces
are regarded as $\mathcal{L}_{OS}$-structures, embeddings and isomorphisms
as defined in Subsection \ref{Subsection: languages and structures}
correspond, respectively, to completely isometric linear maps and completely
isometric linear isomorphisms.

If $X$ and $Y$ are operator spaces, then the space $CB(X,Y)$ of completely
bounded linear maps from $X$ to $Y$ is canonically endowed with an operator
space structure given by identifying isometrically $M_{n}\left(
CB(X,Y)\right) $ with $CB\left( M_{n}(X),M_{n}(Y)\right) $ with the
completely bounded norm. Any linear functional $\phi $ on an operator space $%
X$ is automatically completely bounded with $\left\Vert \phi \right\Vert
_{cb}=\left\Vert \phi \right\Vert $. Therefore the dual space $X^{\ast }$ of 
$X$ can be regarded as the operator space $CB\left( X,\mathbb{C}\right) $.

If $X$ and $Y$ are operator spaces, then their $\infty $-sum $X\oplus
^{\infty }Y$ is the operator space supported on the algebraic direct sum $%
X\oplus Y$ endowed with norms $\left\Vert \left( x,y\right) \right\Vert
_{M_{n}(X\oplus ^{\infty }Y)}=\max \left\{
||x||_{M_{n}(X)},||y||_{M_{n}(Y)}\right\} $. The $\infty $-sum of a sequence
of operator spaces is defined analogously. The $1$-sum $X\oplus ^{1}Y$ is
the operator space obtained by identifying $X\oplus Y$ with $\left( X^{\ast
}\oplus ^{\infty }Y^{\ast }\right) ^{\ast }$. The norm $\left\Vert \left(
x,y\right) \right\Vert _{M_{n}(X\oplus ^{1}Y)}$ can be described as the
supremum of $\left\Vert \left( id_{M_{n}}\otimes u\right) (x)+\left(
id_{M_{n}}\otimes v\right) (y)\right\Vert $, where $u,v$ range over all
completely contractive maps from $X$ and $Y$ into $B(H)$; see \cite[Section
2.6]{pisier_introduction_2003}. In analogous fashion one can define the $1$%
-sum and the $\infty $-sum of a sequence of operator spaces.

We denote the $n$-fold $1$-sum of $\mathbb{C}$ by itself $\ell _{n}^{1}$,
and the $\infty $-sum of $\mathbb{C}$ by itself by $\ell _{n}^{\infty }$.
Moreover we denote by $\bar{e}=(e_{i})$ the canonical basis of $\ell
_{n}^{1} $ and by $\bar{e}^{\ast }$ its dual basis of $\ell _{n}^{\infty }$.
More generally, if $k\in \mathbb{N}$, then we let $\ell _{n}^{\infty }\left(
M_{k}\right) $ be the $n$-fold $\infty $-sum of copies of $M_{k}$.

\subsection{$M_{n}$-spaces\label{Subsection:Mn-spaces}}

In this subsection we recall the definition and basic properties of $M_{n}$%
-spaces as defined in \cite[Chapter I]{lehner_mn-espaces_1997}. A \emph{%
matricial }$n$\emph{-norm }on a space $X$ is a norm on $M_{n}(X)$ such that $%
\left\Vert \alpha .x.\beta \right\Vert \leq \left\Vert \alpha \right\Vert
\left\Vert x\right\Vert \left\Vert \beta \right\Vert $ for $\alpha ,\beta
\in M_{n}$ and $x\in M_{n}(X)$. Such a norm induces a norm on $M_{k}(X)$ for 
$k\leq n$ via the upper-right corner inclusion. An $L^{\infty }$-\emph{%
matrix-}$n$\emph{-norm }is a matricial $n$-norm satisfying moreover $%
\left\Vert x\oplus y\right\Vert =\max \left\{ x,y\right\} $ for $k\leq n$, $%
x\in M_{k}(X)$, and $y\in M_{n-k}(X)$.

Observe that $M_{n}$ has a natural $n$-norm obtained by identifying $%
M_{n}(M_{n})$ with $M_{n}\otimes M_{n}$ (spatial tensor product). More
generally if $K$ is a compact Hausdorff space then $C(K,M_{n})$ is $n$%
-normed by identifying $M_{n}\left( C(K,M_{n})\right) $ with $%
C(K,M_{n}\otimes M_{n})$. In particular $\ell ^{\infty }\left( \mathbb{N}%
,M_{n}\right) $ has a natural $n$-norm obtained by the identification with $%
C\left( \beta \mathbb{N},M_{n}\right) $.

If $X,Y$ are $n$-normed spaces, then a linear map $\phi :X\rightarrow Y$ is $%
n$\emph{-bounded }if $id_{M_{n}}\otimes \phi :M_{n}(X)\rightarrow M_{n}(Y)$
is bounded, and $\left\Vert \phi \right\Vert _{n}$ is by definition $%
\left\Vert id_{M_{n}}\otimes \phi \right\Vert $. The notions of $n$%
-contraction and $n$-isometry are defined similarly. Define $nB(X,Y)$ to be
the space of $n$-bounded linear functions from $X$ to $Y$ with norm $%
\left\Vert \cdot \right\Vert _{n}$. Identifying $M_{n}\left( X^{\ast
}\right) $ with $nB\left( X,\mathbb{C}\right) $ isometrically defines an $%
L^{\infty }$-matrix-$n$-norm on the dual $X^{\ast }$ of $X$. The same
argument allows one to define an $L^{\infty }$-matrix-$n$-norm on the second
dual $X^{\ast \ast }$.

An $L^{\infty }$-matricially-$n$-normed space is called an $M_{n}$-\emph{%
space }if it satisfies any of the following equivalent definitions---see 
\cite[Th\'{e}or\`{e}me I.1.9]{lehner_mn-espaces_1997}:

\begin{enumerate}
\item There is an $n$-isometry from $X$ to $B(H)$;

\item The canonical inclusion of $X$ into $X^{\ast \ast }$ is isometric;

\item $\left\Vert \sum_{i}\alpha _{i}x_{i}\beta _{i}\right\Vert \leq
\left\Vert \sum_{i}\alpha _{i}\alpha _{i}^{\ast }\right\Vert ^{\frac{1}{2}%
}\max_{i}\left\Vert x_{i}\right\Vert \left\Vert \sum_{i}\beta _{i}^{\ast
}\beta _{i}\right\Vert ^{\frac{1}{2}}$ for any $x_{i}\in M_{n}(X)$, $\alpha
_{i},\beta _{i}\in M_{n}$;

\item there is an $n$-isometry from $X$ to $C\left( X,M_{n}\right) $ for
some compact Hausdorff space $K$.
\end{enumerate}

Clearly the case $n=1$ gives the usual notion of Banach space.
Characterization (3) allows one to show that $M_{n}$-spaces can be seen as
structures in a suitable language $\mathcal{L}_{M_{n}}$. This is the same as
the language for operator space described in Subsection \ref{Subsection:
operator spaces} where one replaces $\mathcal{K}_{0}$ with $M_{n}$. When $%
M_{n}$-spaces are regarded as $\mathcal{L}_{M_{n}}$-spaces, embeddings and
isomorphisms as defined in Subsection \ref{Subsection: languages and
structures} correspond, respectively, to $n$-isometric linear maps and $n$%
-isometric linear isomorphisms.

The notions of quotient and subspace of an $M_{n}$-space can be defined
analogously as in the case of operator spaces. For instance, if $N$ is a
subspace of $X$, then the quotient of $X$ by $N$ is the $M_{n}$-space
structure supported on the algebraic quotient $X/N$ defined by $\left\Vert
x+M_{n}\left( N\right) \right\Vert =\sup_{\phi }\left\Vert \left(
id_{M_{n}}\otimes \phi \right) (x)\right\Vert $ for every $x\in M_{n}\left(
X\right) $, where $\phi $ ranges among all the $n$-contractions from $X$ to $%
M_{n}$ whose kernel contains $N$. Similarly the constructions of $1$-sum and 
$\infty $-sum can be performed in this context. More details can be found in 
\cite[Section I.2]{lehner_mn-espaces_1997}. We will use the same notations
for the $1$-sum and $\infty $-sum of operator spaces and $M_{n}$-spaces.
This will be clear from the context and no confusion should arise.

For later use we recall the following observation; see \cite[Remarque I.1.5]%
{lehner_mn-espaces_1997}. Suppose that $X$ is a finite-dimensional $M_{n}$%
-space, and that $\bar{b}$, $\bar{b}^{\ast }$ is a bi-orthonormal system for 
$X$. Then the $n$-norm on $X$ admits the following expression: $%
||\sum_{i}\alpha _{i}\otimes b_{i}||$ is the supremum of $||\sum_{i}\alpha
_{i}\otimes \beta _{i}||$ where the $\beta _{i}$'s range among the elements
of $M_{n}$ such that $||\sum_{i}\beta _{i}\otimes b_{i}^{\ast }||{}\leq 1$.
In particular if $\bar{e}$ is the canonical basis of $\ell ^{1}(n)$ with
dual basis $\bar{e}^{\ast }$ of $\ell ^{\infty }(n)$, then we obtain that $%
||\sum_{i}\alpha _{i}\otimes b_{i}||$ is the supremum of $||\sum_{i}\alpha
_{i}\otimes \beta _{i}||$ where the $\beta _{i}$'s range among the elements
of $M_{n}$ of operator norm at most $1$. Similar expressions hold for the
matrix norms in operator spaces; see \cite{pisier_introduction_2003}.

In the following we will often use tacitly the fact that finite-dimensional $%
M_{n}$-spaces can be approximated by subspaces of finite $\infty $-sums of
copies of $M_{n}$. We let $\ell _{k}^{\infty }\left( M_{n}\right) $ denote
the $\infty $-sum of $k$ copies of $M_{n}$.

\begin{lemma}
\label{Lemma:approximation}Suppose that $X$ is a finite-dimensional $M_{n}$%
-space and $\varepsilon >0$. Then there is $k\in \mathbb{N}$ and an
injective linear $n$-contraction $\phi :X\rightarrow \ell _{k}^{\infty
}\left( M_{n}\right) $ such that $||\phi ^{-1}||_{n}\leq 1+\varepsilon $.
\end{lemma}

In Lemma \ref{Lemma:approximation} the map $\phi $ is not assumed to be
surjective. The expression $||\phi ^{-1}||_{n}$ denotes the $n$-norm of $%
\phi ^{-1}$ when regarded as a map from the range of $\phi $ to $X$. Similar
conventions will be adopted in the rest of the present paper. We conclude by
recalling that the natural analog of the Hahn-Banach theorem holds for $%
M_{n} $-spaces. Such an analog asserts that $M_{n}$ is an injective element
in the category of $M_{n}$-spaces with $n$-contractive maps as morphisms;
see \cite[Proposition I.1.16]{lehner_mn-espaces_1997}.

\section{The Fra\"{\i}ss\'{e} class of finite-dimensional 
\texorpdfstring{$M_{n}$}{Mn}-spaces \label{Section:LimitMn}}

The purpose of this section is to show that the class $\mathcal{M}_{n}$ of
finite-dimensional $M_{n}$-spaces is a complete Fra\"{\i}ss\'{e} class as in
Definition \ref{Definition:Fraisseclass}. This will allow us to consider the
corresponding Fra\"{\i}ss\'{e} limit as in Theorem \ref{Theorem:Fraisselimit}%
. The case $n=1$ of these results recovers the already known fact that
finite-dimensional Banach spaces form a complete Fra\"{\i}ss\'{e} class.
This has been shown by Ben Yaacov \cite[Section 3.3]{ben_yaacov_fraisse_2012}
building on previous works of Henson (unpublished) and Kubis-Solecki \cite%
{kubis_proof_2013}. For Banach spaces the limit is the Gurarij Banach space,
introduced by Gurarij in \cite{gurarij_spaces_1966} and proved to be unique
up to isometric isomorphism by Lusky in \cite{lusky_gurarij_1976}.

\subsection{Amalgamation property\label{Subsection:amalgamationMn}}

The properties (JEP) and (HP) as in Definition \ref{Definition: NAP} are
clear for $\mathcal{M}_{n}$. We now show that $\mathcal{M}_{n}$ has (AP).
The proof is analogous to the one for Banach spaces, and consists in showing
that the category of finite-dimensional $M_{n}$-spaces has pushouts; see 
\cite[Lemma 2.1]{garbulinska-wegrzyn_universal_2013} and also \cite[Lemma 4.5%
]{li_order-unit_2006}.

\begin{lemma}
\label{Lemma: approximate amalgamation}Suppose that $X_{0}\subset X$ and $Y$
are $M_{n}$-spaces, $\delta \geq 0$, and $f:X\rightarrow Y$ is a linear
injective map such that $\left\Vert f\right\Vert _{n}\leq 1+\delta $ and $%
\left\Vert f^{-1}\right\Vert _{n}\leq 1+\delta $. Then there is an $M_{n}$%
-space $Z$ and $n$-isometric linear maps $i:X\rightarrow Z$ and $%
j:Y\rightarrow Z$ such that $\left\Vert j\circ f-i|_{X_{0}}\right\Vert
_{n}\leq \delta $.
\end{lemma}

\begin{proof}
Define $\delta X_{0}$ to be the $M_{n}$-space structure on $X_{0}$ given by
the norm $\left\Vert \left[ x_{ij}\right] \right\Vert _{M_{n}\left( \delta
X\right) }=\delta \left\Vert \left[ x_{ij}\right] \right\Vert _{M_{n}(X)}$.
Let $\widehat{Z}$ be the $1$-sum $X\oplus ^{1}Y\oplus ^{1}\delta X_{0}$, and 
$Z$ be the quotient of $\widehat{Z}$ by the subspace $N=\{\left(
-z,f(z),z\right) \in \widehat{Z}:z\in X_{0}\}$. Finally let $i:X\rightarrow
Z $ and $j:Y\rightarrow Z$ the embeddings given by $x\mapsto \left(
x,0,0\right) +N$ and $y\mapsto \left( 0,y,0\right) +N$. We claim that $i$
and $j$ satisfy the desired conclusions. In fact it is clear that $i$ and $j$
are $n$-contractions such that $\left\Vert i\circ f-j\right\Vert _{n}\leq
\delta $. We will now show that $i$ is an $n$-isometry. The proof that $j$
is an $n$-isometry is similar. Suppose that $x\in M_{n}(X)$ consider a
linear $n$-contraction $\phi :X\rightarrow M_{n}$ such that $\left\Vert \phi
(x)\right\Vert _{M_{n}\otimes M_{n}}=\left\Vert x\right\Vert _{M_{n}(X)}$.
Observe that $\left( 1+\delta \right) ^{-1}(\phi \circ f^{-1}):f\left[ X%
\right] \rightarrow M_{n}$ is an linear $n$-contraction and hence it extends
to a linear $n$-contraction $\psi :f\left[ X\right] \rightarrow M_{n}$.
Similarly the map $\delta \left( 1+\delta \right) ^{-1}\phi :\delta
X_{0}\rightarrow M_{n}$ is a linear $n$-contraction. Therefore we have that
for every $z\in M_{n}(X_{0})$%
\begin{eqnarray*}
\left\Vert \left( x-z,f(z),z\right) \right\Vert _{M_{n}\left( X\oplus
^{1}Y\oplus ^{1}\delta X_{0}\right) } &\geq &\left\Vert \phi \left(
x-z\right) +\psi (z)+\frac{\delta }{1+\delta }\phi (z)\right\Vert
_{M_{n}\otimes M_{n}} \\
&=&\left\Vert \phi (x)-\phi (z)+\frac{1}{1+\delta }\phi (z)+\frac{\delta }{%
1+\delta }\phi (z)\right\Vert _{M_{n}\otimes M_{n}}=\left\Vert \phi
(x)\right\Vert _{M_{n}\otimes M_{n}}\text{.}
\end{eqnarray*}%
Since $\left\Vert \phi (x)\right\Vert _{M_{n}\otimes M_{n}}=\left\Vert
x\right\Vert _{M_{n}\left( X\right) }$, this concludes the proof that $i$ is
an $n$-isometry.
\end{proof}

In particular Lemma \ref{Lemma: approximate amalgamation} for $\delta =0$
shows that the class $\mathcal{M}_{n}$ has (AP).

\subsection{The Fra\"{\i}ss\'{e} metric space\label{Subsection:PolishMn}}

We fix now $k\in \mathbb{N}$ and consider the space $\mathcal{M}_{n}(k)$ of
pairs $\left( \bar{a},X\right) $ such that $X$ is a $k$-dimensional $M_{n}$%
-space and $\bar{a}$ is a linear basis of $X$. Two such pairs $\left( \bar{a}%
,X\right) $ and $\left( \bar{b},Y\right) $ are identified if there is an $n$%
-isometry $\phi $ from $X$ to $Y$ such that $\phi (\bar{a})=\bar{b}$. For
simplicity we will write an element $\left( \bar{a},X\right) $ of $\mathcal{M%
}_{n}(k)$ simply by $\bar{a}$, and denote $X$ by $\left\langle \bar{a}%
\right\rangle $. Our goal is to compute the Fra\"{\i}ss\'{e} metric in $%
\mathcal{M}_{n}(k)$ as in Definition \ref{Definition:FraisseMetric}. The
following result gives an equivalent characterization of such a metric. The
case $n=1$ is a result of Henson (unpublished) that can be found in \cite[%
Fact 3.2]{ben_yaacov_fraisse_2012}. We denote by $\ell ^{1}(k)$ the $k$-fold 
$1$-sum of $\mathbb{C}$ by itself in the category of $M_{n}$-spaces with
canonical basis $\bar{e}$. An explicit formula for the corresponding norm
has been recalled at the end of Section \ref{Subsection:Mn-spaces}.

\begin{proposition}
\label{Proposition:FraissemetricMn}Suppose that $\bar{a},\bar{b}\in \mathcal{%
M}_{n}(k)$ and $M>0$. The following statements are equivalent:

\begin{enumerate}
\item $d_{\mathcal{M}_{n}}(\bar{a},\bar{b})\leq M$;

\item For every $n$-contractive $u:\left\langle \bar{a}\right\rangle
\rightarrow M_{n}$ there is an $n$-contractive $v:\left\langle \bar{b}%
\right\rangle \rightarrow M_{n}$ such that the linear function $w:\ell
^{1}(k)\rightarrow M_{n}$ defined by $w(e_{i})=u(a_{i})-v(b_{i})$ has $n$%
-norm at most $M$, and vice versa.
\end{enumerate}
\end{proposition}

\begin{proof}
After normalizing we can assume that $M=1$. We will denote $\left\langle 
\bar{a}\right\rangle $ by $X$ and $\left\langle \bar{b}\right\rangle $ by $Y$%
.

(1)$\Rightarrow $(2): Suppose that $d_{\mathcal{M}_{n}}(\bar{a},\bar{b})\leq
1$. Then there are $n$-isometries $\phi :X\rightarrow Z$ and $\psi
:Y\rightarrow Z$ for some $M_{n}$-space $Z$ such that $\left\Vert \phi
(a_{i})-\psi (b_{i})\right\Vert \leq 1$ for every $i\leq k$. Suppose that $%
u:X\rightarrow M_{n}$ is $n$-contractive. Consider the $n$-contractive map $%
u\circ \phi ^{-1}:\phi \left[ X\right] \rightarrow M_{n}$. By injectivity of 
$M_{n}$ there is an $n$-contractive map $\eta :Z\rightarrow M_{n}$ extending 
$u\circ \phi ^{-1}$. Define $v=\eta \circ \psi :Y\rightarrow M_{n}$ and
observe that it is $n$-contractive. Define now $w:\ell ^{1}(k)\rightarrow
M_{n}$ by $w(e_{i})=u(a_{i})-v(b_{i})$. We claim that $w$ is $n$%
-contractive. In fact $\left\Vert \eta \left( \phi (e_{i})-\psi
(e_{i})\right) \right\Vert \leq 1$ for every $i\leq k$. Therefore if $\alpha
_{i}\in M_{n}$, then%
\begin{equation*}
\left\Vert \left( id_{M_{n}}\otimes w\right) \left( \sum_{i}\alpha
_{i}\otimes e_{i}\right) \right\Vert =\left\Vert \sum_{i}\alpha _{i}\otimes
\eta \left( \phi (e_{i})-\psi (e_{i})\right) \right\Vert \leq \left\Vert
\sum_{i}\alpha _{i}\otimes e_{i}\right\Vert _{M_{n}(\ell ^{1}(k))}\text{.}
\end{equation*}%
The vice versa is proved analogously.

(2)$\Rightarrow $(1): Suppose that for every $n$-contractive $u:X\rightarrow
M_{n}$ there is an $n$-contractive $v:Y\rightarrow M_{n}$ such that the
linear function $w:\ell ^{1}(k)\rightarrow M_{n}$ defined by $%
w(e_{i})=u(a_{i})-v(b_{i})$ is $n$-contractive, and vice versa. Define $%
\widehat{Z}$ to be $X\oplus ^{1}Y\oplus ^{1}\ell ^{1}(k)$. Denote by $N$ the
closed subspace%
\begin{equation*}
\left\{ \left( -\sum_{i}\lambda _{i}a_{i},\sum_{i}\lambda
_{i}b_{i},\sum_{i}\lambda _{i}e_{i}\right) :\lambda _{i}\in \mathbb{C}%
\right\}
\end{equation*}%
of $\widehat{Z}$. Define $Z$ to be the quotient of $\widehat{Z}$ by $N$. Let 
$\phi $ be the composition of the canonical inclusion of $X$ into $\widehat{Z%
}$ with the quotient map from $\widehat{Z}$ to $Z$. Similarly define $\psi
:Y\rightarrow Z$. By the properties of $1$-sums and quotients, $\phi $ and $%
\psi $ are $n$-contractions. We claim that they are in fact $n$-isometries.
We will only show that $\phi $ is an $n$-isometry, since the proof for $\psi 
$ is entirely analogous. Suppose that $x\in M_{n}(X)$. Pick an $n$%
-contraction $u:X\rightarrow M_{n}$ such that $\left\Vert x\right\Vert
=\left\Vert \left( id_{M_{n}}\otimes u\right) (x)\right\Vert $. By
hypothesis there are $n$-contractions $v:Y\rightarrow M_{n}$ and $w:\ell
^{1}(k)\rightarrow M_{n}$ such that $w(e_{i})=u(e_{i})-v(e_{i})$. Therefore
if $\alpha _{i}\in M_{n}$ then the norm of%
\begin{equation*}
\left( x-\sum_{i}\alpha _{i}\otimes a_{i},\sum_{i}\alpha _{i}\otimes
b_{i},\sum_{i}\alpha _{i}\otimes e_{i}\right)
\end{equation*}%
in $M_{n}(\widehat{Z})$ is bounded from below by the norm of%
\begin{equation*}
\left( id_{M_{n}}\otimes u\right) \left( x-\sum_{i}\alpha _{i}\otimes
a_{i}\right) +\left( id_{M_{n}}\otimes v\right) \left( \sum_{i}\alpha
_{i}\otimes b_{i}\right) +\left( id_{M_{n}}\otimes w\right) (\sum_{i}\alpha
_{i}\otimes e_{i})=\left( id_{M_{n}}\otimes u\right) (x)
\end{equation*}%
which equals $\left\Vert x\right\Vert $. Since this is true for every $%
\alpha _{i}\in M_{n}$, $\phi $ is an $n$-isometry. Similarly $\psi $ is an $%
n $-isometry. The proof is concluded by observing that $\left\Vert \phi
(a_{i})-\psi (b_{i})\right\Vert \leq 1$ for every $i\leq k$.\qedhere
\end{proof}

\begin{corollary}
\label{Corollary:bound norm}If $\bar{a},\bar{b}\in \mathcal{M}_{n}(k)$ and $%
d_{\mathcal{M}_{n}}(\bar{a},\bar{b})\leq M$ then for every $\alpha _{i}\in M$%
, the distance between $||\alpha _{1}\otimes a_{1}+\cdots +\alpha
_{n}\otimes a_{n}||$ and $||\alpha _{1}\otimes b_{1}+\cdots +\alpha
_{n}\otimes b_{n}||$ is at most $M{}||\alpha _{1}\otimes e_{1}+\cdots
+\alpha _{n}\otimes e_{n}||$, where $\bar{e}$ is the canonical basis of $%
\ell ^{1}(k)$.
\end{corollary}

An \emph{Auerbach system} in a Banach space is a basis $\bar{a}$ with dual
basis $\bar{a}^{\ast }$ such that $\left\Vert a_{i}\right\Vert =\left\Vert
a_{i}^{\ast }\right\Vert =1$. By analogy we say that an element $\bar{a}$ of 
$\mathcal{M}_{n}(k)$ is $N$\emph{-Auerbach} if $\left\Vert a_{i}\right\Vert
\leq N$ and $\left\Vert a_{i}^{\ast }\right\Vert \leq N$ for every $i\leq k$%
. Denote by $\mathcal{M}_{n}(k,N)$ the set of $N$-Auerbach $\bar{a}\in 
\mathcal{M}_{n}(k)$. It follows from Corollary \ref{Corollary:bound norm}
that the set $\mathcal{M}_{n}(k,N)$ is closed in $\mathcal{M}_{n}(k)$. It
can be easily verified that if $\bar{a}\in \mathcal{M}_{n}(k,N)$ and $\alpha
_{i}\in M_{n}$ then $||\alpha _{1}\otimes a_{1}+\cdots +\alpha _{n}\otimes
a_{n}||{}\leq N{}||\alpha _{1}\otimes e_{1}+\cdots +\alpha _{n}\otimes
e_{n}||$, and${}||\alpha _{1}\otimes e_{1}+\cdots +\alpha _{n}\otimes
e_{n}||{}\leq $ $kN{}||\alpha _{1}\otimes a_{1}+\cdots +\alpha _{n}\otimes
a_{n}||$, where $\bar{e}$ is the canonical basis of $\ell ^{1}(k)$.

If $\bar{a},\bar{b}\in \mathcal{M}_{n}(k)$, denote by $\iota _{\bar{a},\bar{b%
}}$ the linear isomorphism from $\left\langle \bar{a}\right\rangle $ to $%
\left\langle \bar{b}\right\rangle $ such that $\iota _{\bar{a},\bar{b}%
}(a_{i})=b_{i}$ for $i\leq k$. Define the\emph{\ }$n$\emph{-bounded distance}%
\ $d_{nb}(\bar{a},\bar{b})$ to be $||\iota _{\bar{a},\bar{b}}||_{n}\;||\iota
_{\bar{a},\bar{b}}^{-1}||_{n}$. (Observe that this is not an actual metric,
but $\log \left( d_{nb}\right) $ is a metric.)

In the following lemma we establish a precise relation between the $n$%
-bounded distance $d_{nb}$ and the Fra\"{\i}ss\'{e} metric $d_{nb}$ on $%
\mathcal{M}_{n}(k,N)$.

\begin{proposition}
\label{Proposition: comparison distances}Suppose that $\bar{a},\bar{b}\in 
\mathcal{M}_{n}(k,N)$. Then $d_{nb}(\bar{a},\bar{b})\leq \left( 1+kNd_{%
\mathcal{M}_{n}}(\bar{a},\bar{b})\right) ^{2}$ and $d_{\mathcal{M}_{n}}(\bar{%
a},\bar{b})\leq d_{nb}(\bar{a},\bar{b})-1$.
\end{proposition}

\begin{proof}
The first assertion can be easily deduced from Corollary \ref%
{Corollary:bound norm}, while the other inequality is an immediate
consequence of Lemma \ref{Lemma: approximate amalgamation}.
\end{proof}

We can finally show that the space $\left( \mathcal{M}_{n}(k),d_{\mathcal{M}%
_{n}}\right) $ is separable and complete and, in fact, compact. In view of
Proposition \ref{Proposition: comparison distances} this can be proved by a
standard argument; see for example \cite[Theorem 21.1 and Remark 21.2]%
{pisier_introduction_2003}. A proof is included for the sake of completeness.

\begin{proposition}
\label{Proposition:MnPolish}The space $\left( \mathcal{M}_{n}(k),d_{\mathcal{%
M}_{n}}\right) $ is compact.
\end{proposition}

\begin{proof}
Suppose that $\left( \bar{a}^{\left( m\right) }\right) _{m\in \mathbb{N}}$
is a sequence in $\mathcal{M}_{n}(k)$. If $\alpha _{i}\in M_{n}$ then the
sequence $||\alpha _{1}\otimes a_{1}^{(m)}+\cdots +\alpha _{n}\otimes
a_{n}^{(m)}||$ for $m\in \mathbb{N}$ is a bounded sequence of complex
numbers. Therefore after passing to a subsequence we can assume that such a
sequence converges for any choice of $\alpha _{i}\in M_{n}(\mathbb{Q}(i))$.
This is easily seen to imply that in fact such a sequence convergence for
any choice of $\alpha _{i}\in M_{n}$. Moreover the functions $\left( \alpha
_{1},\ldots ,\alpha _{k}\right) \mapsto ||\alpha _{1}\otimes
a_{1}^{(m)}+\cdots +\alpha _{n}\otimes a_{n}^{(m)}||$ are equiuniformly
continuous on the unit ball of $M_{n}$. Therefore, by the Ascoli-Arzel\'{a}
theorem, after passing to a further subsequence we can assume that the
convergence is uniform on the unit ball of $M_{n}$. We can now define an
element $\bar{a}$ of $\mathcal{M}_{n}(k)$ by setting $||\alpha _{1}\otimes
a_{1}+\cdots +\alpha _{n}\otimes a_{n}||$ to be the limit of $||\alpha
_{1}\otimes a_{1}^{(m)}+\cdots +\alpha _{n}\otimes a_{n}^{(m)}||$ for $%
m\rightarrow +\infty $. The abstract characterization of $M_{n}$-spaces
shows that $\bar{a}$ is indeed an element of $\mathcal{M}_{n}(k)$. By
uniform convergence in the unit ball the sequence $\left( \bar{a}^{\left(
m\right) }\right) _{m\in \mathbb{N}}$ is such that $d_{nb}(\bar{a}^{\left(
m\right) },\bar{a})\rightarrow 1$. Therefore $d_{\mathcal{M}_{n}}(\bar{a}%
^{\left( m\right) },\bar{a})\rightarrow 0$ by Proposition \ref{Proposition:
comparison distances}.
\end{proof}

This concludes the proof that $\mathcal{M}_{n}$ is a complete Fra\"{\i}ss%
\'{e} class.

\subsection{The Fra\"{\i}ss\'{e} limit\label{Subsection:LimitMn}}

We have verified that the class $\mathcal{M}_{n}$ is a Fra\"{\i}ss\'{e}
class in the sense of Definition \ref{Definition:Fraisseclass}. Therefore by
Theorem \ref{Theorem:Fraisselimit} we can consider its Fra\"{\i}ss\'{e}
limit. Observe that the $\mathcal{M}_{n}$-structures are precisely the $%
M_{n} $-spaces. We first provide a characterization of the Fra\"{\i}ss\'{e}
limit of $\mathcal{M}_{n}$ similar in spirit to the universal property
defining the Gurarij Banach space.

\begin{proposition}
\label{Proposition: characterization limit Mn}Suppose that $Z$ is a
separable $M_{n}$-space. The following statements are equivalent:

\begin{enumerate}
\item $Z$ is the Fra\"{\i}ss\'{e} limit of the class $\mathcal{M}_{n}$;

\item If $X\subset Y$ are finite-dimensional $M_{n}$-spaces, $\phi
:X\rightarrow Z$ is a linear $n$-isometry, and $\varepsilon >0$, then there
is a linear function $\widehat{\phi }:Y\rightarrow Z$ extending $\phi $ such
that $||\widehat{\phi }||_{n}\;||\widehat{\phi }^{-1}||_{n}<1+\varepsilon $.
\end{enumerate}
\end{proposition}

\begin{proof}
The proof is entirely analogous to the proof of \cite[Theorem 3.3]%
{ben_yaacov_fraisse_2012}, and is presented here for convenience of the
reader.

(1)$\Rightarrow $(2): Suppose that $Z$ is the Fra\"{\i}ss\'{e} limit of the
class $\mathcal{M}_{n}$. Suppose that $X\subset Y$ are finite-dimensional $%
M_{n}$-spaces, $\phi :X\rightarrow Z$ is a linear $n$-isometry, and $%
\varepsilon >0$. Fix $\delta >0$ small enough. Consider also a basis $\left(
a_{1},\ldots ,a_{k}\right) $ of $X$ and a basis $\left( b_{1},\ldots
,b_{m}\right) $ of $Y$ such that $b_{i}=a_{i}$ for $i\leq k$. Since $Z$ is
by assumption the Fra\"{\i}ss\'{e} limit of the class $\mathcal{M}_{n}$,
there is a linear $n$-isometry $\psi :Y\rightarrow Z$ such that $||\phi
(a_{i})-\psi (a_{i})||{}\leq \delta $ for every $i\leq k$. Define now $%
\widehat{\phi }:Y\rightarrow Z$ by setting $\widehat{\phi }(b_{i})=\phi
(a_{i})$ for $i\leq k$ and $\widehat{\phi }(b_{i})=\psi (b_{i})$ for $%
k<i\leq m$. A routine calculation shows that, for $\delta $ small enough, $%
\psi $ satisfies the desired inequality.

(2)$\Rightarrow $(1): Suppose now that $Z$ satisfies condition (2). Consider 
$X\in \mathcal{M}_{n}(k)$, a finite subset $B$ of $X$, $\psi \in \mathrm{Stx}%
_{\mathcal{M}_{n}}(X,Z)$, and $\varepsilon >0$. By \cite[Lemma 2.16]%
{ben_yaacov_fraisse_2012} in order to show that $Z$ is the Fra\"{\i}ss\'{e}
limit of $\mathcal{M}_{n}$ it is enough to find $\varphi \in \mathrm{\mathrm{%
Stx}}_{\mathcal{M}_{n}}^{<\psi }(X,Z)$ with the following property: for
every $b\in B$ there is $y\in Z$ such that $\varphi \left( b,y\right)
<\varepsilon $. By \cite[Lemma 2.8(iii)]{ben_yaacov_fraisse_2012} after
enlarging $X$, and decreasing $\varepsilon $ we can assume that there is a
finite subset $A$ of $X$ and an $n$-isometric linear map $f:\mathrm{span}%
\left( A\right) \rightarrow Z$ such that $\psi \geq f|_{A}+\varepsilon $.
(Recall our convention of identifying partial isomorphisms between $\mathcal{%
L}$-structures with the corresponding approximate isomorphisms.) Denote by $%
\bar{c}$ a linear basis of $\mathrm{span}\left( A\cup B\right) $. By
assumption if $\delta >0$, then we can extend $f$ to a linear map $%
f:X\rightarrow Z$ satisfying $||f||_{n}{}||f^{-1}||_{n}<1+\delta $. This
implies that $d_{\mathcal{M}_{n}}\left( \overline{c},f\left( \overline{c}%
\right) \right) <\delta $ by Proposition \ref{Proposition: comparison
distances}. Therefore there exist a finite-dimensional $M_{n}$-space $Y$ and
unital complete isometries $\phi _{0}:\left\langle \overline{c}\right\rangle
\rightarrow F$ and $\phi _{1}:\left\langle f\left( \overline{c}\right)
\right\rangle \rightarrow F$ such that $\left\Vert \left( \phi _{0}-\left(
\phi _{1}\circ f\right) \right) \left( c_{i}\right) \right\Vert \leq \delta $
for every $i$. Let $\phi :\left\langle \overline{c}\right\rangle \leadsto
\left\langle f\left( \overline{c}\right) \right\rangle $ be the composition
of $\phi _{1}^{-1}$ and $\phi _{0}$ as approximate morphisms. It is clear
that by choosing $\delta $ small enough we can ensure that $\phi \left(
x,f(x)\right) \leq \varepsilon $ for every $x\in A\cup B$. Observe finally
that $\phi $ is an approximate morphism from $E$ to $Z$ such that $\psi \geq
f|_{A}+\varepsilon \geq \phi $.
\end{proof}

In view of Proposition \ref{Proposition: characterization limit Mn}, items
(1),(2),(3) of Theorem \ref{Theorem:main-G_n} are an immediate consequence
of Theorem \ref{Theorem:Fraisselimit} and the fact that $\mathcal{M}_{n}$ is
a complete Fra\"{\i}ss\'{e} class.

\begin{remark}
\label{Remark:M_n}It follows from \cite[Lemma 3.17]{ben_yaacov_fraisse_2015}
and Lemma \ref{Lemma:approximation} that one can realize $\mathbb{G}_{n}$ as
the limit of an inductive sequence $\left( X_{k}\right) $ of $M_{n}$-spaces
with $n$-isometric connective maps $\phi _{k}:X_{k}\rightarrow X_{k+1}$ such
that $X_{k}$ is $n$-isometrically isomorphic to a finite $\infty $-sum of
copies of $M_{n}$.
\end{remark}

Clearly for $n=1$ one obtains a Banach space which is isometrically
isomorphic to the Gurarij space. One might wonder whether $\mathbb{G}_{n}$
is simply the tensor product of $M_{n}$ with the Gurarij Banach space $%
\mathbb{G}$. We will prove in \S \ref{Subsection:indecomposability} that
this is not the case.

\section{The noncommutative Gurarij space\label{Section:NG}}

\subsection{$\mathrm{MIN}$ and $\mathrm{MAX}$ spaces\label{Subsection:MINMAX}%
}

Clearly any operator space can be canonically regarded as an $M_{n}$-space.
Conversely if $X$ is an $M_{n}$-space, then there are two canonical ways to
regard $X$ as an operator space. It is natural to call an operator space
structure $\widehat{X}$ on $X$ \emph{compatible }if the map $X\mapsto 
\widehat{X}$ is an $n$-isometry. The \emph{minimal and maximal }compatible
operator space structures $\mathrm{MIN}_{n}(X)$ and $\mathrm{MAX}_{n}(X)$ on
an $M_{n}$-space are defined by letting $\left\Vert x\right\Vert
_{M_{k}\left( \mathrm{MIN}_{n}(X)\right) }$ to be the supremum of $%
\left\Vert \left( id_{M_{k}}\otimes \phi \right) (x)\right\Vert
_{M_{k}\otimes M_{n}}$, where $\phi $ varies among all $n$-contractions from 
$X$ to $M_{n}$. Similarly, $\left\Vert x\right\Vert _{M_{k}\left( \mathrm{%
\mathrm{\mathrm{MAX}}}_{n}(X)\right) }$ is defined to be the supremum of $%
\left\Vert \left( id_{M_{k}}\otimes u\right) (x)\right\Vert _{M_{k}\left(
B(H)\right) }$, where $u$ varies among all $n$-contractions from $X$ to $%
B(H) $. These are introduced in \cite[Section I.3]{lehner_mn-espaces_1997}
as a generalization of the minimal and maximal quantization of a Banach
space as in \cite[Section 3.3]{effros_operator_2000}; see also \cite[Section
2]{oikhberg_operator_2004}. If $X$ is an operator space then we define $%
\mathrm{MIN}_{n}(X)$ and $\mathrm{MAX}_{n}(X)$ to be the structures defined
above starting from $X$ regarded just as $M_{n}$-space. This is consistent
with the terminology used in \cite%
{oikhberg_operator_2004,oikhberg_non-commutative_2006}.

The names $\mathrm{MIN}$ and $\mathrm{MAX}$ are suggestive of the following
property; see \cite[Proposition I.3.1]{lehner_mn-espaces_1997}. If $\widehat{%
X}$ is a compatible operator space structure on $X$ then the identity maps $%
\mathrm{\mathrm{MAX}}_{n}(X)\rightarrow \widehat{X}\rightarrow \mathrm{MIN}%
_{n}(X)$ are completely bounded. The operator space structures $\mathrm{MIN}%
_{n}$ and $\mathrm{MAX}_{n}$ are characterized by the following universal
property; see \cite[Proposition I.3.6 and Proposition I.3.7]%
{lehner_mn-espaces_1997}. If $Z$ is an operator space and $u:Z\rightarrow X$
is a linear map, then $u:Z\rightarrow X$ is $n$-bounded if and only if $%
u:Z\rightarrow \mathrm{MIN}_{n}(X)$ is completely bounded, and in such case $%
\left\Vert u:Z\rightarrow \mathrm{MIN}_{n}(X)\right\Vert _{cb}=\left\Vert
u:Z\rightarrow X\right\Vert _{n}$. Similarly if $Z$ is an operator space and 
$u:X\rightarrow Z$ is a linear map, then $u:X\rightarrow Z$ is $n$-bounded
if and only if $u:\mathrm{MAX}_{n}(X)\rightarrow Z$ is completely bounded,
and in such case $\left\Vert u:\mathrm{MAX}(X)\rightarrow Z\right\Vert
_{cb}=\left\Vert u:X\rightarrow Z\right\Vert _{n}$

\begin{remark}
\label{Remark:Mn as OS}In the following we will always consider an $M_{n}$%
-space $X$ as an operator spaces endowed with its minimal compatible
operator space structure.
\end{remark}

It is worth noting at this point that all the proofs of Section \ref%
{Section:LimitMn} go through without change when $M_{n}$-spaces are regarded
as operator spaces with their minimal compatible operator space structure.
This easily follows from the properties of the minimal quantization recalled
above.

\subsection{Exact and $1$-exact operator spaces\label{Subsection:1exact}}

Suppose that $E$ and $F$ are two finite-dimensional operator spaces. Define%
\emph{\ }$d_{cb}(E,F)$ to be the infimum of $||\phi ||_{cb}\;||\phi
^{-1}||_{cb}$ when $\phi $ ranges over all linear isomorphisms from $E$ to $%
F $. The \emph{exactness constant }$ex\left( E\right) $ of a
finite-dimensional operator space is the infimum of $d_{cb}\left( E,F\right) 
$ where $F$ ranges among all subspaces of $M_{n}$ for $n\in \mathbb{N}$.
Equivalently one can define $ex(E)$ to be the limit for $n\rightarrow
+\infty $ of the decreasing sequence $\left\Vert id_{E}:\mathrm{MIN}%
_{n}(E)\rightarrow E\right\Vert _{cb}$, where $id_{E}$ denotes the identity
map of $E$. If $X$ is a not necessarily finite-dimensional operator space,
then its exactness constant $ex\left( X\right) $ is the supremum of $%
ex\left( E\right) $ where $E$ ranges over all finite-dimensional subspaces
of $E$.

An operator space is \emph{exact }if it has finite exactness constant, and $%
1 $-\emph{exact }if it has exactness constant $1$. For C*-algebras exactness
is equivalent to $1$-exactness, which is in turn equivalent to several other
properties; see \cite[Section IV.3.4]{blackadar_operator_2006}. Exactness is
a fundamental notion in the theory of C*-algebras and operator spaces. It is
a purely noncommutative phenomenon: there is no Banach space analog of
nonexactness. In fact every Banach space---and in fact every $M_{n}$%
-space---is $1$-exact. More information and several equivalent
characterizations of exactness can be found in \cite{pisier_exact_1995} and 
\cite[Chapter 17]{pisier_introduction_2003}.

In the following we will denote by $\mathcal{E}_{1}$ the class of
finite-dimensional $1$-exact operator spaces. Moreover we will denote by $%
\mathcal{M}_{\infty }^{0}\subset \mathcal{E}_{1}$ the class of operator
spaces that admit a completely isometric embedding into $M_{n}$ for some $%
n\in \mathbb{N}$. Our goal is to show that $\mathcal{E}_{1}$ is a Fra\"{\i}ss%
\'{e} class.

\subsection{Amalgamation of $1$-exact operator spaces\label%
{Subsection:completionMinfty}}

It is clear that $\mathcal{E}_{1}$ has (HP) from Definition \ref{Definition:
NAP}. It remains to verify that $\mathcal{E}_{1}$ satisfies (AP). This will
give (JEP) as consequence, since the trivial operator space $\left\{
0\right\} $ embeds in every element of $\mathcal{E}_{1}$.

We recall that if $\left( Z_{n}\right) $ is a direct sequence of operator
spaces with completely isometric linear maps $\phi _{n}:Z_{n}\rightarrow
Z_{n+1}$ one can define the \emph{direct limit }$\lim_{\left( \phi
_{n}\right) }Z_{n}$ with canonical completely isometric linear maps $\sigma
_{k}:Z_{k}\rightarrow \lim_{\left( \phi _{n}\right) }Z_{n}$ in the following
way. Let $\ell ^{\infty }\left( \mathbb{N},\left( Z_{n}\right) \right) $ be
the space of sequences $(z_{n})\in \prod_{n}Z_{n}$ with $\sup_{n}\left\Vert
z_{n}\right\Vert <+\infty $. Define an operator seminorm structure on $\ell
^{\infty }\left( \mathbb{N},\left( Z_{n}\right) \right) $ in the sense of 
\cite[1.2.16]{blecher_operator_2004} by setting $\rho _{k}\left(
(z_{n})_{n\in \mathbb{N}}\right) =\limsup_{n\rightarrow +\infty }\left\Vert
z_{n}\right\Vert _{M_{k}\left( Z_{n}\right) }$ for $k\in \mathbb{N}$ and $%
z_{n}\in M_{k}\left( Z_{n}\right) $. Finally define $W$ to be the operator
space associated with such an operator seminorm structure on $\ell ^{\infty
}\left( \mathbb{N},\left( Z_{n}\right) \right) $. For $n,m$ let $\phi
_{n,n}=id_{Z_{n}}$, $\phi _{n,m}=\phi _{m-1}\circ \cdots \circ \phi _{n}$ if 
$n<m$, and $\phi _{n,m}=0$ otherwise. Define the maps $\sigma
_{k}:Z_{k}\rightarrow W$ by $\sigma _{k}(x)=\left( \phi _{k,n}(x)\right)
_{n\in \mathbb{N}}$. Finally set $\lim_{\left( \phi _{n}\right) }Z_{n}$ to
be the closure inside $W$ of the union of the images of $Z_{k}$ under $%
\sigma _{k}$ for $k\in \mathbb{N}$. It is clear that if for every $k\in 
\mathbb{N}$ the space $Z_{k}$ is $1$-exact, then $\lim_{(\phi _{n})}Z_{n}$
is $1$-exact.

The proof of the following proposition is inspired by \cite[Theorem 4.7]%
{effros_injectivity_2001} and \cite[Theorem 1.1]%
{oikhberg_non-commutative_2006}.

\begin{proposition}
\label{Proposition: approximate amalgamation 1-exact}Suppose that $%
X_{0}\subset X$ and $Y$ are finite-dimensional $1$-exact operator spaces, $%
\delta >0$, and $f:X_{0}\rightarrow Y$ is such that $\left\Vert f\right\Vert
_{cb}<1+\delta $ and $\left\Vert f^{-1}\right\Vert _{cb}<1+\delta $. Then
there exists a $1$-exact separable operator space $Z$ and linear complete
isometries $j:Y\rightarrow Z$ and $i:X\rightarrow Z$ such that $\left\Vert
j\circ f-i|_{X_{0}}\right\Vert _{cb}<\delta $.
\end{proposition}

\begin{proof}
Fix $\delta ^{\prime }<\delta $ such that $\left\Vert f\right\Vert
_{cb}<1+\delta ^{\prime }$ and $\left\Vert f^{-1}\right\Vert _{cb}<1+\delta
^{\prime }$ and $\varepsilon >0$ such that $\delta ^{\prime }+2\varepsilon
<\delta $. We will construct by recursion on $k$ sequences $\left(
n_{k}\right) _{k\in \mathbb{N}}$, $\left( Z_{k}\right) _{k\in \mathbb{N}}$, $%
i_{k}:X\rightarrow Z_{k}$, $j_{k}:Y\rightarrow Z_{k}$, $\phi
_{k}:Z_{k}\rightarrow Z_{k+1}$ such that

\begin{enumerate}
\item $\left( n_{k}\right) _{k\in \mathbb{N}}$ is nondecreasing,

\item $Z_{k}$ is an $M_{n_{k}}$-space,

\item $i_{k}$ and $j_{k}$ are injective completely contractive linear maps,

\item $\phi _{k}$ is a completely isometric linear map,

\item $\left\Vert i_{k}^{-1}\right\Vert _{cb}\leq 1+\varepsilon 2^{-k}$, $%
\left\Vert j_{k}^{-1}\right\Vert _{cb}\leq 1+\varepsilon 2^{-k}$,

\item $\left\Vert \phi _{k}\circ i_{k}-i_{k+1}\right\Vert _{cb}\leq
1+\varepsilon 2^{-k}$, $\left\Vert \phi _{k}\circ j_{k}-j_{k+1}\right\Vert
_{cb}\leq 1+\varepsilon 2^{-k}$, and

\item $\left\Vert j_{k}\circ f-\left( i_{k}\right) |_{X_{0}}\right\Vert
_{cb}\leq \delta ^{\prime }+\varepsilon \sum_{i<k}2^{-i+1}$.
\end{enumerate}

We can apply Lemma \ref{Lemma: approximate amalgamation} and Lemma \ref%
{Lemma:approximation} to define $n_{1},Z_{1},i_{1}$, and $j_{1}$. Suppose
that $n_{k}$, $Z_{k}$, $i_{k}$, $j_{k}$, and $\phi _{k-1}$ have been defined
for $k\leq m$. By Lemma \ref{Lemma:approximation} we can pick $n_{m+1}\geq
n_{m}$ and injective completely contractive maps $\theta _{X}:X\rightarrow
M_{n_{m+1}}$ and $\theta _{Y}:Y\rightarrow M_{n_{m+1}}$ such that $||\theta
_{X}^{-1}||_{cb}\leq 1+\varepsilon 2^{-2\left( m+1\right) }$ and $||\theta
_{Y}^{-1}||_{cb}\leq 1+\varepsilon 2^{-\left( m+1\right) }$. By injectivity
of $M_{n_{m+1}}$ there are complete contractions $\alpha _{X},\alpha
_{Y}:Z_{m}\rightarrow M_{n_{m+1}}$ such that $\alpha _{X}\circ i_{m}=\frac{1%
}{1+\varepsilon 2^{-m}}\theta _{X}$ and $\alpha _{Y}\circ j_{m}=\frac{1}{%
1+\varepsilon 2^{-m}}\theta _{Y}$. Define $W$ to be $\mathrm{MIN}%
_{n_{m+1}}(Z_{m}\oplus ^{\infty }M_{n_{m+1}})$. Define linear maps $\widehat{%
\theta }_{X}:X\rightarrow W$, $x\mapsto \left( i_{m}(x),\theta
_{X}(x)\right) $, $\widehat{\theta }_{Y}:Y\rightarrow W$, $y\mapsto \left(
j_{m}(y),\theta _{Y}(y)\right) $, $\widehat{\alpha }_{X}:Z_{m}\rightarrow W$%
, $z\mapsto \left( z,\alpha _{X}(z)\right) $, and $\widehat{\alpha }%
_{Y}:Z_{m}\rightarrow W$, $z\mapsto \left( z,\alpha _{Y}(z)\right) $.
Observe that $\widehat{\alpha }_{X},\widehat{\alpha }_{Y}$ are completely
isometric, while $\widehat{\theta }_{X}$ and $\widehat{\theta }_{Y}$ are
completely contractive with $||\widehat{\theta }_{X}^{-1}||_{cb}\leq
||\theta _{X}^{-1}||_{cb}\leq 1+\varepsilon 2^{-\left( m+1\right) }$ and $||%
\widehat{\theta }_{Y}^{-1}||_{cb}\leq ||\theta _{Y}^{-1}||_{cb}\leq
1+\varepsilon 2^{-\left( m+1\right) }$. Note also that $||\widehat{\theta }%
_{X}-\widehat{\alpha }_{X}\circ i_{m}||_{cb}\leq \varepsilon 2^{-m}$ and $||%
\widehat{\theta }_{Y}-\widehat{\alpha }_{Y}\circ i_{m}||_{cb}\leq
\varepsilon 2^{-m}$. Define now%
\begin{equation*}
N=\left\{ \left( -\left( z_{0}+z_{1}\right) ,\widehat{\alpha }_{X}(z_{0}),%
\widehat{\alpha }_{Y}(z_{1})\right) \in Z_{m}\oplus W\oplus W:z_{0},z_{1}\in
Z_{m}\right\} \text{.}
\end{equation*}%
Let $Z_{m+1}$ be $\mathrm{MIN}_{n_{m+1}}((Z_{m}\oplus ^{1}W\oplus ^{1}W)/N)$%
. Consider the first coordinate inclusion $\phi _{m}:Z_{m}\rightarrow
Z_{m+1} $ of $Z_{m}$ into $Z_{m+1}$. Similarly define $\psi _{X},\psi
_{Y}:W\rightarrow Z_{m+1}$ to be the second and third coordinate inclusions.
Arguing as in the proof of Lemma \ref{Lemma: approximate amalgamation} one
can verify directly that $\phi _{m},\psi _{X},\psi _{Y}$ are complete
isometries. Alternatively one can use \cite[Lemma 2.4]%
{oikhberg_non-commutative_2006} together with the properties of $\mathrm{MIN}
$. Observe that $\phi _{m}=\psi _{X}\circ \widehat{\alpha }_{X}=\psi
_{Y}\circ \widehat{\alpha }_{Y}$. Define now linear complete contractions $%
i_{m+1}:=\psi _{X}\circ \widehat{\theta }_{X}:X\rightarrow Z_{m+1}$ and $%
j_{m+1}:=\psi _{Y}\circ \widehat{\theta }_{Y}:Y\rightarrow Z_{m+1}$.Observe
that $\left\Vert i_{m+1}^{-1}\right\Vert _{cb}\leq ||\widehat{\theta }%
_{X}^{-1}||_{cb}<1+\varepsilon 2^{-\left( m+1\right) }$ and $\left\Vert
j_{m+1}^{-1}\right\Vert _{cb}\leq ||\widehat{\theta }_{Y}^{-1}||_{cb}<1+%
\varepsilon 2^{-\left( m+1\right) }$. Moreover%
\begin{equation*}
\left\Vert \phi _{m}\circ i_{m}-i_{m+1}\right\Vert _{cb}=\left\Vert \phi
_{m}\circ i_{m}-\psi _{X}\circ \widehat{\theta }_{X}\right\Vert _{cb}\leq
\left\Vert \phi _{m}\circ i_{m}-\psi _{X}\circ \widehat{\alpha }_{X}\circ
i_{m}\right\Vert _{cb}+\varepsilon 2^{-m}=\varepsilon 2^{-m}\text{.}
\end{equation*}%
Similarly, $\left\Vert \psi _{m}\circ j_{m}-j_{m+1}\right\Vert _{cb}\leq
\varepsilon 2^{-m}$. Finally we have%
\begin{eqnarray*}
\left\Vert i_{m+1}-j_{m+1}\circ f\right\Vert _{cb} &=&\left\Vert \psi
_{X}\circ \widehat{\theta }_{X}-\psi _{Y}\circ \widehat{\theta }_{Y}\circ
f\right\Vert _{cb}\leq \left\Vert \psi _{X}\circ \widehat{\alpha }_{X}\circ
i_{m}-\psi _{Y}\circ \widehat{\alpha }_{Y}\circ j_{m}\circ f\right\Vert
_{cb}+\varepsilon 2^{-m+1} \\
&\leq &\left\Vert \phi _{m}\circ i_{m}-\phi _{m}\circ j_{m}\circ
f\right\Vert _{cb}+\varepsilon 2^{-m+1}\leq \left\Vert i_{m}-j_{m}\circ
f\right\Vert _{cb}+\varepsilon 2^{-m+1}\leq \delta ^{\prime }+\varepsilon
\sum_{i\leq m}2^{-i+1}\text{.}
\end{eqnarray*}%
This concludes the recursive construction. Let now $Z$ be $\lim_{\left( \phi
_{k}\right) }Z_{k}$ with canonical linear complete isometries $\sigma
_{k}:Z_{k}\rightarrow Z$. Consider also the embeddings $i:X\rightarrow Z$
and $j:Y\rightarrow Z$ defined by $i:=\lim_{k\rightarrow +\infty }\sigma
_{k}\circ i_{k}$ and $j:=\lim_{k\rightarrow +\infty }\sigma _{k}\circ j_{k}$%
.It is easily seen as in the proof of \cite[Theorem 4.7]%
{effros_injectivity_2001} that $Z$ is a $1$-exact separable operator space,
and $i,j$ are well defined completely isometric linear maps such that $%
\left\Vert j\circ f-i|_{X_{0}}\right\Vert _{cb}\leq \delta +2\varepsilon $.
\end{proof}

In particular Proposition \ref{Proposition: approximate amalgamation 1-exact}
for $\delta =0$ shows that the class $\mathcal{E}_{1}$ has (NAP). It is not
difficult to modify the proof above to show that the conclusions of
Proposition \ref{Proposition: approximate amalgamation 1-exact} hold even
when $X$ and $Y$ are (not necessarily finite-dimensional) separable $1$%
-exact operator spaces.

\begin{remark}
\label{Remark:existence}The existence of the noncommutative Gurarij $\mathbb{%
NG}$ can be deduced by a repeated application of Proposition \ref%
{Proposition: approximate amalgamation 1-exact} to the class of
finite-dimensional operator spaces that embed completely isometrically into $%
M_{n}$ for some $n\in \mathbb{N}$.
\end{remark}

The following lemma can be easily obtained from Proposition \ref%
{Proposition: approximate amalgamation 1-exact}, similarly as Lemma 2.2 is
derived from Lemma 2.1 in \cite{kubis_proof_2013}.

\begin{lemma}
\label{Lemma:put-back}If $Z$ is a Gurarij operator space, $E\subset Z$ is
finite-dimensional, $Y$ is finite-dimensional and $1$-exact, and $%
f:E\rightarrow F$ is an invertible linear map such that $||f||_{cb}<1+\delta 
$ and $||f^{-1}||_{cb}<1+\delta $ then for every $\varepsilon >0$ there
exists $g:F\rightarrow Z$ such that $||g||_{cb}<1+\varepsilon $, $%
||g^{-1}||_{cb}<1+\varepsilon $, and $||g\circ f-id_{E}||_{cb}<\delta $.
\end{lemma}

\subsection{The Fra\"{\i}ss\'{e} metric space}

Fix $k\in \mathbb{N}$ and denote by $\mathcal{E}_{1}(k)$ the space of pairs $%
\left( \bar{a},X\right) $ such that $X$ is a $k$-dimensional $1$-exact
operator space and $\bar{a}$ is a basis of $X$. Two such pairs $\left( \bar{a%
},X\right) $ and $\left( \bar{b},Y\right) $ are identified if there is a
complete isometry $\phi $ from $X$ to $Y$ such that $\phi (\bar{a})=\bar{b}$%
. To simplify the notation the pair $\left( \bar{a},X\right) $ will be
simply denoted $\bar{a}$, where we set $X=\left\langle \bar{a}\right\rangle $%
. Denote by $d_{\mathcal{E}_{1}}$ the Fra\"{\i}ss\'{e} metric on $\mathcal{E}%
_{1}(k)$ as in Definition \ref{Definition:FraisseMetric}. We further denote
by $\mathcal{M}_{\infty }^{0}(k)$ the subspace of $\mathcal{E}_{1}(k)$
consisting of pairs $\left( \bar{a},X\right) $ such that $X$ admits a
completely isometric embedding into $M_{n}$ for some $n\in \mathbb{N}$. As
in Subsection \ref{Subsection:PolishMn} we define an element $\bar{a}$ of $%
\mathcal{E}_{1}$ to be $N$\emph{-Auerbach} if $\left\Vert a_{i}\right\Vert
\leq N$ and $\left\Vert a_{i}^{\ast }\right\Vert \leq N$ for every $i\leq k$%
, where $\bar{a}^{\ast }$ denotes the dual basis of $\bar{a}$. We denote by $%
\mathcal{E}_{1}(k,N)$ the set of $N$-Auerbach $\bar{a}\in \mathcal{E}_{1}(k)$%
. Observe that the set $\mathcal{E}_{1}(k,N)$ is closed in $\mathcal{E}%
_{1}(k)$. If $\bar{a},\bar{b}\in \mathcal{E}_{1}(k)$, denote as in
Subsection \ref{Subsection:PolishMn} by $\iota _{\bar{a},\bar{b}}$ the
linear isomorphism from $\left\langle \bar{a}\right\rangle $ to $%
\left\langle \bar{b}\right\rangle $ such that $\iota _{\bar{a},\bar{b}%
}(a_{i})=b_{i}$ for $i\leq k$. Define the completely bounded distance $%
d_{cb}(\bar{a},\bar{b})$ to be $||\iota _{\bar{a},\bar{b}}||_{cb}\;||\iota _{%
\bar{a},\bar{b}}^{-1}||_{cb}$.

\begin{proposition}
\label{Proposition: comparison distances E1}Suppose that $\bar{a},\bar{b}\in 
\mathcal{E}_{1}(k,N)$. Then $d_{\mathcal{E}_{1}}(\bar{a},\bar{b})\leq d_{cb}(%
\bar{a},\bar{b})-1$, and $d_{cb}(\bar{a},\bar{b})\leq \frac{1+kNd_{\mathcal{E%
}_{1}}(\bar{a},\bar{b})}{1-kNd_{\mathcal{E}_{1}}(\bar{a},\bar{b})}$.
\end{proposition}

\begin{proof}
The first inequality follows from Proposition \ref{Proposition: approximate
amalgamation 1-exact}. Suppose now that $d_{\mathcal{E}_{1}}(\bar{a},\bar{b}%
)<\delta $ for some $\delta >0$. Then one can realize $\left\langle \bar{a}%
\right\rangle ,\left\langle \bar{b}\right\rangle $ as subspaces of a $1$%
-exact operator space $Z$ in such a way that $\left\Vert
a_{i}-b_{i}\right\Vert <\delta $ for every $i\leq k$. The conclusion then
follows from the small perturbation lemma \cite[Lemma 2.13.1]%
{pisier_introduction_2003}.
\end{proof}

Using Proposition \ref{Proposition: comparison distances E1} one can show
that $\left( \mathcal{E}_{1}(k),d_{\mathcal{E}_{1}}\right) $ is a separable
metric space. Recall that $\mathcal{M}_{\infty }^{0}\subset \mathcal{E}_{1}$
denotes the class of operator spaces that admit a completely isometric
embedding into $M_{n}$ for some $n\in \mathbb{N}$.

\begin{proposition}
\label{Proposition:FraissespaceE1}For every $k\in \mathbb{N}$, $\left( 
\mathcal{E}_{1}(k),d_{\mathcal{E}_{1}}\right) $ is a complete metric space,
and $\mathcal{M}_{\infty }^{0}(k)$ is a dense subset of $\mathcal{E}_{1}(k)$.
\end{proposition}

\begin{proof}
Suppose that $(\bar{a}^{\left( m\right) })_{m\in \mathbb{N}}$ is a Cauchy
sequence in $\mathcal{E}_{1}(k)$. Applying Proposition \ref{Proposition:
approximate amalgamation 1-exact} one can construct a sequence $\left(
X_{n}\right) $ of separable $1$-exact operator spaces such that $%
X_{n}\subset X_{n+1}$ for every $n\in \mathbb{N}$ and $\left\langle \bar{a}%
^{\left( m\right) }\right\rangle \subset X_{n}$ for $m\leq n$ in such a way
that $\max_{i}||a_{i}^{(m_{0})}-a_{i}^{(m_{1})}||{}\leq d_{\mathcal{E}_{1}}(%
\bar{a}^{\left( m_{0}\right) },\bar{a}^{\left( m_{1}\right) })$ for $%
m_{0},m_{1}\leq n$. One can then let $X$ be the inductive limit of the
sequence $(X_{n})$, and $\bar{a}_{i}$ for $i\leq k$ be the limit of the
sequence $(\bar{a}_{i}^{(m)})_{i\in \mathbb{N}}$ in $X$. This defines an
element $\bar{a}$ of $\mathcal{E}_{1}\left( k\right) $ that is the limit of $%
(\bar{a}^{\left( m\right) })_{m\in \mathbb{N}}$. Observe now that by
Proposition \ref{Proposition: comparison distances E1} and by the definition
of $1$-exact operator space, $\mathcal{M}_{\infty }^{0}(k)$ is dense in $(%
\mathcal{E}_{1}(k),d_{\mathcal{E}_{1}})$. It follows from Proposition \ref%
{Proposition:MnPolish} that $\mathcal{M}_{\infty }^{0}(k)$ is a separable
subspace of $\mathcal{E}_{1}(k)$.
\end{proof}

\subsection{The noncommutative Gurarij space as a Fra\"{\i}ss\'{e} limit 
\label{Subsection: NG as limit}}

We have shown in Subsection \ref{Subsection:completionMinfty} that $\mathcal{%
E}_{1}$ is a Fra\"{\i}ss\'{e} class. Therefore we can consider its
corresponding Fra\"{\i}ss\'{e} limit according to Theorem \ref%
{Theorem:Fraisselimit}. The following definition has been suggested in \cite%
{oikhberg_non-commutative_2006}.

\begin{definition}
\label{Definition:NG}A separable $1$-exact operator space $Z$ is \emph{%
(noncommutative) Gurarij} if for any finite-dimensional $1$-exact operator
spaces $X\subset Y$, complete isometry $\phi :X\rightarrow Z$, and $%
\varepsilon >0$, there is an injective linear map $\psi :Y\rightarrow Z$
extending $\phi $ such that $||\psi ||_{cb}\;||\psi
^{-1}||_{cb}<1+\varepsilon $.
\end{definition}

\begin{proposition}
\label{Proposition: characterization limit E1}Suppose that $Z$ is a
separable $1$-exact. The following conditions are equivalent:

\begin{enumerate}
\item $Z$ is a Fra\"{\i}ss\'{e} limit of the class $\mathcal{E}_{1}$;

\item For every $n\in \mathbb{N}$, subspace $X$ of $M_{n}$, complete
isometry $\phi :X\rightarrow Z$, and $\varepsilon >0$ there is a complete
isometry $\psi :M_{n}\rightarrow Z$ such that $\left\Vert \psi |_{X}-\phi
\right\Vert <\varepsilon $;

\item $Z$ is Gurarij;

\item $Z$ satisfies the same property as in Definition \ref{Definition:NG}
when $Y=M_{n}$ for some $n\in \mathbb{N}$.
\end{enumerate}
\end{proposition}

\begin{proof}
The implications (1)$\Rightarrow $(2) and (3)$\Rightarrow $(4) are obvious.
The implications (2)$\Rightarrow $(4) and (1)$\Rightarrow $(3) can be proved
using Proposition \ref{Proposition: comparison distances E1} similarly as
the implication (1)$\Rightarrow $(2) in Proposition \ref{Proposition:
characterization limit Mn}. The implication (4)$\Rightarrow $(1) can be
proved as (2)$\Rightarrow $(1) of Proposition \ref{Proposition:
characterization limit Mn}, or \cite[Theorem 3.3]{ben_yaacov_fraisse_2012},
with the extra ingredient of \cite[Lemma 2.16]{ben_yaacov_fraisse_2012} and
the fact that $\mathcal{M}_{\infty }^{0}(k)$ is dense in $\left( \mathcal{E}%
_{1}(k),d_{\mathcal{E}_{1}}\right) $.
\end{proof}

With such a characterization of a limit of the Fra\"{\i}ss\'{e} class $%
\mathcal{E}_{1}$ at hand, one can prove the existence, uniqueness,
homogeneity, and universality properties of $\mathbb{NG}$ as stated in
Theorem \ref{Theorem:main-NG} (1),(2),(3). Indeed the existence, uniqueness,
and universality statements follow from Theorem \ref{Theorem:Fraisselimit}
and Proposition \ref{Proposition: characterization limit E1}. The
homogeneity property of $\mathbb{NG}$ can be deduced from Lemma \ref%
{Lemma:put-back} via an intertwining argument, similarly as Theorem 1.1 is
deduced from Lemma 2.1 in \cite{kubis_proof_2013}.

\begin{remark}
\label{Remark:limit-NG}It follows from the proof of \cite[Lemma 3.17]%
{ben_yaacov_fraisse_2015} and density of $\mathcal{M}_{\infty }^{0}(k)$
inside $\mathcal{E}_{1}\left( k\right) $ that $\mathbb{NG}$ can be realized
as the limit of an inductive sequence $\left( X_{n}\right) $ of operator
spaces with completely isometric connective maps $\phi _{n}:X_{n}\rightarrow
X_{n+1}$ such that $X_{n}$ is completely isometric to $M_{d_{n}}$ for some $%
d_{n}\in \mathbb{N}$. This shows that $\mathbb{NG}$ is a \emph{rigid} $%
\mathcal{OL}_{\infty ,1+}$ as introduced in \cite%
{effros_OL_1998,junge_OL_2003}. In particular, it is \emph{nuclear} and it
has the \emph{weak expectation property} \cite[Theorem 4.5]%
{effros_injectivity_2001}.
\end{remark}

\begin{remark}
\label{Remark:limit-Gn}It also follows from the proof of \cite[Lemma 3.17]%
{ben_yaacov_fraisse_2015} and the fact that, for every $n,k\in \mathbb{N}$, $%
\mathcal{M}_{n}(k)\subset \mathcal{E}_{1}(k)$ and $\bigcup_{n}\mathcal{M}%
_{n}(k)$ is dense in $\mathcal{E}_{1}(k)$, that one can realize $\mathcal{E}%
_{1}(k)$ as the limit of an inductive sequence $\left( X_{n}\right) $ of
operator spaces with completely isometric connective maps $\phi
_{n}:X_{n}\rightarrow X_{n+1}$ such that $X_{n}$ is completely isometric to $%
\mathbb{G}_{n}$ endowed with its canonical operator space structure. In this
sense $\mathbb{NG}$ can be regarded as the limit of the spaces $\mathbb{G}%
_{n}$ for $n\in \mathbb{N}$.
\end{remark}

One can regard $\mathbb{NG}$ as the operator space analog of the Cuntz
algebra $\mathcal{O}_{2}$. Indeed, it is observed in \cite%
{eagle_fraisse_2014} that $\mathcal{O}_{2}$ is the Fraisse limit of the
class of finitely generated exact C*-algebras. This is indeed an immediate
consequence of Kirchberg's exact embedding theorem \cite[Theorem 2.8]%
{kirchberg_embedding_2000} asserting that every separable exact C*-algebra
admits an embedding into $\mathcal{O}_{2}$, and any two such embeddings are
approximately unitarily conjugate. The universality and homogeneity
properties of $\mathbb{NG}$ can be seen as the natural operator space
analogs of these properties of $\mathcal{O}_{2}$. Kirchberg's nuclear
embedding theorem asserts furthermore that a separable C*-algebra is nuclear
if and only if it is *-isomorphic to a subalgebra of $\mathcal{O}_{2}$ that
is the range of a conditional expectation of $\mathcal{O}_{2}$ \cite[Theorem
6.3.12]{rordam_classification_2002}. The operator space version of this fact
has been proved in \cite{lupini_operator_2015}: a separable operator space
is nuclear if and only if it is completely isometrically isomorphic to the
range of a completely contractive projection of $\mathbb{NG}$.

Since a C*-algebra is exact if and only if it is $1$-exact as an operator
space \cite[Corollary 17.5]{pisier_introduction_2003}, it follows that $%
\mathbb{NG}$ contains a completely isometric copy of any separable exact
C*-algebra. Particularly, $\mathbb{NG}$ contains a completely isometric copy
of $\mathcal{O}_{2}$. However, $\mathbb{NG}$ is \textquotedblleft
larger\textquotedblright\ than $\mathcal{O}_{2}$, as we will prove in \S \ref%
{Subsection:triple-NG}: $\mathbb{NG}$ does not embed completely
isometrically into $\mathcal{O}_{2}$. In fact, $\mathbb{NG}$ does not admit
any completely isometric embedding into any exact C*-algebra. This seems to
be the first example of a separable nuclear operator space that does embed
into an exact C*-algebra. An example of a separable nuclear operator system
that does not admit any \emph{unital }completely isometric embedding into a
unital exact C*-algebra has been constructed in \cite%
{kirchberg_c*-algebras_1998}.

Since $\mathbb{NG}$ is itself $1$-exact, one can conclude that $\mathbb{NG}$
is not completely isometric to a C*-algebra or to a ternary ring of
operators. This answers a question of Oikhberg from \cite%
{oikhberg_non-commutative_2006}. An explicit example of a separable nuclear
operator system that is not \emph{unitally} completely isometrically
isomorphic to a unital C*-algebra is provided in \cite%
{han_approximation_2011}.

\section{Further properties of the noncommutative Gurarij space\label%
{Section:properties}}

\subsection{The triple envelope of the Gurarij operator space\label%
{Subsection:triple-NG}}

A \emph{ternary ring of operators }(TRO) is a subspace $V$ of $B\left(
H,K\right) $ for some Hilbert spaces $H,K$ such that $xy^{\ast }z\in V$ for
any $x,y,z\in V$. The operation $\left( x,y,z\right) \mapsto xy^{\ast }z$ on 
$V$ is called \emph{triple product}. A \emph{triple morphism} between TROs
is a linear map that preserves the triple product. Observe that (the
restriction of) a $\ast $-homomorphism is in particular a triple morphism. A
TRO has a canonical operator space structure, where the matrix norms are
uniquely determined by the triple product \cite[Proposition 2.1]%
{kaur_local_2002}. A triple morphism between TROs is automatically
completely contractive \cite{harris_generalization_1981,hamana_triple_1999},
and it is injective if and only if is completely isometric \cite[Lemma 8.3.2]%
{blecher_operator_2004}. The range of a completely contractive projection $P$
of a TRO $X$ is also a TRO when endowed with the triple product $\left(
x,y,z\right) \mapsto P\left( xy^{\ast }z\right) $, where $xy^{\ast }z$ is
the triple produce evaluated in $X$ \cite{youngson_completely_1983}. A \emph{%
representation} of a TRO $X$ is a triple morphism from $X$ to $B\left(
H,K\right) $ for some Hilbert spaces $H,K$.

A W*-TRO is a TRO $V\subset B\left( H,K\right) $ that is closed in the weak
operator topology. This is equivalent to the assertion that $V$ is a dual
Banach space \cite[Theorem 2.6]{effros_injectivity_2001}. Every triple
isomorphism between W*-TROs is automatically w*-continuous. In the language
of Hilbert modules, TROs correspond precisely to the \emph{full Hilbert
modules }over C*-algebras, while W*-TRO correspond to \emph{self-dual and
weakly full Hilbert modules }over von Neumann algebras; see \cite%
{zettl_characterization_1983,ruan_type_2004}.

In the following, we denote by $\mathcal{L}\left( V\right) $ the \emph{%
linking }C*-algebra of $V$, and by $e$ the canonical projection such that $e%
\mathcal{L}\left( V\right) e^{\bot }=V$. Following \cite{kaur_local_2002},
we also let $\mathcal{R}\left( V\right) $ be the weak*-closure of $\mathcal{L%
}\left( V\right) $. (It is proved in \cite{kaur_local_2002} that this is
well defined, and does not depend from the concrete representation of $%
\mathcal{L}\left( V\right) $.) The corner $e\mathcal{L}\left( V\right) e$
coincides with the closed linear span $V^{\star }V$ of operators of the form 
$x^{\ast }y$ for $x,y\in V$. One can similarly identify $e^{\bot }\mathcal{L}%
\left( V\right) e^{\bot }$ with $VV^{\star }$. It is observed in \cite[page
269]{kaur_local_2002} that $e\in \mathcal{R}\left( V\right) $, and if $V$ is
a W*-TRO, then $\mathcal{R}\left( V\right) $ is a von Neumann algebra called
the \emph{linking von Neumann algebra }of $V$, and $V=e\mathcal{R}\left(
V\right) e^{\bot }$. For an arbitrary TRO $V$, the linking von Neumann
algebra $\mathcal{R}\left( V^{\ast \ast }\right) $ of the second dual of $V$
can be identified with the weak*-closure of the linking C*-algebra of $V$ 
\cite[Proposition 2.4]{kaur_local_2002}.

A TRO is injective (resp.\ nuclear, exact) if and only if its linking
C*-algebra is injective (resp.\ nuclear, exact) \cite[Theorem 6.3, Theorem
6.5]{kaur_local_2002}. A finite-dimensional operator space $X$ is a TRO if
and only if it is injective \cite{smith_finite_2000}. In this case, $X$ is
the finite $\infty $-sum of spaces of matrices $M_{n,m}$ for $n,m\in \mathbb{%
N}$. We will use repeatedly the following result of Hamana \cite[Theorem
3.2(ii)]{hamana_triple_1999}, which can be regarded as a nonunital analog of
a well known result of Choi and Effros \cite[Theorem 4.1]%
{choi_completely_1976}.

\begin{proposition}[Hamana]
\label{Proposition:nonunitalCE}Suppose that $V,W$ are TROs and $\theta
:V\rightarrow W$ is a linear complete isometry. Then there exists a triple
morphism $\eta :\mathcal{T}\left( \theta \left[ V\right] \right) \rightarrow
V$ such that $\eta \circ \theta =id_{V}$, where $\mathcal{T}\left( \theta %
\left[ V\right] \right) $ is the subTRO of $W$ generated by the image of $V$
under $\theta $.
\end{proposition}

We will call a TRO which is an $M_{n}$-space an $n$-minimal TRO. The $1$%
-minimal TROs are precisely the commutative TROs in the sense of \cite[\S %
8.6.4]{blecher_operator_2004}, and are usually called $C_{\sigma }$-spaces
in the Banach space literature. More information on the theory of TROs can
be found in \cite%
{hestenes_ternary_1962,harris_generalization_1981,zettl_characterization_1983,junge_OL_2003,junge_OL_2004,kaur_local_2002,ruan_type_2004,blecher_multipliers_2004,neal_operator_2003}%
.

Suppose that $X$ is an operator space. A \emph{triple cover }of $X$ is a
pair $\left( \phi ,V\right) $ where $V$ is a TRO and $\phi :X\rightarrow V$
is a linear complete isometry. Triple covers naturally form a category,
where a morphism from $\left( \phi _{0},V_{0}\right) $ to $\left( \phi
_{1},V_{1}\right) $ is a triple morphism $\psi :V_{0}\rightarrow V_{1}$ such
that $\psi \circ \phi _{0}=\phi _{1}$. The (unique up to isomorphism)
terminal object in such a category is called the \emph{triple envelope} $%
\mathcal{T}_{e}\left( X\right) $ of $X$. The existence of such an object has
been proved independently by Ruan \cite{ruan_injectivity_1989} and Hamana 
\cite{hamana_triple_1999}.

The same proof as \cite[Proposition 8]{kirchberg_c*-algebras_1998}, where 
\cite[Lemma 6]{kirchberg_c*-algebras_1998} is replaced by Proposition \ref%
{Proposition:nonunitalCE}, shows that any operator space $X$ has also a
(unique) \emph{universal TRO}. This is a triple cover $\left( u_{X},\mathcal{%
T}_{u}\left( X\right) \right) $ of $X$ that satisfies the following
universal property: whenever $\phi $ is a complete contraction from $X$ to a
TRO $W$, there exists a (necessarily unique) triple morphism $\theta :%
\mathcal{T}_{u}\left( X\right) \rightarrow W$ such that $\theta \circ
u_{X}=\phi $. The proof also shows that $\mathcal{T}_{u}\left( X\right) $
has a faithful family of finite-dimensional representations. Indeed one can
consider the collection $\mathcal{F}$ of all surjective completely
contractive maps $s:X\rightarrow V_{s}$ where $V_{s}$ is a
finite-dimensional TRO. Let $W$ be the $\infty $-sum of $V_{s}$ for $s\in 
\mathcal{F}$. One can then define $\mathcal{T}_{u}\left( X\right) $ to be
the subTRO of $W$ generated by the elements of the form $\left( s(x)\right)
_{s\in \mathcal{F}}$ for $x\in X$. The completely isometric map $%
u_{X}:X\rightarrow \mathcal{T}_{u}\left( X\right) $ is given by $x\mapsto
\left( s(x)\right) _{s\in \mathcal{F}}$.

In the following, we will identify an operator space with a subspace of its
triple envelope. We will denote the canonical inclusion of an operator space 
$X$ into its universal TRO by $u_{X}$. The universal property of $\mathcal{T}%
_{u}\left( X\right) $ shows that there exists a unique (necessarily
surjective) triple morphism $\sigma _{X}:\mathcal{T}_{u}\left( X\right)
\rightarrow \mathcal{T}_{e}\left( X\right) $ such that $\sigma _{X}\circ
u_{X}$ is the inclusion map from $X$ into $\mathcal{T}_{e}\left( X\right) $.

The same proof as \cite[Proposition 9]{kirchberg_c*-algebras_1998}---where
one replaces Arveson's extension theorem \cite[Theorem 1.2.3]%
{arveson_subalgebras_1969} with the Haagerup-Paulsen-Wittstock extension
theorem \cite[Theorem 8.2]{paulsen_completely_2002}---shows that a complete
isometry between operator spaces \textquotedblleft lifts\textquotedblright\
to an injective triple morphism between the corresponding universal TROs.
The precise statement is th following:

\begin{lemma}
\label{Lemma:lift}If $\phi :X\rightarrow Y$ is a linear complete isometry,
then there exists a unique injective triple morphism $\overline{\phi }:%
\mathcal{T}_{u}\left( X\right) \rightarrow \mathcal{T}_{u}\left( Y\right) $
such that $\overline{\phi }\circ u_{X}=u_{Y}\circ \phi $.
\end{lemma}

Our goal is to prove that the canonical triple morphism $\sigma _{\mathbb{NG}%
}:\mathcal{T}_{u}\left( \mathbb{NG}\right) \rightarrow \mathcal{T}_{e}\left( 
\mathbb{NG}\right) $ is injective and, hence, a triple isomorphism.

\begin{theorem}
\label{Theorem:universalNG}The canonical triple morphism $\sigma _{\mathbb{NG%
}}:\mathcal{T}_{u}\left( \mathbb{NG}\right) \rightarrow \mathcal{T}%
_{e}\left( \mathbb{NG}\right) $ is a triple isomorphism.
\end{theorem}

\begin{proof}
Since $\sigma _{\mathbb{NG}}$ is a surjective triple morphism, it only
remains to prove that $\sigma _{\mathbb{NG}}$ is injective. As observed in
Remark \ref{Remark:limit-NG}, $\mathbb{NG}$ is the limit of an inductive
sequence $\left( X_{k}\right) _{k\in \mathbb{N}}$ with unital completely
isometric connective maps $\phi _{k}:X_{k}\rightarrow X_{k+1}$, where $X_{k}$
is completely isometrically isomorphic to a full matrix algebra. Denote by $%
\iota _{k}:X_{k}\rightarrow \mathbb{NG}$ the canonical inclusion. Here and
below we denote by $\overline{\iota }_{k}$ the unique triple morphism from $%
\mathcal{T}_{u}\left( X_{k}\right) $ to $\mathcal{T}_{u}\left( \mathbb{NG}%
\right) $ that \textquotedblleft lifts\textquotedblright\ $\iota _{k}$, in
the sense that $\overline{\iota }_{k}\circ u_{X_{k}}=u_{\mathbb{NG}}\circ
\iota _{k}$. Observe that the set of elements of the form $\overline{\iota }%
_{k}(z)$ for $k\in \mathbb{N}$ and $z\in \mathcal{T}_{u}\left( X_{k}\right) $
is dense in $\mathbb{NG}$. Suppose that $\overline{\iota }_{k}(z)$ is such
an element. Fix $\delta >0$ small enough. As observed above, $\mathcal{T}%
_{u}\left( X_{k}\right) $ has a faithful family of finite-dimensional
representations. Therefore there exist $d\in \mathbb{N}$ and a triple
momorphism $\pi :\mathcal{T}_{u}\left( X_{k}\right) \rightarrow M_{d}$ such
that $\pi \circ u_{X_{k}}$ is a complete isometry and $\left\Vert \pi
(z)\right\Vert \geq \left( 1-\delta \right) \left\Vert z\right\Vert $. Set $%
\theta :=\pi \circ u_{X_{k}}:X_{k}\rightarrow M_{d}$, and observe that $\pi $
is the unique triple momorphism from $\mathcal{T}_{u}\left( X_{k}\right) $
to $M_{d}$ such that $\pi \circ u_{X_{k}}=\theta $. By the homogeneity
property of $\mathbb{NG}$, there is a unital complete isometry $\eta
:M_{d}\rightarrow \mathbb{GS}$ such that $\left\Vert \eta \circ \theta
-\iota _{k}\right\Vert _{cb}\leq \delta $. Observe that, for $\delta $ small
enough, we have $\left\Vert \left( \overline{\eta \circ \theta }\right) (z)-%
\overline{\iota }_{k}(z)\right\Vert \leq \varepsilon $. By Proposition \ref%
{Proposition:nonunitalCE} there is a triple momorphism $\mu :\mathcal{T}%
\left( \eta \left[ M_{d}\right] \right) \subset \mathcal{T}_{u}\left( 
\mathbb{\mathbb{NG}}\right) \rightarrow M_{d}$ such that $\mu \circ \eta
=id_{M_{d}}$. Observe now that $\mu \circ \sigma _{\mathbb{\mathbb{NG}}%
}\circ \overline{\eta \circ \theta }:\mathcal{T}_{u}\left( X_{k}\right)
\rightarrow Y$ is a $\ast $-homomorphism such that%
\begin{equation*}
\mu \circ \sigma _{\mathbb{\mathbb{NG}}}\circ \overline{\eta \circ \theta }%
\circ u_{X_{k}}=\mu \circ \sigma _{\mathbb{\mathbb{NG}}}\circ u_{\mathbb{%
\mathbb{NG}}}\circ \eta \circ \theta =\mu \circ \eta \circ \theta =\theta 
\text{.}
\end{equation*}%
Therefore $\mu \circ \sigma _{\mathbb{\mathbb{NG}}}\circ \overline{\eta
\circ \theta }=\pi $. Hence we have that%
\begin{equation*}
\left\Vert \sigma _{\mathbb{\mathbb{NG}}}\left( \overline{\iota }%
_{k}(z)\right) \right\Vert \geq \left\Vert \left( \sigma _{\mathbb{\mathbb{NG%
}}}\circ \overline{\eta \circ \theta }\right) (z)\right\Vert -\varepsilon
\geq \left\Vert \left( \mu \circ \sigma _{\mathbb{\mathbb{NG}}}\circ 
\overline{\eta \circ \theta }\right) (z)\right\Vert -\varepsilon =\left\Vert
\pi (z)\right\Vert -\varepsilon \geq \left\Vert \overline{\iota }%
_{k}(z)\right\Vert -2\varepsilon \text{.}
\end{equation*}%
This concludes the proof that $\sigma $ is injective.
\end{proof}

\begin{corollary}
\label{Corollary:non-embed}The noncommutative Gurarij space $\mathbb{NG}$
does not admit any unital completely isometric embedding into an exact
C*-algebra. Moreover if $\mathbb{NG}\subset B\left( H\right) $ is a unital
completely isometric representation of $\mathbb{NG}$ then the TRO generated
by $\mathbb{NG}$ inside $B\left( H\right) $ is isomorphic to $\mathcal{T}%
_{e}\left( \mathbb{NG}\right) $.
\end{corollary}

\begin{proof}
The second assertion is an immediate consequence of Theorem \ref%
{Theorem:universalNG}. We now prove the first assertion. As recalled above,
a TRO is exact if and only if it is $1$-exact as an operator space, if and
only if its linking C*-algebra is exact \cite[Theorem 4.4]{kaur_local_2002}.
Moreover a (surjective) triple morphism between TROs induces a (surjective) $%
\ast $-homomorphism between the corresponding linking algebras. Since the
class of exact C*-algebras is closed under quotients \cite[Corollary 9.4.3]%
{brown_C*-algebras_2008}, it follows that the image of an exact TRO under a
triple morphism is exact. Observe now that if $X$ is an operator system, $%
C_{u}^{\ast }\left( X\right) $ is its universal C*-algebra as defined in 
\cite[\S 3]{kirchberg_c*-algebras_1998}, and $\mathcal{T}_{u}\left( X\right) 
$ is the triple envelope of $X$, then the universal property of $\mathcal{T}%
_{u}\left( X\right) $ implies the existence of a surjective triple morphism
from $\mathcal{T}_{u}\left( X\right) $ to $C_{u}^{\ast }\left( X\right) $.
Therefore if $C_{u}^{\ast }\left( X\right) $ is not exact, then $\mathcal{T}%
_{u}\left( X\right) $ is not exact as well. Since the universal C*-algebra
of $M_{2}\left( \mathbb{C}\right) $ is not exact \cite[Section 5]%
{kirchberg_c*-algebras_1998}, it follows that the universal TRO of $%
M_{2}\left( \mathbb{C}\right) $ is not exact. Since $M_{2}\left( \mathbb{C}%
\right) $ embeds completely isometrically into $\mathbb{NG}$, it follows
from Lemma \ref{Lemma:lift} that $\mathcal{T}_{u}\left( M_{2}\left( \mathbb{C%
}\right) \right) $ embeds as a subTRO of $\mathcal{T}_{u}\left( \mathbb{NG}%
\right) \cong \mathcal{T}_{e}\left( \mathbb{NG}\right) $. Therefore $%
\mathcal{T}_{e}\left( \mathbb{NG}\right) $ is not exact. Now, if $\mathbb{NG}
$ embeds completely isometrically into a TRO $V$, then the universal
property of $\mathcal{T}_{e}\left( \mathbb{NG}\right) $ implies that $%
\mathcal{T}_{e}\left( \mathbb{NG}\right) $ is the image of under a triple
morphism of the subTRO of $V$ generated by (the image of) $\mathbb{NG}$.
Hence $V$ is not exact.
\end{proof}

\begin{corollary}
\label{Corollary:nonisomorphic}The noncommutative Gurarij space $\mathbb{NG}$
is not completely isometric to a TRO or a C*-algebra.
\end{corollary}

\subsection{Characterization of the noncommutative Gurarij space\label%
{Subsection:characterization-NG}}

Theorem \ref{Theorem:universalNG} asserts that the canonical triple morphism 
$\sigma _{\mathbb{NG}}:\mathcal{T}_{u}\left( \mathbb{NG}\right) \rightarrow 
\mathcal{T}_{e}\left( \mathbb{NG}\right) $ is injective. We will now prove
that such a property characterizes $\mathbb{NG}$ among the separable nuclear
operator spaces of dimension at least $1$.

\begin{theorem}
\label{Theorem:characterization-NG}If $X$ is a separable nuclear operator
space of dimension at least $1$ such that $\sigma _{X}:\mathcal{T}_{u}\left(
X\right) \rightarrow \mathcal{T}_{e}\left( X\right) $ is injective, then $X$
is completely isometrically isomorphic to $\mathbb{NG}$.
\end{theorem}

The following notion as been consider in \cite{junge_OL_2003}: an operator
space $X$ is a \emph{rigid rectangular }$\mathcal{OL}_{\infty ,1+}$ space if
for any finite subset $F$ of $X$ there exist a finite-dimensional TRO $V$
and a linear complete isometry $\phi :V\rightarrow F$ such that any element
of $F$ is at distance at most $\varepsilon $ from the range of $\phi $. Any
rigid $\mathcal{OL}_{\infty ,1+}$ space is, in particular, a rigid
rectangular $\mathcal{OL}_{\infty ,+1}$ space. We have observed in Remark %
\ref{Remark:limit-NG} that $\mathbb{NG}$ is a rigid $\mathcal{OL}_{\infty
,+1}$ space. Using Proposition \ref{Proposition: approximate amalgamation
1-exact} together with an approximate intertwining argument one obtains the
following characterization of $\mathbb{NG}$ among the separable rigid
rectangular $\mathcal{OL}_{\infty ,1+}$ spaces.

\begin{lemma}
\label{Lemma:characterization}Fix $n\in \mathbb{N}$. Suppose that $X$ is a
separable rigid $\mathcal{OL}_{\infty ,1+}$ operator space containing an
isometric copy of $M_{n}$. Assume that, if $d\in \mathbb{N}$, $V$ is a
finite-dimensional TRO containing a completely isometric copy of $M_{n}$, $%
\phi :V\rightarrow M_{d}$ and $f:V\rightarrow X$ are complete isometries,
and $\varepsilon >0$, then there exists an injective complete contraction $%
\widehat{f}:M_{d}\rightarrow X$ such that $||\widehat{f}^{-1}||_{cb}{}<1+%
\varepsilon $ and $||\widehat{f}\circ \phi -f||{}<\varepsilon $. Then $X$ is
completely isometrically isomorphic to $\mathbb{\mathbb{NG}}$.
\end{lemma}

Let $X$ be a nuclear operator space. Since $X$ is nuclear, its second dual $%
X^{\ast \ast }$ is an injective operator space \cite[Theorem 4.5]%
{effros_injectivity_2001}, and hence a dual TRO \cite[Proposition 3.2]%
{effros_injectivity_2001}. We will identify $X$ with its image under the
canonical embedding inside $X^{\ast \ast }$. The same argument as \cite[%
Proposition 10]{kirchberg_c*-algebras_1998}, where one replaces \cite[Lemma 6%
]{kirchberg_c*-algebras_1998} with\ Proposition \ref{Proposition:nonunitalCE}%
, shows that one can identify $\mathcal{T}_{e}\left( X\right) $ with the
subTRO of $X^{\ast \ast }$ generated by $X$.

As mentioned in \S \ref{Subsection:triple-Gn}, we denote by $\mathcal{R}%
\left( X^{\ast \ast }\right) $ the linking von\ Neumann algebra of $X^{\ast
\ast }$.\ Then there is a projection $e\in \mathcal{R}\left( X^{\ast \ast
}\right) $ with central cover $1$ such that $X^{\ast \ast }=e\mathcal{R}%
\left( X^{\ast \ast }\right) e^{\bot }$, where $e^{\bot }=1-e$. We denote by 
$\mathcal{R}\left( X^{\ast \ast }\right) _{I_{<\infty }}$ the sum of the
type $I_{n}$ parts of $\mathcal{R}\left( X^{\ast \ast }\right) $ for $n\in 
\mathbb{N}$. Let $z_{I_{<\infty }}$ be the corresponding central projection,
and $e_{I_{<\infty }}=ez_{I_{<\infty }}$. We abbreviate $e_{I_{<\infty }}%
\mathcal{R}\left( X^{\ast \ast }\right) e_{I_{<\infty }}^{\bot }$ by $%
X_{I_{<\infty }}^{\ast \ast }$. It is showed in \cite{ruan_type_2004} that
the type decomposition of von Neumann algebras can be generalized to dual
TROs. Referring to such a type decomposition, $X_{I_{<\infty }}^{\ast \ast }$
is the sum of the type $I_{n,m}$ parts of $X^{\ast \ast }$ for $n,m\in 
\mathbb{N}$.

It will be crucial in the following a lifting theorem of Effros and Ruan for
quotients of operator spaces by $M$\emph{-ideals}. The theory of $M$-ideals
for Banach spaces was initiated by Alfsen and Effros in the influential
papers \cite{alfsen_structure_1972-1,alfsen_structure_1972-2}. The
generalization of this theory to operator spaces has been initially
considered by Effros and Ruan \cite{effros_mapping_1994}, where \emph{%
complete }$M$-\emph{ideals }are introduced. A subspace $J$ of an operator
space $X$ is a complete $M$-ideal if the weak*-closure $L$ of $J$ inside $%
X^{\ast \ast }$ is a complete $L$-summand. This means that there exists a
completely contractive projection $P$ of $X^{\ast \ast }$ onto $L$ such that 
$||x||{}=\max \left\{ ||P(x)||,\left\vert \left\vert x-P^{(n)}(x)\right\vert
\right\vert \right\} $ for every $n\in \mathbb{N}$ and $x\in M_{n}\left(
X^{\ast \ast }\right) $.

A one-sided theory that studies left $M$-ideals and right $M$-ideals in
operator spaces has later been developed \cite%
{blecher_one-sided_2004-2,blecher_one-sided_2004-1,blecher_multipliers_2001,blecher_multiplier_2004, blecher_calculus_2006,bohle_k-theory_2011,bohle_universal_2014,bohle_k-theoretic_2015}%
. In the following we will only consider (bilateral) complete $M$-ideals in
operator spaces. These coincide with the subspaces that are simultaneously
left $M$-ideals and right $M$-ideals. A subspace $X$ of a TRO $V$ is called
a \emph{triple ideal }if it is a $V^{\star }V$-$VV^{\star }$-bimodule \cite[%
\S 8.3.1]{blecher_operator_2004}. The complete $M$-ideals of a TRO $V$ are
precisely the triple ideals, and any $M$-ideal is automatically a complete $%
M $-ideal \cite[\S 8.5.20]{blecher_operator_2004}. The proof of \cite[%
Theorem 5.2]{effros_mapping_1994} starting from \cite[Lemma 5.1]%
{effros_mapping_1994} can be easily adapted to show the following slightly
more general statement.\ Recall that an operator space $X$ satisfies the
operator metric approximation property if the identity map of $X$ is the
pointwise limit of finite-rank completely contractive maps from $X$ to $X$.

\begin{theorem}[Effros-Ruan]
\label{Theorem:lifting}Suppose that $A$ is an operator space that satisfies
the operator metric approximation property, $F\subset A$ is a
finite-dimensional subspace, $Z,X$ are operator spaces, and $P:Z\rightarrow
X $ is a complete quotient mapping whose kernel is a complete $M$-ideal of $%
Z $. If $\phi :A\rightarrow X$ and $\psi :A\rightarrow Z$ are complete
contractions such that $||\left( P\circ \psi -\phi \right) |_{F}||{}<\delta $%
, then there exists a complete contraction $\tilde{\psi}:A\rightarrow Z$
such that $P\circ \widetilde{\psi }=\phi $ and $||\tilde{\psi}|_{F}-\psi
||<6\varepsilon $.
\end{theorem}

From now until the end of the section we will assume that $X$ is a nuclear
operator space such that the map $\sigma _{X}:\mathcal{T}_{u}\left( X\right)
\rightarrow \mathcal{T}_{e}\left( X\right) $ is injective. In particular
this implies that for any completely isometric inclusion $X\subset V$ inside
a TRO $V$ such that $X$ generates $V$ as a TRO, the canonical morphisms from 
$\mathcal{T}_{u}\left( X\right) $ to $V$ and from $V$ to $\mathcal{T}%
_{u}\left( X\right) $ are both isomorphism. We will denote $\mathcal{T}%
_{u}\left( X\right) \cong \mathcal{T}_{e}\left( X\right) $ simply by $%
\mathcal{T}\left( X\right) $, and will identify it with the subTRO of $%
X^{\ast \ast }$ generated by $X$. Matrix convexity is a noncommutative
generalization of the usual theory of convexity developed in \cite%
{effros_aspects_1978,wittstock_matrix_1984,effros_matrix_1997,effros_matrix_2009,winkler_non-commutative_1999,webster_krein-milman_1999,farenick_extremal_2000,farenick_pure_2004}%
. Compact matrix convex sets provide the natural geometric counterpart of
operator systems, generalizing the classical correspondence between compact
convex sets and function system. It is proved in \cite[Proposition 3.5]%
{webster_krein-milman_1999} that compact matrix convex sets are in natural
1:1 correspondence with matrix state spaces of operator systems. On the
operator space side, it is natural to consider matrix convex sets where not
necessarily square matrices are considered. These are introduced and studied
in \cite{lupini_rectangular_2016} under the name of \emph{compact
rectangular convex sets}. It is proved there that any compact rectangular
convex set arises as the space of completely contractive maps from an
operator space to $M_{n,m}$ for $n,m\in \mathbb{N}$. The natural analog of
the Krein-Milman theorem in this context is also proved in \cite%
{lupini_rectangular_2016}. In the proo of Lemma \ref{Lemma:typeI}, we will
use a consequence of such a fact as stated in \cite[Corollary 1.10]%
{lupini_rectangular_2016}.

\begin{lemma}
\label{Lemma:typeI}The map $x\mapsto e_{I_{<\infty }}xe_{I_{<\infty }}^{\bot
}$ from $X$ to $X_{I_{<\infty }}^{\ast \ast }$ is a complete isometry.
\end{lemma}

\begin{proof}
Suppose that $n,m\in \mathbb{N}$ and $x_{0}\in M_{n,m}\left( X\right) $ is
such that $||x_{0}||{}=1$. By \cite[Corollary 1.10]{lupini_rectangular_2016}
there exists completely contractive map $\phi :X\rightarrow M_{n,m}$ such
that $||\phi ^{(n,m)}\left( x_{0}\right) ||{}=1$ and that is a rectangular
extreme point of the space of completely contractive maps from $X$ into
matrix algebras. In particular, $\phi $ can not be written as the direct sum
of completely contractive maps $\psi _{i}:X\rightarrow M_{n_{i},m_{i}}$ for $%
n_{i}\leq n$ and $m_{i}\leq m$. Since $\mathcal{T}_{u}\left( X\right) =%
\mathcal{T}\left( X\right) $ is the subTRO of $X^{\ast \ast }$ generated by $%
X$, there exists a triple morphism $\pi :\mathcal{T}\left( X\right)
\rightarrow M_{n,m}$ that extends $\phi $. Consider now the second dual $\pi
^{\ast \ast }:\mathcal{T}\left( X\right) ^{\ast \ast }\rightarrow M_{n,m}$.
We have that $X\subset \mathcal{T}\left( X\right) \subset X^{\ast \ast
}\subset \mathcal{T}\left( X\right) ^{\ast \ast }$. We can consider the
restriction $\rho $ of $\pi ^{\ast \ast }$ to $X^{\ast \ast }$. Observe that 
$\rho $ is a normal representation of $X^{\ast \ast }$ that extends $\phi $.
Since $\phi $ is a rectangular extreme point, $\rho $ is irreducible.
Therefore we have that $||(I_{n,m}\otimes e_{I_{<\infty }})x(I_{n,m}\otimes
e_{I_{<\infty }}^{\bot })||{}\geq ||\rho ^{(n,m)}(x_{0})||{}=1$. This
concludes the proof.
\end{proof}

\begin{lemma}
\label{Lemma:rigid}The operator space $X$ is a rigid rectangular $\mathcal{OL%
}_{\infty ,1+}$ space.
\end{lemma}

\begin{proof}
Suppose that $E\subset X$ is a finite-dimensional subspace. Since the map $%
x\mapsto e_{I_{<\infty }}xe_{I_{<\infty }}^{\bot }$ from $X$ to $%
e_{I_{<\infty }}A\left( X^{\ast \ast }\right) e_{I_{<\infty }}^{\bot }$ is
completely isometric, there exists a finite rank central projection $q\in
A\left( X^{\ast \ast }\right) $ with the following property. Let $\pi
:X^{\ast \ast }\rightarrow X^{\ast \ast }$ the triple morphism $x\mapsto qx$%
. Then $\pi |_{E}$ is injective and $||\pi |_{E}^{-1}||_{cb}<1+\delta $.
Observe that, since $q$ has finite rank, $qX^{\ast \ast }=qX$ is a
finite-dimensional TRO. Let $J$ be the kernel of $\pi |_{X}$, which is a
complete $M$-ideal of $X$. Indeed the w*-closure of $X$ inside $X^{\ast \ast
}$ is the kernel of $\pi $, which is a complete $M$-summand of $X^{\ast \ast
}$ since $q$ is a central projection. Therefore by Theorem \ref%
{Theorem:lifting} there exists a complete isometry $\rho :qX\rightarrow X$
such that $\pi \circ \rho =id_{X}$ and $||\rho \circ \pi -id_{E}||{}<6\delta 
$.
\end{proof}

\begin{lemma}
\label{Lemma:lift-universal}Suppose that $E$ is a finite-dimensional
operator space, $W$ is a finite-dimensional TRO, $\eta :E\rightarrow W$ is
an injective complete contraction such that the range of $\eta $ generates $%
W $ as a TRO, and $\varepsilon ,\delta >0$. If $||\eta ^{-1}||_{cb}<1+\delta 
$ and $f:E\rightarrow X$ is a complete isometry, then there exists a
complete contraction $g:W\rightarrow X$ such that $||g\circ \eta
-f||{}<6\delta $ and $||g^{-1}||_{cb}<\varepsilon $.
\end{lemma}

\begin{proof}
By the universal property of the universal triple cover, there exists a
surjective triple morphism $P:\mathcal{T}_{u}\left( E\right) \rightarrow W$
such that $P\circ u_{E}=\eta $. The kernel of $P$ is a complete $M$-ideal of 
$\mathcal{T}_{u}\left( E\right) $. Therefore by Theorem \ref{Theorem:lifting}
there exists a completely isometric map $\phi :W\rightarrow \mathcal{T}%
_{u}\left( E\right) $ such that $P\circ \phi $ is the identity of $W$ and $%
||P\circ u_{E}-\phi \circ \eta ||{}<6\delta $. Suppose now that the linear
complete isometry $\overline{f}:\mathcal{T}_{u}\left( E\right) \rightarrow 
\mathcal{T}_{u}\left( X\right) =\mathcal{T}\left( X\right) $ is obtained
from $f$ as in Lemma \ref{Lemma:lift}. Then $\overline{f}\circ \phi
:W\rightarrow \mathcal{T}\left( X\right) $ is a linear complete isometry
such that $||\overline{f}\circ \phi \circ \eta -f|_{E}||{}<6\delta $. One
can then proceed as in Lemma \ref{Lemma:rigid} to obtain a finite rank
central projection $q\in A\left( X^{\ast \ast }\right) $ and a complete
isometry $\rho :qX\rightarrow X$ such that the following holds: if $\pi
:X\rightarrow qX=qX^{\ast \ast }$ is the map $x\mapsto qx$, and $g=\rho
\circ \pi \circ \overline{f}\circ \phi $, then $g$ is an injective complete
contraction with $||g^{-1}||_{cb}{}<\varepsilon $ such that $||g\circ \eta
-f||{}<6\delta $. This concludes the proof.
\end{proof}

The same proof as Lemma \ref{Lemma:lift-universal} with the extra ingredient
of Proposition \ref{Proposition:perturbation} shows that, if $\delta $ is
obtained from $\varepsilon $ and $W$ as in Proposition \ref%
{Proposition:perturbation}, then one can find a completely isometric $%
g:W\rightarrow X$ such that $||g\circ \eta -f||{}<6\delta $.

It is easy to see that for any finite-dimensional nonzero TRO $V$ one has
that $\mathcal{T}_{u}\left( V\right) $ is infinite-dimensional; see\ \cite[%
\S 5.1]{kirchberg_c*-algebras_1998}. This fact together with Lemma \ref%
{Lemma:rigid} shows that $X$ is infinite-dimensional. In particular, since $%
X $ is a rigid rectangular $\mathcal{OL}_{\infty ,1+}$ space by Lemma \ref%
{Lemma:rigid}, it contains a completely isometric copy of the unital
C*-algebra $V=\mathbb{C}\oplus \mathbb{C}\oplus \mathbb{C}$. Since the
C*-envelope of $V$ is not exact \cite[\S 5.4]{kirchberg_c*-algebras_1998},
it follows that $\mathcal{T}_{u}\left( V\right) $ is not exact as well.
Since $V\subset X$ completely isometrically, $\mathcal{T}_{u}\left( V\right)
\subset \mathcal{T}_{u}\left( X\right) $ and hence $\mathcal{T}_{u}\left(
X\right) $ is not exact. Thus the proof of Lemma \ref{Lemma:typeI} shows
that $X$ contains a completely isometric copy of $M_{n}$ for every $n\in 
\mathbb{N}$. We now verify that $X$ satisfies the assumptions of Lemma \ref%
{Lemma:characterization} for $n=3$.

If $r,s\in \mathbb{N}$ we denote by $e_{ij}$ the matrix units of $M_{r,s}$.
We canonically identify $M_{n_{1},m_{1}}\oplus M_{n_{2},m_{2}}$ with a
subspace of $M_{n_{1}+n_{2},m_{1}+m_{2}}$, via the usual identification of $%
B\left( H_{1},K_{1}\right) \oplus ^{\infty }B\left( H_{2},K_{2}\right) $
with $B\left( H_{1}\oplus H_{2},K_{1}\oplus K_{2}\right) $ for Hilbert
spaces $H_{1},H_{2},K_{1},K_{2}$. A similar convention will apply to sums
with more addends. It follows from Proposition \ref{Proposition:nonunitalCE}
and \cite[Lemma 3.2.4]{bohle_k-theory_2011} that if $n:=n_{1}+\cdots
+n_{\ell }$, $m:=m_{1}+\cdots +m_{\ell }$, and $\phi :M_{n_{1},m_{1}}\oplus
^{\infty }\cdots \oplus ^{\infty }M_{n_{\ell },m_{\ell }}\rightarrow M_{r,s}$
is a linear complete isometry, then $r\geq n$, $s\geq m$, and $\phi $ has
the form $x\mapsto u\left( x\oplus \psi (x)\right) v$, where $v\in M_{s}$
and $u\in M_{r}$ are unitaries, and $\psi :M_{n_{1},m_{1}}\oplus ^{\infty
}\cdots \oplus ^{\infty }M_{n_{\ell },m_{\ell }}\rightarrow M_{r-n,s-m}$ is
a complete contraction.

Suppose that $V$ is a finite-dimensional TRO containing an isometric copy of 
$M_{3}$, $d\in \mathbb{N}$, $\phi :V\rightarrow M_{d}$ and $f:V\rightarrow X$
are complete isometries, and $\varepsilon >0$. Without loss of generality we
can assume that $\varepsilon \in \left( 0,1\right) $. Set $\delta
:=\varepsilon /14$. After replacing $V$ with its image inside $M_{d}$, we
can assume that $V\subset M_{d}$. Our goal is to prove that there exists an
injective complete contraction $g:M_{d}\rightarrow X$ such that $||g\circ
\eta -d||{}<\varepsilon $ and $||g^{-1}||_{cb}<\varepsilon $. If $d=3$, then 
$V=M_{d}$ and there is nothing to prove, so we can assume that $d\geq 4$.
Let $\iota :V\rightarrow M_{d}$ be the inclusion map. After replacing $V$
with $uVw$ for some unitaries $u,w\in M_{d}$ we can assume that $V$ contains
the matrix units $e_{ij}$ for $1\leq i,j\leq 3$. Fix an element $a\in M_{d}$
such that $C^{\ast }\left( a\right) =span\left\{ e_{ij}:3\leq i,j\leq
d\right\} $, and $t>0$. Let $J=e_{33}+\cdots +e_{dd}$ and $L=e_{14}+\cdots
+e_{1d}$. Define $\eta _{t}:V\rightarrow M_{d}$ by setting $\eta _{t}\left(
e_{12}\right) =e_{12}+tJ$, $\eta _{t}\left( e_{21}\right) =e_{21}+ta$, $\eta
\left( e_{33}\right) =e_{33}+tL$, and $\eta _{t}\left( e_{ij}\right) =e_{ij}$
whenever $1\leq i,j\leq d$ and $\left( i,j\right) \notin \left\{ \left(
1,2\right) ,\left( 2,1\right) ,\left( 3,3\right) \right\} $. Let $T$ be the
subTRO of $M_{d}$ generated by the range of $\eta _{t}$. Observe that $%
T=M_{d}$. Indeed by the choice of $a$ and $J$ and the definition of $\eta
_{t}$ at $e_{12}$ and $e_{21}$ we have that for any $b\in span\left\{
e_{ij}:3\leq i,j\leq d\right\} $ there exist $\lambda ,\mu \in \mathbb{C}$
such that $\lambda e_{12}+\mu e_{21}+b\in T$. Since $e_{11},e_{22}\in T$ we
conclude that $e_{12},e_{21}\in T$ and hence $span\left\{ e_{ij}:3\leq
i,j\leq d\right\} \subset T$. In particular $e_{jj}\in T$ for every $1\leq
j\leq d$. Therefore for $4\leq j\leq d$ we have $e_{1j}=t^{-1}e_{11}\eta
_{t}\left( e_{33}\right) e_{jj}\in T$ and $e_{2j}=e_{21}e_{11}e_{1j}\in T$.
This concludes the proof that $T=M_{d}$.

Fix $t>0$ such that $\left\Vert \eta _{t}-\iota \right\Vert _{cb}<\delta $.
Let now $\eta =\frac{1}{1+\delta }\eta _{t}$ and observe that $\eta $ is an
injective complete contraction such that $\left\Vert \eta ^{-1}\right\Vert
_{cb}\leq \frac{1+\delta }{1-\delta }<1+2\delta $ and $\left\Vert \eta
-\iota \right\Vert <2\delta $. Therefore by Lemma \ref{Lemma:lift-universal}
there exists a completely contractive map $g:M_{d}\rightarrow X$ such that $%
\left\Vert g\circ \eta -f\right\Vert <12\delta $ and $\left\Vert
g^{-1}\right\Vert _{cb}<\varepsilon $. As a consequence we also have $%
\left\Vert g\circ \eta _{t}-f\right\Vert <14\delta =\varepsilon $. This
concludes the proof that $X$ is completely isometric to $\mathbb{NG}$.

\subsection{The triple envelope of $\mathbb{G}_{n}$\label%
{Subsection:triple-Gn}}

In this sections we consider $M_{n}$-spaces as operator spaces endowed with
their canonical minimal operator space structure. An $M_{n}$-space is a
rigid rectangular $\mathcal{OL}_{\infty ,1+}$ space if and only if for any
finite subset $F$ of $X$ there exist a finite-dimensional $n$-minimal TRO $V$
and a linear complete isometry $\phi :V\rightarrow F$ such that any element
of $F$ is at distance at most $\varepsilon $ from the range of $\phi $. In
particular, as pointed out in \cite{junge_OL_2003}, a Banach space is a
rigid (rectangular) $\mathcal{OL}_{\infty ,1+}$ space if and only if it is a
Lindenstrauss space.

It follows from Hamana's proof of the existence of the triple envelope \cite%
{hamana_triple_1999} that, if $X$ is an operator space and $I\left( X\right) 
$ is its injective envelope, then $\mathcal{T}_{e}\left( X\right) $ is the
subTRO of $I\left( X\right) $ generated by $X$; see also \cite[\S 4.4.2]%
{blecher_operator_2004}. (An injective operator space has a canonical
operator space structure, since the range of a complete contractive
projection of $B\left( H\right) $ is a TRO \cite[Theorem 4.4.9]%
{blecher_operator_2004}.) If $X$ is a $M_{n}$-space, then $X\subset \ell
^{\infty }\left( \kappa ,M_{n}\right) $ for some index set $\kappa $, and
the latter space is injective. Therefore $\mathcal{T}_{e}\left( X\right)
\subset I\left( X\right) \subset \ell ^{\infty }\left( \kappa ,M_{n}\right) $%
, and hence $\mathcal{T}_{e}\left( X\right) $ is an $n$-minimal TRO.

The universal triple envelope $\mathcal{T}_{u}\left( X\right) $ of an $M_{n}$%
-space is in general not an $M_{n}$-space.\ Indeed as observed above it is
nonexact whenever $X$ is a nonzero finite-dimensional TRO. However one can
consider the $n$-minimal analog of the universal triple envelope.

\begin{definition}
The (necessarily unique) $n$-minimal universal TRO $\mathcal{T}%
_{u}^{n}\left( X\right) $ of an $M_{n}$-space $X$ is a pair $\left(
u_{X}^{n},\mathcal{T}_{u}^{n}\left( X\right) \right) $ where $\mathcal{T}%
_{u}^{n}\left( X\right) $ is an $n$-minimal TRO and $u_{X}^{n}:X\rightarrow 
\mathcal{T}_{u}^{n}\left( X\right) $ is a complete isometry such that the
image of $u_{X}^{n}$ generates $\mathcal{T}_{u}^{n}\left( X\right) $ as a
TRO, satisfying the following universal property: whenever $\phi $ is a
complete contraction from $X$ to an $n$-minimal TRO $W$, there exists a
(necessarily unique) triple morphism $\theta :\mathcal{T}_{u}^{n}\left(
X\right) \rightarrow Z$ such that $\theta \circ u_{X}^{n}=\phi $.
\end{definition}

The same argument as for the existence of $\mathcal{T}_{u}\left( X\right) $
shows that the $n$-minimal universal TRO exists. Indeed let $\mathcal{F}^{n}$
be the collection of all the completely contractive maps $\phi :X\rightarrow
M_{n}$, and $W$ be the $\infty $-sum of copies of $M_{n}$ indexed by $%
\mathcal{F}^{n}$. One can then define $u_{X}^{n}:X\rightarrow W$ to be the
map $x\mapsto \left( s(x)\right) _{s\in \mathcal{F}^{n}}$, and $\mathcal{T}%
_{u}^{n}\left( X\right) $ to be the subTRO of $W$ generated by the range of $%
u_{X}^{n}$.

The universality property of $\mathcal{T}_{u}\left( X\right) $ shows that
there exists a unique surjective triple morphism $\sigma _{X}^{n}:\mathcal{T}%
_{u}^{n}\left( X\right) \rightarrow \mathcal{T}_{e}\left( X\right) $. The
arguments from \S \ref{Subsection:triple-NG} and \S \ref%
{Subsection:characterization-NG} can be easily adapted to prove the
following characterization of $\mathbb{G}_{n}$.

\begin{theorem}
\label{Theorem:characterization-Gn}For every $n\in \mathbb{N}$, $\mathbb{G}%
_{n}^{u}$ is the unique separable nuclear $M_{n}$-space of dimension at
least $1$ with the property that the canonical map $\sigma _{\mathbb{G}_{n}}:%
\mathcal{T}_{u}^{n}\left( \mathbb{G}_{n}\right) \rightarrow \mathcal{T}%
_{e}\left( \mathbb{G}_{n}\right) $ is injective.
\end{theorem}

In particular, when $n=1$,\ Theorem \ref{Theorem:characterization-Gn}
provide the following new characterization of the Gurarij Banach space $%
\mathbb{G}$. Recall that a TRO is $1$-minimal if and only if it is a
commutative TRO in the sense of \cite[Proposition 8.6.5]%
{blecher_operator_2004}, and a Banach space is a nuclear when endowed with
its canonical minimal operator space structure if and only if it is a
Lindenstrauss space \cite[\S 8.6.4]{blecher_operator_2004}.

\begin{corollary}
\label{Corollary:characterization-G}The Gurarij Banach space $\mathbb{G}$ is
the unique separable Lindenstrauss space of dimension at least $1$ with the
property that the canonical map $\sigma _{\mathbb{G}}^{1}$ from the
universal commutative TRO $\mathcal{T}_{u}^{1}\left( \mathbb{G}\right) $ to
the triple envelope $\mathcal{T}_{e}\left( \mathbb{G}\right) $ is injective.
\end{corollary}

\subsection{A canonical construction of $\mathbb{\mathbb{NG}}$ and $\mathbb{G%
}_{n}$\label{Subsection:canonical-NG}}

The characterizations of $\mathbb{NG}$ and $\mathbb{G}_{n}$ provided by
Theorem \ref{Theorem:characterization-NG} and Theorem \ref%
{Theorem:characterization-Gn} can be used to provide a new canonical
construction of $\mathbb{\mathbb{NG}}$ and $\mathbb{G}_{n}$ for every $n\in 
\mathbb{N}$. In the particular case $n=1$ this provides a new construction
of the Gurarij space $\mathbb{G}$.

Start with a separable $M_{n}$ space $X_{n}$ of dimension at least $1$ and
then define for $k>n$, $X_{k+1}=\mathcal{T}_{u}^{k+1}\left( X_{k}\right) $.
One obtains an inductive sequence $\left( X_{k}\right) _{k>n}$ of TROs with
completely isometric connective maps $u_{X_{k}}^{k+1}:X_{k}\rightarrow
X_{k+1}$. Let $X_{\infty }$ be the corresponding limit in the category of
operator spaces, and let $\iota _{k}:X_{k}\rightarrow X_{\infty }$ be the
canonical inclusions. Observe that $X_{\infty }$ is a rigid rectangular $%
\mathcal{OL}_{\infty ,1+}$ space and, in particular, nuclear. We will now
show that $\sigma _{X_{\infty }}:\mathcal{T}_{u}\left( X_{\infty }\right)
\rightarrow \mathcal{T}_{e}\left( X_{\infty }\right) $ is injective. In view
of Theorem \ref{Theorem:characterization-NG}, this will prove that $%
X_{\infty }$ is completely isometrically isomorphic to $\mathbb{NG}$. We let 
$\iota _{i,j}:X_{i}\rightarrow X_{j}$ for $j>i$ be $u_{X_{j-1}}^{j}\circ
\cdots \circ u_{X_{i}}^{i+1}$.

We will identify $X_{\infty }$ with a subspace of $\mathcal{T}_{e}\left(
X_{\infty }\right) $. Suppose that $k\geq n$, $z\in \mathcal{T}_{u}\left(
X_{k}\right) $, and $\overline{\iota }_{k}(z)\in \mathcal{T}_{u}\left(
X\right) $. The set of such elements is dense in $\mathcal{T}_{u}\left(
X_{\infty }\right) $. Fix $\varepsilon >0$. Since $\mathcal{T}_{u}\left(
X_{n}\right) $ has a separating family of finite-dimensional
representations, there exists $m>k$ such that $\left\Vert z\right\Vert \leq
||\overline{\iota }_{km}(z)||{}+\varepsilon $, where $\overline{\iota }_{km}:%
\mathcal{T}_{u}\left( X_{k}\right) \rightarrow X_{m}$ is the triple morphism
obtained from $\iota _{km}:X_{k}\rightarrow X_{m}$ by applying the universal
property of the universal TRO. By Proposition \ref{Proposition:nonunitalCE},
there exists a triple morphism $\mu $ from the subTRO $Z$ of $\mathcal{T}%
_{e}\left( X\right) $ generated by the image of $\iota _{m}$ such that $\mu
\circ \iota _{m}=id_{X_{m}}$. We now have%
\begin{equation*}
\mu \circ \sigma _{X_{\infty }}\circ \overline{\iota }_{k}\circ
u_{X_{k}}=\mu \circ \sigma _{X_{\infty }}\circ u_{X_{\infty }}\circ \iota
_{k}=\mu \circ \iota _{k}=\iota _{km}\text{.}
\end{equation*}%
Therefore $\mu \circ \sigma _{X_{\infty }}\circ \overline{\iota }_{k}=%
\overline{\iota }_{km}$. Hence%
\begin{equation*}
\left\Vert \left( \sigma _{X_{\infty }}\circ \overline{\iota }_{k}\right)
(z)\right\Vert \geq \left\Vert \left( \mu \circ \sigma _{X_{\infty }}\circ 
\overline{\iota }_{k}\right) (z)\right\Vert \geq \left\Vert \overline{\iota }%
_{mk}(z)\right\Vert \geq \left\Vert \overline{\iota }_{k}(z)\right\Vert
-\varepsilon \text{.}
\end{equation*}%
This concludes the proof that $\sigma _{X_{\infty }}$ is injective, and
hence $X_{\infty }$ is completely isometrically isomorphic to $\mathbb{NG}$.

One can similarly give a canonical construction of $\mathbb{G}_{n}$. Suppose
that $X$ is any separable $M_{n}$-space. Define recursively $X_{0}:=X$ and $%
X_{k+1}=\mathcal{T}_{u}^{n}\left( X_{k}\right) $. This gives an inductive
sequence $\left( X_{k}\right) _{k\geq 1}$ of $n$-minimal\ TROs with
completely isometric connective maps $u_{X_{k}}^{n}:X_{k}\rightarrow X_{k+1}$%
. Reasoning as before, one can see that the corresponding limit $X_{\infty }$
in the category of operator spaces is completely isometrically isomorphic to 
$\mathbb{G}_{n}$.

\subsection{Universal automorphism group\label{Subsection:universal}}

It is clear from the definition of the universal TRO $\mathcal{T}_{u}\left(
X\right) $ that the assignment $X\mapsto \mathcal{T}_{u}\left( X\right) $ is
a \emph{functor }from the category of operator spaces and completely
contractive maps to the category of TROs and triple morphisms, which maps
completely isometric maps to injective triple morphisms. In particular any
surjective linear complete contraction $\alpha $ of $X$ admits a unique
\textquotedblleft lift\textquotedblright\ $\mathcal{T}_{u}\left( \alpha
\right) $ to a triple automorphism of $\mathcal{T}_{u}\left( X\right) $ with
the property that $\mathcal{T}_{u}\left( \alpha \right) \circ
u_{X}=u_{X}\circ \alpha $.

If $X$ is an operator space, then we denote by $\mathrm{Aut}\left( X\right) $
the group of surjective complete isometries from $X$ to itself. When $X$ is
a TRO, this coincides with the group of triple automorphisms of $X$. We
consider $\mathrm{Aut}\left( X\right) $ as a topological group endowed with
the topology of pointwise convergence. It is easy to see that whenever $X$
is separable, $\mathrm{Aut}\left( X\right) $ is a Polish group. The
assignment $\alpha \mapsto \mathcal{T}_{u}\left( \alpha \right) $ defined
above is an injective group homomorphism from $\mathrm{Aut}\left( X\right) $
to $\mathrm{Aut}\left( \mathcal{T}_{u}\left( X\right) \right) $. Since the
range of $u_{X}$ generates $\mathcal{T}_{u}\left( X\right) $ as a TRO, the
map $\alpha \mapsto \mathcal{T}_{u}\left( \alpha \right) $ is continuous
and, in fact, a homeomorphism onto its image. The same argument applies with
no change to the functor $\mathcal{T}_{u}^{n}$ from the category of $M_{n}$%
-spaces to the category of $n$-minimal TROs.

It follows from these remarks together with the canonical construction of $%
\mathbb{NG}$ from \S \ref{Subsection:canonical-NG} that any separable $M_{n}$%
-space $X$ can be realized completely isometrically as a subspace of $%
\mathbb{NG}$ in such a way that every surjective complete isometry $\alpha $
of $X$ extends to a surjective complete isometry $\widehat{\alpha }$ of $%
\mathbb{NG}$, and the map $\alpha \mapsto \widehat{\alpha }$ is a
topological embedding of Polish groups from $\mathrm{Aut}\left( X\right) $
to $\mathrm{Aut}\left( \mathbb{NG}\right) $. Since any Polish group is
isomorphic to the group of surjective isometries of a Banach space \cite[%
Theorem 2.2.1]{pestov_topological_1999}, it follows in particular that any
Polish group is isomorphic to a closed subgroup of $\mathrm{Aut}\left( 
\mathbb{NG}\right) $.

Similar conclusions apply to the spaces $\mathbb{G}_{n}$. If $X$ is a
separable $M_{n}$-space, then it can be realized completely isometrically as
a subspace of $\mathbb{G}_{n}$ in such a way that every surjective complete
isometry $\alpha $ of $X$ extends to a surjective complete isometry $%
\widehat{\alpha }$ of $\mathbb{G}_{n}$, and the map $\alpha \mapsto \widehat{%
\alpha }$ is a topological embedding of Polish groups from $\mathrm{Aut}%
\left( X\right) $ to $\mathrm{Aut}\left( \mathbb{G}_{n}\right) $. Again,
this entails that the group $\mathrm{Aut}\left( \mathbb{G}_{n}\right) $ of
surjective complete isometries of $\mathbb{G}_{n}$ is a universal Polish
groups. These observations provide the proof of item (6) in Theorem \ref%
{Theorem:main-NG} and Theorem \ref{Theorem:main-G_n}.

\subsection{Perturbation\label{Subsection:perturbation}}

The same proof as \cite[Lemma 7.1]{kerr_gromov-hausdorff_2009} allows one to
deduce from \cite[4.2.8]{blackadar_generalized_1997} the following
perturbation result: for any finite-dimensional C*-algebra $A$ and any $%
\varepsilon >0$, there exists $\delta >0$ (depending on $A$ and $\varepsilon
>0$) such that for any finite-dimensional C*-algebra $B$ and any injective
unital linear map $\phi :A\rightarrow B$ such that $||\phi ||_{cb}{}||\phi
^{-1}||_{cb}<1+\delta $, there exists a unital linear complete isometry $%
\psi :A\rightarrow B$ such that $||\phi -\psi ||{}<\varepsilon $. In this
section, we present the proof of the natural nonunital analog of such a
perturbation result, concerning approximately completely isometric linear
maps between finite-dimensional TROs. This result will be used in \S \ref%
{Subsection:indecomposability} to prove an indecomposability result for the
spaces $\mathbb{G}_{n}$. One can also use Proposition \ref%
{Proposition:perturbation} to give an alternative proof of Theorem \ref%
{Theorem:characterization-NG}.

\begin{proposition}
\label{Proposition:perturbation}Suppose that $L$ is a finite-dimensional $q$%
-minimal TRO and $\varepsilon >0$.\ There exists $\delta $ (depending on $%
\varepsilon $ and $L$) such that whenever $H_{1},K_{1},\ldots ,H_{\ell
},K_{\ell }$ are Hilbert spaces, $V$ is the sum $B\left( H_{1},K_{1}\right)
\oplus ^{\infty }\cdots \oplus ^{\infty }B\left( H_{\ell },K_{\ell }\right) $%
, and $\phi :L\rightarrow V$ is a linear map such that $||\phi ||_{q}||\phi
^{-1}||_{q}<1+\delta $, then there exists a linear complete isometry $\psi
:L\rightarrow V$ such that $||\phi -\psi ||{}<\varepsilon $.
\end{proposition}

\begin{remark}
\label{Remark:limit}It follows from Proposition \ref%
{Proposition:perturbation} that a separable operator space is a rigid
rectangular $\mathcal{OL}_{\infty ,1+}$ space if and only if it can be
realized as the direct limit of a sequence $\left( X_{n}\right) $ of
finite-dimensional TROs with completely isometric connective maps $\phi
_{n}:X_{n}\rightarrow X_{n+1}$. It also follows from Proposition \ref%
{Proposition:perturbation} and the small perturbation lemma \cite[Lemma
2.13.2]{pisier_introduction_2003} that for any $q\in \mathbb{N}$, $%
\varepsilon >0$, and finite-dimensional $q$-minimal TRO $L$, there exists $%
\delta >0$ (depending on $L$ and $\varepsilon $) such that if $\psi
:L\rightarrow B\left( H\right) $ is a linear map such that $||\phi
||_{q}||\phi ^{-1}||_{q}<1+\delta $, then $||\phi ||_{cb}||\phi
^{-1}||_{cb}<1+\varepsilon $.
\end{remark}

The following proof of Proposition \ref{Proposition:perturbation} has been
suggested by Caleb Eckhardt. The proof of \cite[Corollary 3.7]%
{eckhardt_perturbations_2010} can be easily adapted to prove the following:

\begin{lemma}
\label{Lemma:perturbation1}Suppose that $n,m\in \mathbb{N}$, $W$ is a finite 
$\infty $-sum of TROs of the form $B\left( H,K\right) $, and $\varepsilon >0$%
. Then there exists $\delta >0$ such that for any index set $I$, if $\phi
:M_{n,m}\rightarrow \ell ^{\infty }\left( I\right) \otimes W$ is a linear
map such that $||\phi ||_{cb}<1+\delta $ and $||\phi ^{-1}||_{cb}<1+\delta $%
, then there exists $i\in I$ such that the $i$-th coordinate map $\phi _{i}$
satisfies $||\phi _{i}^{-1}||_{cb}<1+\varepsilon $.
\end{lemma}

The proof of the following lemma is similar to the proof of \cite[Theorem 3.2%
]{eckhardt_perturbations_2010} from \cite[Corollary 3.7]%
{eckhardt_perturbations_2010}.

\begin{lemma}
\label{Lemma:perturbation2}Suppose that $n,m\in \mathbb{N}$ and $\varepsilon
>0$. Let $q=\max \left\{ n,m\right\} $. There exists $\delta >0$ such that
whenever $H_{1},K_{1},\ldots ,H_{\ell },K_{\ell }$ are Hilbert spaces, $V$
is the sum $B\left( H_{1},K_{1}\right) \oplus ^{\infty }\cdots \oplus
^{\infty }B\left( H_{\ell },K_{\ell }\right) $, and $\phi =\left( \phi
_{1},\ldots ,\phi _{\ell }\right) :M_{n,m}\rightarrow V$ is an injective
linear map such that $||\phi ||_{q}<1+\delta $ and $||\phi
^{-1}||_{cb}<1+\delta $, then there exist $1\leq i\leq l$, projections $p\in
B\left( H_{i}\right) $ and $q\in B\left( K_{i}\right) $ of rank $n$ and $m$
and a triple isomorphism $\theta :M_{n,m}\rightarrow pB\left(
H_{i},K_{i}\right) q$ such that, if $\psi $ is the map $x\mapsto \theta
(x)+p\phi (x)q$, then $||\phi _{i}-\psi ||{}<\varepsilon $.
\end{lemma}

\begin{proof}
Let $\varepsilon _{0}>0$ be small enough (depending on $\varepsilon ,n,m$).
By compactness there exists $\eta >0$ such that whenever $s,t\leq q$ and $%
f:M_{n,m}\rightarrow M_{s,t}$ is a complete contraction such that $%
||f^{-1}||_{cb}<1+\eta $, then $n\leq s$, $m\leq t$, and there exist
projections $p\in M_{s}$ and $q\in M_{t}$ of rank $n$ and $m$ respectively
and a triple isomorphism $\theta :M_{n,m}\rightarrow pM_{s,t}q$ such that $%
||\theta -p\phi q||{}<\varepsilon _{0}$. We set $H=H_{1}\oplus \cdots \oplus
H_{\ell }$, $K=K_{1}\oplus \cdots \oplus K_{\ell }$. We will identify $H_{i}$
with a subspace of $H$, $K_{i}$ with a subspace of $K$, and $V$ with a
subspace of $B\left( H,K\right) $. Consider the set $I$ of pairs $\left(
p,q\right) $ such that $p\in B\left( H_{i}\right) $ and $q\in B\left(
K_{i}\right) $ for some $i\in \left\{ 1,\ldots ,\ell \right\} $ are
projections of rank at most $\max \left\{ n,m\right\} $. Let $W$ be the $%
\infty $-sum of $pVq$ for $\left( p,q\right) \in I$, and $s:V\rightarrow W$
be the map $y\mapsto \left( pyq\right) _{(p,q)\in I}$. Observe that $S$ is a 
$q$-isometry. It follows from this and Smith's lemma \cite[Proposition 2.2.2]%
{effros_operator_2000} together with Lemma \ref{Lemma:perturbation1} that
choosing $\delta >0$ small enough guarantees that there exists $\left(
p_{0},q_{0}\right) \in I$ such that the completely bounded norm of the
inverse of the map $x\mapsto p\phi (x)q$ is at least $1-\eta $. Suppose that 
$1\leq i\leq \ell $ is such that $p_{0}\in B\left( H_{i}\right) $ and $%
q_{0}\in B\left( K_{i}\right) $. By the choice of $\eta $ this implies that
there exist subprojections $p\leq p_{0}$ and $q\leq q_{0}$ of rank $n$ and $%
m $ and a triple isomorphism $\theta :M_{n,m}\rightarrow pB\left(
H_{i},K_{i}\right) q$ such that $\left\Vert \theta -p\phi q\right\Vert
<\varepsilon _{0}$. It is clear that choosing $\varepsilon _{0}$ small
enough ensures that the conclusion holds.
\end{proof}

\begin{proof}[Proof of Proposition \protect\ref{Proposition:perturbation}]
Let us use the notation from the statement of Proposition \ref%
{Proposition:perturbation}. Let $L=M_{n_{1},m_{1}}\oplus ^{\infty }\cdots
\oplus ^{\infty }M_{n_{k},m_{k}}$. We identify canonically $M_{n_{i},m_{i}}$
with a subTRO of $T$. We let $\phi _{ij}$ for $i=1,2,\ldots ,k$ and $%
j=1,\ldots ,\ell $ be the $j$-th component of the restriction of $\phi $ to $%
M_{n_{i},m_{i}}$. In view of Lemma \ref{Lemma:perturbation2}, up to
replacing $\delta $ with a smaller strictly positive real number and after
perturbing $\phi $ we can assume that for every $i$ there exists $j_{i}$ and
projections $p_{i}\in B\left( H_{j_{i}}\right) $ and $q_{i}\in B\left(
H_{j_{i}}\right) $ such that $\phi _{ij}=\theta _{ij}(x)+p_{i}\phi
_{ij}q_{i} $ for some triple isomorphism $\theta
_{ij}:M_{n_{i}m_{i}}\rightarrow p_{i}B\left( H_{j_{i}},K_{j_{i}}\right)
q_{i} $. Suppose that $i_{0},i_{1}$ are distinct elements of $\left\{
1,2,\ldots ,k\right\} $ such that $j_{i_{1}}=j_{i_{2}}$. We claim that the
corresponding projections $q_{i_{1}},q_{i_{2}}\in B\left( K_{j_{i}}\right) $%
, as well as the projection $p_{i_{1}},p_{i_{2}}\in B\left( H_{j_{i}}\right) 
$, are approximately orthogonal. For simplicity of notation we can assume
that $i_{0}=1$ and $i_{1}=2$. Suppose that $v_{i}\in M_{n_{i},m_{i}}$ is a
partial isometry such that $\theta \left( v_{i}^{\ast }v_{i}\right) \leq
q_{i}$ for $i=1,2$. Consider the element $a$ of $M_{2,1}\left( L\right) $
such that $a_{11}=\left( v_{1},0,\ldots ,0\right) $ and $a_{21}=\left(
0,v_{2},0,\ldots ,0\right) $. Then we have that $\left\Vert a\right\Vert
^{2}=1$ and hence%
\begin{equation*}
\left( 1+\delta \right) ^{2}\geq \left\Vert (id_{M_{2,1}}\otimes S)\left(
a\right) \right\Vert ^{2}\geq \left\Vert \theta \left( v_{1}^{\ast
}v_{1}\right) +\theta \left( v_{2}^{\ast }v_{2}\right) \right\Vert \text{.}
\end{equation*}%
Since this is true for any such pair of partial isometries, we can conclude
that $\left\Vert q_{1}+q_{2}\right\Vert \leq \left( 1+\delta \right) ^{2}$.
A similar argument involving $1\times 2$ matrices rather than $2\times 1$
matrices shows that $\left\Vert p_{1}+p_{2}\right\Vert \leq \left( 1+\delta
\right) ^{2}$. Therefore after further perturbing the map $\phi $---with an
error which is arbitrarily small for $\delta $ small enough---we can assume
that $q_{i_{1}}$ and $q_{i_{2}}$ and $p_{i_{1}}$ and $p_{i_{2}}$ are
pairwise orthogonal whenever $i_{0},i_{1}$ are distinct elements of $\left\{
1,2,\ldots ,k\right\} $ such that $j_{i_{1}}=j_{i_{2}}$. One can then easily
further perturb such a map to get a complete isometry from $L$ to $V$.
\end{proof}

\subsection{Indecomposability\label{Subsection:indecomposability}}

\begin{proposition}
\label{Proposition:indecomposable}Suppose that $X$ and $Y$ are rigid
rectangular $\mathcal{OL}_{\infty ,1+}$-spaces of dimension at least $2$. If 
$n\geq 2$, then $X\otimes Y$ is not completely isometrically isomorphic to $%
\mathbb{G}_{n}$.
\end{proposition}

\begin{proof}
We denote by $\mathbb{E}_{D_{n}}$ the conditional expectation onto the
diagonal subalgebra of $M_{n}$. Suppose by contradiction that $X\otimes Y$
is completely isometrically isomorphic to $\mathbb{G}_{n}$ for some $n\geq 2$%
. In particular $X$ contains a copy of $M_{d_{0},d_{1}}$ and $Y$ contains a
copy of $M_{e_{0},e_{1}}$ where $e_{0}d_{0}=e_{1}d_{1}=n$, and both $\left(
d_{0},d_{1}\right) $ and $\left( e_{0},e_{1}\right) $ are different from $%
\left( 1,1\right) $. For simplicity of notation, we will assume that $%
d_{0}=e_{1}=n$ and $d_{1}=e_{0}=1$. In the general case one can proceed
analogously, by replacing the operator $\mathbb{E}_{D_{n}}$ with another
completely contractive projection depending on $d_{0},d_{1}$. In view of
Remark \ref{Remark:limit}, we can write $X$ and $Y$ as the direct limit of
sequences $\left( X_{k}\right) $ and $\left( Y_{k}\right) $ of
finite-dimensional $n$-minimal TROs with completely isometric connective
maps $\phi _{k}:X_{k}\rightarrow X_{k+1}$ and $\psi _{k}:Y_{k}\rightarrow
Y_{k+1}$ such that $X_{1}=M_{n,1}$ and $Y_{1}=M_{1,n}$. We let $\phi _{k,d}$
be $\phi _{d-1}\circ \phi _{d-2}\circ \cdots \circ \phi _{k}$ for $k<d$ and
similarly define $\psi _{k,d}$. We will identify $X_{k}$ and $Y_{k}$ with
subspaces of $X$ and $Y$ respectively. Let $f:M_{n}\rightarrow X_{1}\otimes
Y_{1}\subset X\otimes Y$ be the complete isometry given by the canonical
identification of $M_{n}$ with $M_{n,1}\otimes M_{1,n}$. Let $\phi
:M_{n}\rightarrow M_{n}\oplus ^{\infty }M_{n}$ be defined by $a\mapsto
\left( a,\mathbb{E}_{D_{n}}\left( a\right) \right) $. Fix $\varepsilon >0$
and $\delta >0$ small enough. By the homogeneity property of $\mathbb{G}_{n}$%
, there exist $k\in \mathbb{N}$ and a complete isometry $g:M_{n}\oplus
^{\infty }M_{n}\rightarrow X_{k}\otimes Y_{k}$ such that $\left\Vert g\circ
\phi -\left( \phi _{1,k}\otimes \psi _{1,k}\right) \circ f\right\Vert
<\delta $. We have that $X_{k}=M_{s_{1},t_{1}}\oplus ^{\infty }\cdots \oplus
^{\infty }M_{s_{\ell },t_{\ell }}$ and $Y_{k}=M_{v_{1},w_{1}}\oplus ^{\infty
}\cdots \oplus ^{\infty }M_{v_{m},w_{m}}$ for some $s_{i},t_{i},v_{i},w_{i}%
\in \mathbb{N}$. When $\delta $ is small enough, there exist $1\leq i\leq
\ell $ and $1\leq j\leq m$, completely contractive maps $\eta
:M_{n,1}\rightarrow M_{s_{i}t_{i}}$ and $\rho :M_{1,n}\rightarrow
M_{v_{j}w_{j}}$, and unitaries $u,v\in M_{m}$ such that $%
s_{i}v_{j}=t_{i}w_{j}=n$, and $\left\Vert u^{\ast }\left( \eta \left(
a\right) \otimes \rho \left( b\right) \right) v-\mathbb{E}_{D_{n}}\left(
a\otimes b\right) \right\Vert <\varepsilon $ for every $a\in M_{n,1}$ and $%
b\in M_{1,n}$ such that $\left\Vert a\right\Vert \leq 1$ and $\left\Vert
b\right\Vert \leq 1$. This leads to a contradiction when $\varepsilon $ is
small enough.
\end{proof}

\begin{corollary}
$\mathbb{G}_{n}$ is not completely isometrically isomorphic to $M_{n}\otimes 
\mathbb{G}$.
\end{corollary}

A similar argument can be used to prove that the Gurarij Banach space can
not be written in a nontrivial way as a tensor product of Lindenstrauss
spaces. Indeed, suppose that $X$ and $Y$ are Lindenstrauss spaces of
dimension at least $2$ such that $X\otimes Y$ is isometrically isomorphic to 
$\mathbb{G}$. Then $X$ and $Y$ contain an isometric copy of $\ell
_{2}^{\infty }$. One then reason as in the proof of Proposition \ref%
{Proposition:indecomposable} where $X_{1}=Y_{1}=\ell _{2}^{\infty }$, $%
f:\ell _{4}^{\infty }\rightarrow X_{1}\otimes Y_{1}$ is the canonical
isometric isomorphism, and $\phi :\ell _{4}^{\infty }\rightarrow \ell
_{5}^{\infty }$ is defined by $\left( z_{1},z_{2},z_{3},z_{4}\right) \mapsto
\left( z_{1},z_{2},z_{3},z_{4},\left( z_{1}+z_{4}\right) /2\right) $. This
generalizes a result of Oikhberg, who showed that the Gurarij space $\mathbb{%
G}$ is not isometrically isomorphic to $\mathbb{G}\otimes \mathbb{G}$ \cite%
{oikhberg_personal_2015}.

One can also prove with similar methods that $\mathbb{NG}$ is not completely
isometrically isomorphic to $V\otimes X$ for any $\mathcal{OL}_{\infty ,1+}$%
-space $X$ and finite-dimensional TRO $V$ of dimension at least $2$. We do
not know whether $\mathbb{NG}\otimes \mathbb{NG}$ is completely
isometrically isomorphic to $\mathbb{NG}$.

\subsection{The model theory of $\mathbb{NG}$\label{Subsection:model}}

Here we want to observe that the noncommutative Gurarij space, from the
model-theoretic perspective, plays a canonical role within the class of
operator spaces, in close similarity with the commutative case. A good
introduction to the logic for metric structures can be found in \cite%
{ben_yaacov_model_2008}. We will regard operator spaces as structures in the
language considered in \cite[Appendix B]{goldbring_kirchbergs_2015}, where
it is observed that---in view of Ruan's characterization---operator spaces
form an axiomatizable class. The arguments from \cite%
{goldbring_model-theoretic_2015}, where Choi's theorem \cite[Theorem 2]%
{choi_completely_1975} is replaced by Proposition \ref%
{Proposition:perturbation} and Remark \ref{Remark:limit}, can be easily
adapted to prove the following result.

\begin{theorem}
\label{Theorem:model}Let $\mathbb{NG}$ be the noncommutative Gurarij space,
regarded as a structure in the language for operator spaces.

\begin{enumerate}
\item $\mathbb{NG}$ is the unique separable $1$-exact existentially closed
operator space.

\item $\mathbb{NG}$ is the unique separable nuclear model of its first-order
theory,

\item Every embedding of $\mathbb{NG}$ into its ultrapower $\mathbb{NG}%
^{\omega }$ admits a lift of surjective complete isometries of $\mathbb{NG}$
and, in particular, it is elementary;

\item $\mathbb{NG}$ is the prime model if its first-order theory;

\item $\mathbb{NG}$ does not have quantifier elimination, hence the theory
of operator spaces does not have a model companion;

\item A $1$-exact operator space is nuclear if and only if it is positively
existentially closed.
\end{enumerate}
\end{theorem}

The analog of Theorem \ref{Theorem:model} for the spaces $\mathbb{G}_{n}$
for $n\in \mathbb{N}$ has been proved in \cite%
{goldbring_model-theoretic_2015}. It is worth remarking that the Cuntz
algebra $\mathcal{O}_{2}$ satisfies similar properties within the class of
unital C*-algebras: it is the prime model of its first-order theory \cite[%
Proposition 5.1]{goldbring_model-theoretic_2015}, does not have quantifier
elimination \cite[Theorem 2.7]{eagle_quantifier_2015}, and it is the only
possible existentially closed nuclear unital separable C*-algebra \cite[%
Proposition 2.18]{goldbring_kirchbergs_2015}. The assertion that $\mathcal{O}%
_{2}$ is indeed existentially closed is equivalent to Kirchberg's embedding
problem, asking whether every unital C*-algebra embeds into an ultrapower of 
$\mathcal{O}_{2}$ \cite[Theorem 3.3]{goldbring_model-theoretic_2015}.

The remarks above allow one to conclude that, also from the model-theoretic
perspective, $\mathbb{NG}$ can be regarded as the operator-space analog of
the Cuntz algebra $\mathcal{O}_{2}$.

\bibliographystyle{amsplain}
\bibliography{OS-Fr-bib}

\begin{bibdiv}
\begin{biblist}

\bib{blackadar_operator_2006}{book}{
      author={Blackadar, Bruce},
       title={Operator algebras},
      series={Encyclopaedia of Mathematical Sciences},
   publisher={Springer-Verlag},
     address={Berlin},
        date={2006},
      volume={122},
        ISBN={978-3-540-28486-4; 3-540-28486-9},
         url={http://www.ams.org/mathscinet-getitem?mr=2188261},
}

\bib{blecher_operator_2004}{book}{
      author={Blecher, David~P.},
      author={Le~Merdy, Christian},
       title={Operator algebras and their modules{\textemdash}an operator space
  approach},
      series={London Mathematical Society Monographs. New Series},
   publisher={Oxford University Press},
     address={Oxford},
        date={2004},
      volume={30},
        ISBN={0-19-852659-8},
         url={http://www.ams.org/mathscinet-getitem?mr=2111973},
}

\bib{ben_yaacov_fraisse_2012}{article}{
      author={Ben~Yaacov, Ita{\"i}},
       title={Fra\"iss\'e limits of metric structures},
     journal={Journal of Symbolic Logic},
        note={to appear.},
}

\bib{ben_yaacov_model_2008}{incollection}{
      author={Ben~Yaacov, Ita{\"i}},
      author={Berenstein, Alexander},
      author={Henson, C.~Ward},
      author={Usvyatsov, Alexander},
       title={Model theory for metric structures},
        date={2008},
   booktitle={Model theory with applications to algebra and analysis. vol. 2},
      series={London Mathematical Society Lecture Note Series},
      volume={350},
   publisher={Cambridge University Press},
       pages={315\ndash 427},
         url={http://www.ams.org/mathscinet-getitem?mr=2436146},
}

\bib{eagle_Fraisse_2014}{article}{
      author={Eagle, Christopher~J.},
      author={Farah, Ilijas},
      author={Hart, Bradd},
      author={Kadets, Boris},
      author={Kalashnyk, Vlad},
      author={Lupini, Martino},
       title={Fra\"iss\'e limits of {C}*-algebras},
        note={preprint.},
}

\bib{elliott_isomorphism_2013}{article}{
      author={Elliott, George~A.},
      author={Farah, Ilijas},
      author={Paulsen, Vern},
      author={Rosendal, Christian},
      author={Toms, Andrew~S.},
      author={T{\"o}rnquist, Asger},
       title={The isomorphism relation for separable {C}*-algebras},
        date={2013},
     journal={Mathematical Research Letters},
      volume={20},
      number={6},
       pages={1071\ndash 1080},
         url={http://arxiv.org/abs/1301.7108},
}

\bib{effros_injectivity_2001}{article}{
      author={Effros, Edward~G.},
      author={Ozawa, Narutaka},
      author={Ruan, Zhong-Jin},
       title={On injectivity and nuclearity for operator spaces},
        date={2001},
        ISSN={0012-7094},
     journal={Duke Mathematical Journal},
      volume={110},
      number={3},
       pages={489\ndash 521},
         url={http://www.ams.org/mathscinet-getitem?mr=1869114},
}

\bib{effros_spaces_1998}{incollection}{
      author={Effros, Edward~G.},
      author={Ruan, Zhong-Jin},
       title={{$\mathcal{OL}_p$} spaces},
        date={1998},
   booktitle={Operator algebras and operator theory (shanghai, 1997)},
      series={Contemp. Math.},
      volume={228},
   publisher={American Mathematica Society, Providence, {RI}},
       pages={51\ndash 77},
         url={http://www.ams.org/mathscinet-getitem?mr=1667654},
}

\bib{effros_operator_2000}{book}{
      author={Effros, Edward~G.},
      author={Ruan, Zhong-Jin},
       title={Operator spaces},
      series={London Mathematical Society Monographs. New Series},
   publisher={Oxford University Press},
        date={2000},
      volume={23},
        ISBN={0-19-853482-5},
         url={http://www.ams.org/mathscinet-getitem?mr=1793753},
}

\bib{fraisse_lextension_1954}{article}{
      author={Fra{\"i}ss{\'e}, Roland},
       title={Sur l'extension aux relations de quelques propri{\'e}t{\'e}s des
  ordres},
        date={1954},
        ISSN={0012-9593},
     journal={Annales Scientifiques de l'{\'E}cole Normale Sup{\'e}rieure.
  Troisi{\`e}me S{\'e}rie},
      volume={71},
       pages={363\ndash 388},
         url={http://www.ams.org/mathscinet-getitem?mr=0069239},
}

\bib{farah_model_2010}{article}{
      author={Farah, Ilijas},
      author={Hart, Bradd},
      author={Sherman, David},
       title={Model theory of operator algebras {II}: {M}odel theory},
     journal={Israel Journal of Mathematics},
        note={to appear.},
}

\bib{gurarij_spaces_1966}{article}{
      author={Gurari{\u\i}, Vladimir~I.},
       title={Spaces of universal placement, isotropic spaces and a problem of
  {M}azur on rotations of {B}anach spaces},
        date={1966},
     journal={Akademija Nauk SSSR. Sibirskoe Otdelenie. Sibirski\u\i\
  Matemati\v ceski\u\i\ \v Zurnal},
      volume={7},
       pages={1002\ndash 1013},
}

\bib{goldbring_kirchbergs_2014}{article}{
      author={Goldbring, Isaac},
      author={Sinclair, Thomas},
       title={On {K}irchberg's embedding problem},
     journal={Journal of Functional Analysis},
        note={to appear.},
}

\bib{garbulinska-wegrzyn_universal_2013}{article}{
      author={Garbuli\'{n}ska-W\c{e}grzyn, Joanna},
      author={Kubi{\'s}, Wies{\l }aw},
       title={A universal operator on the {G}urarii space},
     journal={Journal of Operator Theory},
        note={to appear.},
}

\bib{junge_structures_2003}{article}{
      author={Junge, Marius},
      author={Ozawa, Narutaka},
      author={Ruan, Zhong-Jin},
       title={On {$\mathcal{OL}_{\infty}$} structures of nuclear
  {C}*-algebras},
        date={2003},
        ISSN={0025-5831, 1432-1807},
     journal={Mathematische Annalen},
      volume={325},
      number={3},
       pages={449\ndash 483},
         url={http://link.springer.com/article/10.1007/s00208-002-0384-7},
}

\bib{junge_bilinear_1995}{article}{
      author={Junge, Marius},
      author={Pisier, Gilles},
       title={Bilinear forms on exact operator spaces and {$B(H)\otimes
  B(H)$}},
        date={1995},
     journal={Geometric and Functional Analysis},
      volume={5},
      number={2},
       pages={329\ndash 363},
}

\bib{kubis_fraisse_2014}{article}{
      author={Kubi{\'s}, Wies{\l }aw},
       title={Fra{\"i}ss{\'e} sequences: category-theoretic approach to
  universal homogeneous structures},
        date={2014},
        ISSN={0168-0072},
     journal={Annals of Pure and Applied Logic},
      volume={165},
      number={11},
       pages={1755\ndash 1811},
  url={http://www.sciencedirect.com/science/article/pii/S0168007214000773},
}

\bib{kubis_proof_2013}{article}{
      author={Kubi{\'s}, Wies{\l}aw},
      author={Solecki, S{\l}awomir},
       title={A proof of uniqueness of the {G}urari\u\i\ space},
        date={2013},
     journal={Israel Journal of Mathematics},
      volume={195},
      number={1},
       pages={449\ndash 456},
}

\bib{lehner_mn-espaces_1997}{thesis}{
      author={Lehner, Franz},
       title={{$M_n$}-espaces, sommes d'unitaires et analyse harmonique sur le
  groupe libre},
        type={Ph.D. Thesis},
        date={1997},
}

\bib{lusky_gurarij_1976}{article}{
      author={Lusky, Wolfgang},
       title={The {G}urarij spaces are unique},
        date={1976},
        ISSN={0003-889X},
     journal={Archiv der Mathematik},
      volume={27},
      number={6},
       pages={627\ndash 635},
         url={http://www.ams.org/mathscinet-getitem?mr=0433198},
}

\bib{melleray_extremely_2014}{article}{
      author={Melleray, Julien},
      author={Tsankov, Todor},
       title={Extremely amenable groups via continuous logic},
        date={2014},
     journal={{arXiv}:1404.4590},
         url={http://arxiv.org/abs/1404.4590},
}

\bib{oikhberg_non-commutative_2006}{article}{
      author={Oikhberg, Timur},
       title={The non-commutative {G}urarii space},
        date={2006},
        ISSN={0003-889X, 1420-8938},
     journal={Archiv der Mathematik},
      volume={86},
      number={4},
       pages={356\ndash 364},
         url={http://link.springer.com/article/10.1007/s00013-005-1631-4},
}

\bib{oikhberg_operator_2004}{article}{
      author={Oikhberg, Timur},
      author={Ricard, {\'E}ric},
       title={Operator spaces with few completely bounded maps},
        date={2004},
        ISSN={0025-5831, 1432-1807},
     journal={Mathematische Annalen},
      volume={328},
      number={1-2},
       pages={229\ndash 259},
  url={http://link.springer.com.ezproxy.library.yorku.ca/article/10.1007/s00208-003-0481-2},
}

\bib{paulsen_completely_2002}{book}{
      author={Paulsen, Vern},
       title={Completely bounded maps and operator algebras},
      series={Cambridge Studies in Advanced Mathematics},
   publisher={Cambridge University Press},
     address={Cambridge},
        date={2002},
      volume={78},
        ISBN={0-521-81669-6},
         url={http://www.ams.org/mathscinet-getitem?mr=1976867},
}

\bib{pisier_exact_1995}{article}{
      author={Pisier, Gilles},
       title={Exact operator spaces},
        date={1995},
        ISSN={0303-1179},
     journal={Ast{\'e}risque},
      number={232},
       pages={159\ndash 186},
         url={http://www.ams.org/mathscinet-getitem?mr=1372532},
        note={Recent advances in operator algebras (Orl{\'e}ans, 1992)},
}

\bib{pisier_introduction_2003}{book}{
      author={Pisier, Gilles},
       title={Introduction to operator space theory},
      series={London Mathematical Society Lecture Note Series},
   publisher={Cambridge University Press},
     address={Cambridge},
        date={2003},
      volume={294},
        ISBN={0-521-81165-1},
         url={http://www.ams.org/mathscinet-getitem?mr=2006539},
}

\bib{ruan_subspaces_1988}{article}{
      author={Ruan, Zhong-Jin},
       title={Subspaces of {C}*-algebras},
        date={1988},
        ISSN={0022-1236},
     journal={Journal of Functional Analysis},
      volume={76},
      number={1},
       pages={217\ndash 230},
  url={http://www.sciencedirect.com/science/article/pii/0022123688900572},
}

\bib{Schoretsanitis_fraisse_2007}{thesis}{
      author={Schoretsanitis, Konstantinos},
       title={Fra\"iss\'e theory for metric structures},
        type={Ph.D. Thesis},
        date={2007},
}

\bib{tomiyama_transpose_1983}{article}{
      author={Tomiyama, Jun},
       title={On the transpose map of matrix algebras},
        date={1983},
        ISSN={0002-9939, 1088-6826},
     journal={Proceedings of the American Mathematical Society},
      volume={88},
      number={4},
       pages={635\ndash 638},
         url={http://www.ams.org/proc/1983-088-04/S0002-9939-1983-0702290-4/},
}

\end{biblist}
\end{bibdiv}


\providecommand{\bysame}{\leavevmode\hbox to3em{\hrulefill}\thinspace}
\providecommand{\MR}{\relax\ifhmode\unskip\space\fi MR }
\providecommand{\MRhref}[2]{%
  \href{http://www.ams.org/mathscinet-getitem?mr=#1}{#2}
}
\providecommand{\href}[2]{#2}
\begin{thebibliography}{10}

\bibitem{alfsen_structure_1972-1}
Erik~M. Alfsen and Edward~G. Effros, \emph{Structure in real {Banach} spaces.
  {I}}, Annals of Mathematics. Second Series \textbf{96} (1972), 98--128.

\bibitem{alfsen_structure_1972-2}
\bysame, \emph{Structure in real {Banach} spaces. {II}}, Annals of Mathematics.
  Second Series \textbf{96} (1972), 129--173.

\bibitem{arveson_subalgebras_1969}
William~B. Arveson, \emph{Subalgebras of {C}*-algebras}, Acta Mathematica
  \textbf{123} (1969), no.~1, 141--224.

\bibitem{ben_yaacov_fraisse_2012}
I.~Ben~Yaacov, \emph{Fraïssé limits of metric structures}, J. Symbolic Logic
  \textbf{80} (2015), no.~1, 100--115.

\bibitem{ben_yaacov_linear_2014}
Ita{\"i} Ben~Yaacov, \emph{The linear isometry group of the {G}urarij space is
  universal}, Proceedings of the American Mathematical Society \textbf{142}
  (2014), no.~7, 2459--2467.

\bibitem{ben_yaacov_fraisse_2015}
Ita{\"{i}} Ben~Yaacov, \emph{Fra{\"{i}}ss{\'{e}} limits of metric structures},
  Journal of Symbolic Logic \textbf{80} (2015), no.~1, 100--115.

\bibitem{ben_yaacov_model_2008}
Ita{\"i} Ben~Yaacov, Alexander Berenstein, C.~Ward Henson, and Alexander
  Usvyatsov, \emph{Model theory for metric structures}, Model theory with
  applications to algebra and analysis. Vol. 2, London Mathematical Society
  Lecture Note Series, vol. 350, Cambridge University Press, 2008,
  pp.~315--427.

\bibitem{blackadar_operator_2006}
Bruce Blackadar, \emph{Operator algebras}, Encyclopaedia of Mathematical
  Sciences, vol. 122, Springer-Verlag, Berlin, 2006.

\bibitem{blackadar_generalized_1997}
Bruce Blackadar and Eberhard Kirchberg, \emph{Generalized inductive limits of
  finite-dimensional {C}*-algebras}, Mathematische Annalen \textbf{307} (1997),
  no.~3, 343--380.

\bibitem{blecher_shilov_2001}
David~P. Blecher, \emph{The {Shilov} {Boundary} of an {Operator} {Space} and
  the {Characterization} {Theorems}}, Journal of Functional Analysis
  \textbf{182} (2001), no.~2, 280--343.

\bibitem{blecher_multipliers_2004}
\bysame, \emph{Multipliers, {C}*-modules, and algebraic structure in spaces of
  {Hilbert} space operators}, Operator algebras, quantization, and
  noncommutative geometry, Contemp. {Math}., vol. 365, Amer. Math. Soc.,
  Providence, RI, 2004, pp.~85--128.

\bibitem{blecher_one-sided_2004-1}
\bysame, \emph{One-sided ideals and approximate identities in operator
  algebras}, Journal of the Australian Mathematical Society \textbf{76} (2004),
  no.~3, 425--447.

\bibitem{blecher_operator_2004}
David~P. Blecher and Christian Le~Merdy, \emph{Operator algebras and their
  modules{\textemdash}an operator space approach}, London Mathematical Society
  Monographs. New Series, vol.~30, Oxford University Press, Oxford, 2004.

\bibitem{blecher_multipliers_2001}
David~P. Blecher and Vern~I. Paulsen, \emph{Multipliers of operator spaces, and
  the injective envelope}, Pacific Journal of Mathematics \textbf{200} (2001),
  no.~1, 1--17.

\bibitem{blecher_one-sided_2004-2}
David~P. Blecher, Roger~R. Smith, and Vrej Zarikian, \emph{One-sided
  projections on {C}*-algebras}, Journal of Operator Theory \textbf{51} (2004),
  no.~1, 201--219.

\bibitem{blecher_multiplier_2004}
David~P. Blecher and Vrej Zarikian, \emph{Multiplier operator algebras and
  applications}, Proceedings of the National Academy of Sciences of the United
  States of America \textbf{101} (2004), no.~3, 727--731.

\bibitem{blecher_calculus_2006}
\bysame, \emph{The calculus of one-sided {$M$}-ideals and multipliers in
  operator spaces}, Memoirs of the American Mathematical Society \textbf{179}
  (2006), no.~842, viii+85.

\bibitem{bohle_k-theory_2011}
Dennis Bohle, \emph{K-theory for {ternary} {structures}}, Ph.D. thesis,
  Westf{\"{a}}lische Wilhelms-Universit{\"{a}}t M{\"{u}}nster, 2011.

\bibitem{bohle_universal_2014}
Dennis Bohle and Wend Werner, \emph{The universal enveloping ternary ring of
  operators of a {JB}*-triple system}, Proceedings of the Edinburgh
  Mathematical Society. Series II \textbf{57} (2014), no.~2, 347--366.

\bibitem{bohle_k-theoretic_2015}
\bysame, \emph{A {K}-theoretic approach to the classification of symmetric
  spaces}, Journal of Pure and Applied Algebra \textbf{219} (2015), no.~10,
  4295--4321.

\bibitem{brown_C*-algebras_2008}
Nathanial~P. Brown and Narutaka Ozawa, \emph{{C}*-algebras and
  finite-dimensional approximations}, Graduate Studies in Mathematics, vol.~88,
  American Mathematical Society, Providence, {RI}, 2008.

\bibitem{choi_completely_1975}
Man~Duen Choi, \emph{Completely positive linear maps on complex matrices},
  Linear Algebra and its Applications \textbf{10} (1975), 285--290.

\bibitem{choi_completely_1976}
Man-Duen Choi and Edward~G. Effros, \emph{The {completely} {positive} {lifting}
  {problem} for {C}*-algebras}, Annals of Mathematics \textbf{104} (1976),
  no.~3, 585--609.

\bibitem{davidson_noncommutative_2016}
Kenneth Davidson, Matthew Kennedy, and Martino Lupini, \emph{Noncommutative
  {C}hoquet simplices}, in preparation, 2016.

\bibitem{eagle_fraisse_2014}
Christopher~J. Eagle, Ilijas Farah, Bradd Hart, Boris Kadets, Vladyslav
  Kalashnyk, and Martino Lupini, \emph{Fra{\"{i}}ss\'{e} limits of
  {C}*-algebras}, arXiv:1411.4066 (2014).

\bibitem{eagle_quantifier_2015}
Christopher~J. Eagle, Ilijas Farah, Eberhard Kirchberg, and Alessandro Vignati,
  \emph{Quantifier elimination in {C}*-algebras}, arXiv:1502.00573 (2015).

\bibitem{eckhardt_perturbations_2010}
Caleb Eckhardt, \emph{Perturbations of completely positive maps and strong {NF}
  algebras}, Proceedings of the London Mathematical Society \textbf{101}
  (2010), no.~3, 795--820.

\bibitem{effros_aspects_1978}
Edward~G. Effros, \emph{Aspects of noncommutative order}, C*-algebras and
  applications to physics ({Proc}. {Second} {Japan}-{USA} {Sem}., {Los}
  {Angeles}, {Calif}., 1977), Lecture {Notes} in {Mathematics}, vol. 650,
  Springer, Berlin, 1978, pp.~1--40.

\bibitem{effros_matrix_2009}
\bysame, \emph{A matrix convexity approach to some celebrated quantum
  inequalities}, Proceedings of the National Academy of Sciences of the United
  States of America \textbf{106} (2009), no.~4, 1006--1008.

\bibitem{effros_injectivity_2001}
Edward~G. Effros, Narutaka Ozawa, and Zhong-Jin Ruan, \emph{On injectivity and
  nuclearity for operator spaces}, Duke Mathematical Journal \textbf{110}
  (2001), no.~3, 489--521.

\bibitem{effros_mapping_1994}
Edward~G. Effros and Zhong-Jin Ruan, \emph{Mapping {spaces} and {liftings} for
  {operator} {spaces}}, Proceedings of the London Mathematical Society
  \textbf{s3-69} (1994), no.~1, 171--197.

\bibitem{effros_OL_1998}
\bysame, \emph{{$\mathcal{OL}_p$} spaces}, Operator algebras and operator
  theory, Contemp. {Math}., vol. 228, American Mathematical Society,
  Providence, RI, 1998, pp.~51--77.

\bibitem{effros_operator_2000}
\bysame, \emph{Operator spaces}, London Mathematical Society Monographs. New
  Series, vol.~23, Oxford University Press, 2000.

\bibitem{effros_matrix_1997}
Edward~G. Effros and Soren Winkler, \emph{Matrix convexity: operator analogues
  of the bipolar and {Hahn}-{Banach} theorems}, Journal of Functional Analysis
  \textbf{144} (1997), no.~1, 117--152.

\bibitem{elliott_isomorphism_2013}
George~A. Elliott, Ilijas Farah, Vern Paulsen, Christian Rosendal, Andrew~S.
  Toms, and Asger T{\"o}rnquist, \emph{The isomorphism relation for separable
  {C}*-algebras}, Mathematical Research Letters \textbf{20} (2013), no.~6,
  1071--1080.

\bibitem{farah_model_2010}
Ilijas Farah, Bradd Hart, and David Sherman, \emph{Model theory of operator
  algebras {II}: model theory}, Israel Journal of Mathematics \textbf{201}
  (2014), no.~1, 477--505.

\bibitem{farenick_extremal_2000}
Douglas Farenick, \emph{Extremal matrix states on operator systems}, Journal of
  the London Mathematical Society \textbf{61} (2000), no.~3, 885--892.

\bibitem{farenick_pure_2004}
\bysame, \emph{Pure matrix states on operator systems}, Linear Algebra and its
  Applications \textbf{393} (2004), 149--173.

\bibitem{fraisse_lextension_1954}
Roland Fra{\"i}ss{\'e}, \emph{Sur l'extension aux relations de quelques
  propri{\'e}t{\'e}s des ordres}, Annales Scientifiques de l'{\'E}cole Normale
  Sup{\'e}rieure. Troisi{\`e}me S{\'e}rie \textbf{71} (1954), 363--388.

\bibitem{garbulinska-wegrzyn_universal_2013}
Joanna Garbuli{\'{n}}ska-W{\c{e}}grzyn and Wies{\l}aw Kubi{\'{s}}, \emph{A
  universal operator on the {G}urari{\u\i} space}, Journal of Operator Theory
  \textbf{73} (2015), no.~1, 143--158.

\bibitem{goldbring_model-theoretic_2015}
Isaac Goldbring and Martino Lupini, \emph{Model-theoretic aspects of the
  {G}urarij operator space}, {arXiv}:1501.04332 (2015).

\bibitem{goldbring_kirchbergs_2015}
Isaac Goldbring and Thomas Sinclair, \emph{On {Kirchberg}'s embedding problem},
  Journal of Functional Analysis \textbf{269} (2015), no.~1, 155--198.

\bibitem{gurarij_spaces_1966}
Vladimir~I. Gurari{\u\i}, \emph{Spaces of universal placement, isotropic spaces
  and a problem of {M}azur on rotations of {B}anach spaces}, Akademija Nauk
  SSSR. Sibirskoe Otdelenie. Sibirski\u\i\ Matemati\v ceski\u\i\ \v Zurnal
  \textbf{7} (1966), 1002--1013.

\bibitem{hamana_triple_1999}
Masamichi Hamana, \emph{Triple envelopes and {\v{s}}ilov boundaries of operator
  spaces}, Mathematics Journal of Toyama University \textbf{22} (1999), 77--93.

\bibitem{han_approximation_2011}
Kyung~Hoon Han and Vern~I. Paulsen, \emph{An approximation theorem for nuclear
  operator systems}, Journal of Functional Analysis \textbf{261} (2011), no.~4,
  999--1009.

\bibitem{harris_generalization_1981}
Lawrence~A. Harris, \emph{A generalization of {C}*-algebras}, Proceedings of
  the London Mathematical Society. Third Series \textbf{42} (1981), no.~2,
  331--361.

\bibitem{hestenes_ternary_1962}
Magnus~R. Hestenes, \emph{A ternary algebra with applications to matrices and
  linear transformations}, Archive for Rational Mechanics and Analysis
  \textbf{11} (1962), 138--194.

\bibitem{junge_OL_2004}
Marius Junge, Niels~J. Nielsen, Zhong-Jin Ruan, and Quan~H. Xu,
  \emph{{$\mathcal{COL}_p$} spaces---the local structure of non-commutative
  {$L_{p}$} spaces}, Advances in Mathematics \textbf{187} (2004), no.~2,
  257--319.

\bibitem{junge_OL_2003}
Marius Junge, Narutaka Ozawa, and Zhong-Jin Ruan, \emph{On
  {$\mathcal{OL}_{\infty}$} structures of nuclear {C}*-algebras}, Mathematische
  Annalen \textbf{325} (2003), no.~3, 449--483.

\bibitem{junge_bilinear_1995}
Marius Junge and Gilles Pisier, \emph{Bilinear forms on exact operator spaces
  and {$B(H)\otimes B(H)$}}, Geometric and Functional Analysis \textbf{5}
  (1995), no.~2, 329--363.

\bibitem{kaur_local_2002}
Manmohan Kaur and Zhong-Jin Ruan, \emph{Local {properties} of {ternary} {rings}
  of {operators} and {their} {linking} {C}*-{algebras}}, Journal of Functional
  Analysis \textbf{195} (2002), no.~2, 262--305.

\bibitem{kerr_gromov-hausdorff_2009}
David Kerr and Hanfeng Li, \emph{On {Gromov}-{Hausdorff} convergence for
  operator metric spaces}, Journal of Operator Theory \textbf{62} (2009),
  no.~1, 83--109.

\bibitem{kirchberg_embedding_2000}
Eberhard Kirchberg and N.~Christopher Phillips, \emph{Embedding of exact
  {C}*-algebras in the {C}untz algebra {$\mathcal{O}_2$}}, Journal f{\"u}r die
  reine und angewandte Mathematik \textbf{525} (2000), 17--53.

\bibitem{kirchberg_c*-algebras_1998}
Eberhard Kirchberg and Simon Wassermann, \emph{C*-algebras generated by
  operator systems}, Journal of Functional Analysis \textbf{155} (1998), no.~2,
  324--351.

\bibitem{kubis_fraisse_2014}
Wies{\l }aw Kubi{\'s}, \emph{Fra{\"i}ss{\'e} sequences: category-theoretic
  approach to universal homogeneous structures}, Annals of Pure and Applied
  Logic \textbf{165} (2014), no.~11, 1755--1811.

\bibitem{kubis_proof_2013}
Wies{\l}aw Kubi{\'s} and S{\l}awomir Solecki, \emph{A proof of uniqueness of
  the {G}urari\u\i\ space}, Israel Journal of Mathematics \textbf{195} (2013),
  no.~1, 449--456.

\bibitem{lehner_mn-espaces_1997}
Franz Lehner, \emph{{$M_n$}-espaces, sommes d'unitaires et analyse harmonique
  sur le groupe libre}, Ph.D. thesis, Universit\'{e} de Paris 6, 1997.

\bibitem{li_order-unit_2006}
Hanfeng Li, \emph{Order-unit quantum {G}romov--{H}ausdorff distance}, Journal
  of Functional Analysis \textbf{231} (2006), no.~2, 312--360.

\bibitem{lindenstrauss_poulsen_1978}
Joram Lindenstrauss, Gunnar Olsen, and Yaki Sternfeld, \emph{The {Poulsen}
  simplex}, Annales de l'Institut Fourier \textbf{28} (1978), no.~1, 91--114.

\bibitem{lupini_operator_2015}
Martino Lupini, \emph{Operator space and operator system analogs of
  {Kirchberg}'s nuclear embedding theorem}, Journal of Mathematical Analysis
  and Applications \textbf{431} (2015), no.~1, 47--56.

\bibitem{lupini_rectangular_2016}
\bysame, \emph{Rectangular convex sets}, research notes. Available at
  \url{https://sites.google.com/site/martinolupini/rectangular-convex-sets4.pdf},
  2016.

\bibitem{lusky_gurarij_1976}
Wolfgang Lusky, \emph{The {G}urarij spaces are unique}, Archiv der Mathematik
  \textbf{27} (1976), no.~6, 627--635.

\bibitem{masumoto_jiang-su_2016}
S.~Masumoto, \emph{Jiang-{S}u algebra as a {F}ra\"{i}ss\'{e} limit},
  arXiv:1602.00124 (2016).

\bibitem{melleray_extremely_2014}
Julien Melleray and Todor Tsankov, \emph{Extremely amenable groups via
  continuous logic}, {arXiv}:1404.4590 (2014).

\bibitem{neal_operator_2003}
Matthew Neal and Bernard Russo, \emph{Operator space characterizations of
  {C}*-algebras and ternary rings}, Pacific Journal of Mathematics \textbf{209}
  (2003), no.~2, 339--364.

\bibitem{oikhberg_non-commutative_2006}
Timur Oikhberg, \emph{The non-commutative {G}urarii space}, Archiv der
  Mathematik \textbf{86} (2006), no.~4, 356--364.

\bibitem{oikhberg_personal_2015}
\bysame, personal communication, 2015.

\bibitem{oikhberg_operator_2004}
Timur Oikhberg and {\'E}ric Ricard, \emph{Operator spaces with few completely
  bounded maps}, Mathematische Annalen \textbf{328} (2004), no.~1-2, 229--259.

\bibitem{paulsen_completely_2002}
Vern Paulsen, \emph{Completely bounded maps and operator algebras}, Cambridge
  Studies in Advanced Mathematics, vol.~78, Cambridge University Press,
  Cambridge, 2002.

\bibitem{pestov_topological_1999}
Vladimir Pestov, \emph{Topological groups: where to from here?}, Topology
  Proceedings, vol.~24, 1999, pp.~421--502 (2001).

\bibitem{pisier_exact_1995}
Gilles Pisier, \emph{Exact operator spaces}, Ast{\'e}risque (1995), no.~232,
  159--186, Recent advances in operator algebras (Orl{\'e}ans, 1992).

\bibitem{pisier_introduction_2003}
\bysame, \emph{Introduction to operator space theory}, London Mathematical
  Society Lecture Note Series, vol. 294, Cambridge University Press, Cambridge,
  2003.

\bibitem{poulsen_simplex_1961}
Ebbe~T. Poulsen, \emph{A simplex with dense extreme points}, Annales de
  l'institut Fourier \textbf{11} (1961), 83--87.

\bibitem{rordam_classification_2002}
Mikael R{\o}rdam and Erling St{\o}rmer, \emph{Classification of nuclear
  {C}*-algebras. {Entropy} in operator algebras}, Encyclopaedia of
  {Mathematical} {Sciences}, vol. 126, Springer-Verlag, Berlin, 2002.

\bibitem{ruan_subspaces_1988}
Zhong-Jin Ruan, \emph{Subspaces of {C}*-algebras}, Journal of Functional
  Analysis \textbf{76} (1988), no.~1, 217--230.

\bibitem{ruan_injectivity_1989}
\bysame, \emph{Injectivity of operator spaces}, Transactions of the American
  Mathematical Society \textbf{315} (1989), no.~1, 89--104.

\bibitem{ruan_type_2004}
\bysame, \emph{Type {decomposition} and the {rectangular} {AFD} {property} for
  {W}*-{TRO}'s}, Canadian Journal of Mathematics \textbf{56} (2004), no.~4,
  843--870.

\bibitem{Schoretsanitis_fraisse_2007}
Konstantinos Schoretsanitis, \emph{Fra\"iss\'e theory for metric structures},
  Ph.D. thesis, University of Illinois at Urbana-Champaign, 2007.

\bibitem{smith_finite_2000}
Roger~R. Smith, \emph{Finite dimensional injective operator spaces},
  Proceedings of the American Mathematical Society \textbf{128} (2000), no.~11,
  3461--3462.

\bibitem{webster_krein-milman_1999}
Corran Webster and Soren Winkler, \emph{The {Krein}-{Milman} theorem in
  operator convexity}, Transactions of the American Mathematical Society
  \textbf{351} (1999), no.~1, 307--322.

\bibitem{winkler_non-commutative_1999}
Soren Winkler, \emph{The non-commutative {Legendre}-{Fenchel} transform},
  Mathematica Scandinavica \textbf{85} (1999), no.~1, 30--48.

\bibitem{wittstock_matrix_1984}
Gerd Wittstock, \emph{On matrix order and convexity}, Functional analysis:
  surveys and recent results, {III} ({Paderborn}, 1983), North-{Holland}
  {Math}. {Stud}., vol.~90, North-Holland, Amsterdam, 1984, pp.~175--188.

\bibitem{youngson_completely_1983}
Martin~A. Youngson, \emph{Completely contractive projections on {C}*-algebras},
  The Quarterly Journal of Mathematics. Oxford. Second Series \textbf{34}
  (1983), no.~136, 507--511.

\bibitem{zettl_characterization_1983}
Heinrich Zettl, \emph{A characterization of ternary rings of operators},
  Advances in Mathematics \textbf{48} (1983), no.~2, 117--143.

\end{thebibliography}

\end{document}